\documentclass[11pt,letterpaper]{article}
\usepackage{fullpage}
\usepackage[margin=1.038in]{geometry}
\usepackage{multirow}
\usepackage{url}
\usepackage[utf8]{inputenc}
\usepackage{amsmath,amssymb,amsfonts}
\usepackage{graphicx}
\usepackage{color}
\usepackage{fancyhdr}
\usepackage{tikz}
\usepackage{soul}
\usepackage{comment}
\usepackage{bm}
\usepackage{enumitem,kantlipsum}
\usepackage{wrapfig}
\usepackage[compact]{titlesec}
\titlespacing*{\section} {0pt}{1ex}{1ex}
\titlespacing*{\subsection} {0pt}{1ex}{1ex}
\titlespacing*{\subsubsection}{0pt}{1ex}{1ex}

\usepackage{amsopn}

\usepackage{amsthm}
\usepackage[square,sort,comma,numbers]{natbib}
\DeclareMathOperator*{\argmin}{\arg\!\min}

\newcommand{\real}{{\rm{I\hspace{-.75mm}R}}}

\renewcommand*{~}{\relax\ifmmode\sim\else\nobreakspace{}\fi}%GET TILDE FOR "IS DISTRIBUTED AS"

\newcommand{\Var}{\mbox{Var}}

\newcommand{\Cov}{\mbox{Cov}}

\newcommand{\Ebar}{\bar{E}}
\newcommand{\Fbar}{\bar{F}}

\newcommand{\rhohat}{\hat{\rho}}
\newcommand{\sigmahat}{\hat{\sigma}}

\newcommand{\Ftilde}{\widetilde{F}}
\newcommand{\Etilde}{\widetilde{E}}
\newcommand{\BFx}{\bm{x}}

\newcommand{\BFz}{\bm{z}}
\newcommand{\BFG}{\bm{G}}

\newcommand{\BFX}{\bm{X}}
\newcommand{\BFS}{\bm{S}}

\newcommand{\BFnu}{\bm{\nu}}

\newcommand{\sfbf}{\mathsf{bf}}

\newcommand{\sfH}{\mathsf{H}}
\newcommand{\sfM}{\mathsf{M}}

\newcommand{\mcA}{\mathcal{A}}
\newcommand{\mcB}{\mathcal{B}}

\newcommand{\mcD}{\mathcal{D}}

\newcommand{\mcF}{\mathcal{F}}

\newcommand{\mcH}{\mathcal{H}}

\newcommand{\mcK}{\mathcal{K}}
\newcommand{\mcL}{\mathcal{L}}

\newcommand{\mcO}{\mathcal{O}}

\newcommand{\mcX}{\mathcal{X}}

\newcommand{\mbE}{\mathbb{E}}

\newcommand{\mbN}{\mathbb{N}}

\newcommand{\mbP}{\mathbb{P}}

\usepackage{array, caption, floatrow, tabularx, makecell, booktabs}%
\setcellgapes{3pt}
\usepackage{algorithmic}
\usepackage{algorithm}

\newtheorem{theorem}{Theorem}

\newtheorem{lemma}{Lemma}

\newtheorem{proposition}{Proposition}

\newtheorem{corollary}[theorem]{Corollary}

\newtheorem{definition}{Definition}

\newtheorem{assumption}{Assumption}

\newtheorem{remark}{Remark}

\usepackage{bm}
\usepackage{multicol}
\usepackage{lipsum}
\usepackage{mwe}
\usepackage{graphicx}
\usepackage{subfig}
\usepackage{amssymb}
\usepackage{booktabs}
\usepackage{enumitem}
\usepackage{cleveref}

\usepackage{multirow}
\usepackage{amsfonts}
\usepackage{graphicx}
\usepackage{epstopdf}
\usepackage{algorithmic}
\usepackage{algorithm}

\usepackage{qcircuit}
\usepackage{physics}

\usepackage{authblk}

\ifpdf
  \DeclareGraphicsExtensions{.eps,.pdf,.png,.jpg}
\else
  \DeclareGraphicsExtensions{.eps}
\fi

\makeatletter
\newcommand{\ALOOP}[1]{\ALC@it\algorithmicloop\ #1%
  \begin{ALC@loop}}
\newcommand{\ENDALOOP}{\end{ALC@loop}\ALC@it\algorithmicendloop}

\newcommand{\algorithmicbreak}{\textbf{break}}
\newcommand{\BREAK}{\STATE \algorithmicbreak}
\makeatother

\title{Adaptive Sampling-Based Bi-Fidelity Stochastic Trust Region Method for Derivative-Free Stochastic Optimization}
\author[1]{Yunsoo Ha\thanks{yunsoo.ha@nrel.gov}}
\author[1]{Juliane Mueller\thanks{juliane.mueller@nrel.gov}}
\affil[1]{Computational Science Center, National Renewable Energy Laboratory, \protect\\ 15013 Denver West Parkway, Golden, 80401, Colorado, USA}
\date{}

\begin{document}
\vspace{-3 mm}

\maketitle

\begin{abstract}
%Insert your abstract here. Include keywords, PACS and mathematical subject classification numbers as needed.
Bi-fidelity stochastic optimization { has gained } increasing { attention as an efficient approach to reduce computational costs } by { leveraging } a \textit{low-fidelity} (LF) { model to optimize an } expensive \textit{high-fidelity} (HF) { objective}. In this paper, we { propose } ASTRO-BFDF, an adaptive sampling trust region method specifically designed for unconstrained bi-fidelity stochastic derivative-free optimization problems. { In } ASTRO-BFDF, the LF function serves two purposes: { (i) } to identify better iterates for the HF function { when the optimization process indicates a high correlation between them, } and { (ii) } to reduce the variance of the HF function estimates { using } bi-fidelity Monte Carlo (BFMC). { The algorithm dynamically determines sample sizes while adaptively choosing between crude Monte Carlo and BFMC to balance the trade-off between optimization and sampling errors. } We { prove }that the iterates generated by ASTRO-BFDF converge to the first-order stationary point almost surely. Additionally, we demonstrate the { effectiveness } of { the } proposed algorithm { through numerical experiments } on synthetic problems and simulation optimization problems { involving }discrete event { systems}.

\end{abstract}

\section{Introduction}
\label{sec:intro}
We consider the stochastic optimization (SO) problem 
\begin{equation}
    \min_{\BFx\in \real^d} f^h(\BFx) = \mbE_{{\Xi^h}}[F^h(\BFx,{\xi^h})], %= \int_\Xi F^h(\BFx,\xi)\,  P(\mbox{\rm d} \xi),
    \label{eq:problem}
\end{equation}
where $f^h:\real^d \rightarrow \real$ is nonconvex and bounded from below, $F^h:\real^d\times{\Xi^h} \rightarrow \real$ is a random function, and {$\xi^h:\Omega\rightarrow\Xi^h$ is a random element. In particular, we are interested in the case where $f^h(\BFx)$ can only be observed with noise through evaluations of $F^h(\BFx,\xi^h)$. Consequently, an estimator of $f^h(\BFx)$ is obtained by repeatedly evaluating $F^h$}, as shown below:
\begin{equation} \label{eq:crude-mc}
\Fbar^h(\BFx,n) = \frac{1}{n} \sum_{i=1}^n F^h(\BFx,{\xi_i^h}),
\end{equation}
{and an estimate of the variance is computed as }
$$(\sigmahat^h)^2(\BFx,n)=n^{-1}\sum_{i=1}^n \left( F^h(\BFx,{\xi_i^h})- \Fbar^h(\BFx,n)\right)^2.$$
{Furthermore, we assume that derivative information is not directly available from a Monte Carlo simulation, and each evaluation of $F^h$ is computationally expensive. Consequently, the main challenge in solving Problem~\eqref{eq:problem} lies in the need for a large number of model evaluations. One way to mitigate this burden is through bi-fidelity techniques, which leverage both a high-fidelity (HF) model and a low-fidelity (LF) model that is computationally cheaper but less accurate. For example, the LF model primarily helps identify promising solution candidates, while the HF model is used to further evaluate and refine those candidates. Hence, }we assume that there exists an additional stochastic simulation oracle capable of approximating $F^h(\BFx,{\xi^{h}})$ at a lower cost. This cost-effective oracle, termed the LF simulation, generates $F^{\ell}:\real^d\times{\Xi^{\ell}} \rightarrow \real$ with $f^{\ell}(\BFx) = \mbE_{{\Xi^{\ell}}}[F^{\ell}(\BFx,{\xi^{\ell}})],$ {where $\xi^{\ell}:\Omega\rightarrow\Xi^{\ell}$ is a random element. This framework is commonly referred to as \textit{bi-fidelity stochastic optimization} (BFSO) and has recently gained popularity due to advancements in digital twins, with applications in manufacturing~\cite{hsieh2017equipment}, production~\cite{kang2020multifidelity,zhang2022improved}, and engineering design~\cite{chaudhuri2018multifidelity,hamdia2022multilevel,pisaroni2017multi}.} %To help understand HF and LF simulation oracles, we introduce the following examples: }
%\cite{de2020bi,ng2014mf,peherstorfer2018survey,xu2014efficient}. }

{HF and LF stochastic simulation models are typically developed through hierarchical modeling, where the HF model captures all relevant system details, while the LF model provides a simplified approximation. For example, in manufacturing systems, the HF model simulates a complete process including all machines, whereas the LF model is obtained by excluding machines that are not critical to the performance metrics (see Figure 8 in \cite{zhang2022improved}). In airfoil design, fidelity is controlled by mesh resolution: the HF model uses a fine grid with many nodes around the airfoil, whereas LF models are generated by coarsening the mesh, thereby reducing computational cost (see Figure 17 in \cite{pisaroni2017continuation}). In discrete event simulation (DES), fidelity levels can be adjusted by changing simulation run length; the HF model uses long simulation runs to obtain accurate estimates, whereas LF models use shorter runs for faster evaluation \cite{chen2017stochastic}.}

%\begin{itemize}
%    \item[$\cdot$] {\textit{Hierarchical Modeling}: The HF model captures all process details, while the LF model represents a simplified version. For example, the HF model simulates the entire manufacturing process, whereas the LF model simplifies it by excluding unnecessary machines (see Figure 8 in \cite{zhang2022improved}). As another example, in airfoil design, the hierarchy is made up of 5 non nested grid levels generated by doubling the number of nodes around the airfoil (see Figure 17 in \cite{pisaroni2017continuation}). %\cite{li2017hierarchical,kang2020multifidelity,hamdia2022multilevel} 
%}
%    \item[$\cdot$] {Example 2}
%\end{itemize}

{Multi-fidelity approaches have primarily been used in Bayesian optimization (BO) with Co-kriging serving as the surrogate model due to its flexibility~\cite{do2023multi}. Most of the existing literature focuses on deterministic engineering design problems~\cite{charayron2023bi,meliani2019multi,shu2021multi,tran2019sbf,xu2023bi}, since BO tends to perform poorly in the presence of high stochastic noise \cite{daulton2021parallel,letham2019noisebo,picheny2014noisy}. Specifically, the sample size $n$ must be sufficiently large to ensure accurate function estimates, which becomes especially challenging when dealing with heteroskedastic variances~\cite{diwale2022bayesian}. Moreover, the LF function should remain closely correlated with the HF function throughout the entire search space, as BO is a global optimization method~\cite{foumani2023effects}. Otherwise, the LF function can negatively impact the surrogate model, reducing its accuracy~\cite{andres2024characterising,toal2015some}. For instance, the LF function may provide no useful information about the HF function, which disrupts the optimization process, as illustrated in Figure \ref{fig:bf-example}. } 
Hence, the specifics of how and when to utilize the LF simulation oracle have {remained } elusive, prompting us to pose two overarching questions:

\begin{enumerate}
    \item[Q1.] When is it appropriate to utilize the LF simulation oracle, and when should it not be used during optimization?
    \item[Q2.] What sample sizes for HF and LF simulation oracles are necessary to attain sufficiently accurate function estimates for optimization?
\end{enumerate}

%The question Q1 can be largely divided as two parts: 1) 
In this paper, we propose a \textit{sample-efficient} {local search } solver for BFSO, aiming to address questions Q1 and Q2. We begin by introducing relevant existing sampling methods that have been used for sample-efficient uncertainty quantification, regardless of their purpose for optimization.

\begin{figure} [htp]
\centering
\includegraphics[width=0.6\columnwidth]{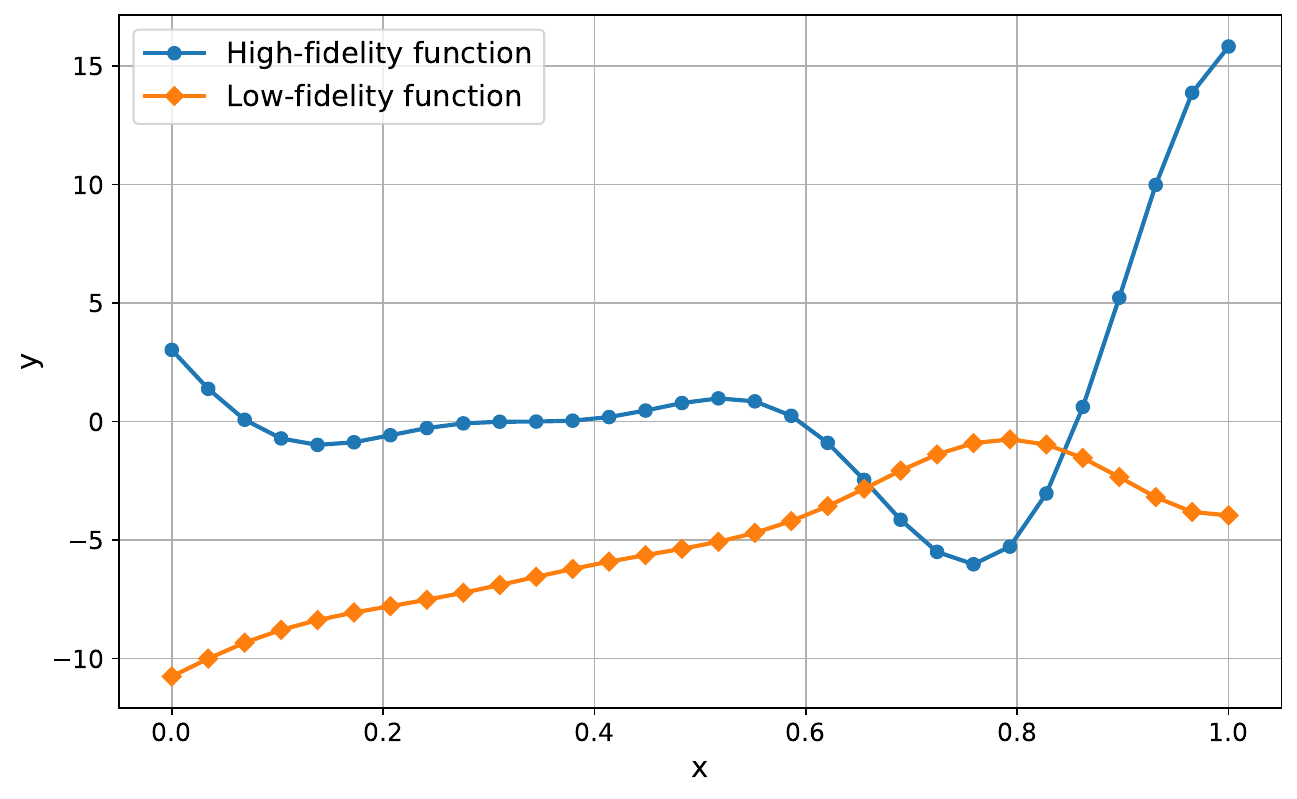}
\caption{An illustration of bi-fidelity functions{, where the LF function serves as a harmful source.} %The black and red curves represent the true objective functions of the HF and LF versions, respectively. Meanwhile, the blue and orange curves illustrate a \emph{single sample path} of the stochastic objective functions for the HF and LF versions, respectively.
} \label{fig:bf-example}
\end{figure}

%It is reasonable to expect that employing a LF simulation oracle can expedite convergence by providing more information at a lower cost. For example, the iterative algorithm can yield a near-optimal solution of the HF function, even though the LF function is solely utilized, as illustrated in Figure \ref{fig:bf-example}. 

\subsection{Adaptive Sampling}
In derivative-free stochastic optimization, {a suitable sample size is essential for reducing computational cost while ensuring the convergence of the solution sequence $\{\BFX_k\}$ to a stationary point. } Typically, {when the random samples 
$\{\xi_i^h\}_{i=1}^n$ vary across evaluations of $x$, } a fixed sample size may {fail to } ensure { almost sure convergence } because the stochastic error may exceed the gap in function values between { successive iterates}. 
Therefore, an Adaptive Sampling (AS) strategy has been { incorporated into } iterative algorithms~\cite{bollapragada2018adaptive,bollapragada2023adaptive,Sara2018ASTRO}. 
The AS strategy dynamically { determines } the sample size by { balancing } an estimation error at each point { against an } optimality measure {such as gradient approximations}. With the estimation error continuously updated through replicated function evaluations, the AS strategy { selects a dynamically adjusted sample size } based on observations generated at the design point of interest. 
The AS strategy has been proposed { using } a crude Monte Carlo (CMC) estimation (See~\eqref{eq:crude-mc}) for scenarios where only one simulation oracle is available~\cite{bollapragada2018adaptive,bollapragada2024asdf,Sara2018ASTRO}. However, when multiple simulation oracles of different fidelities are available, { a variance reduction technique known as multi-fidelity Monte Carlo (MFMC)~\cite{karen2016mfmc} can be employed, as introduced in the next section.}

\subsection{Bi-fidelity Monte Carlo}
In the domain of uncertainty quantification, a function estimate for a single design point is typically derived through a CMC estimation. However, due to the slow convergence rate of CMC, where {the estimation error decays at a rate proportional to $1/\sqrt{n}$, } it may be impractical to obtain sufficiently accurate function estimates within a reasonable timeframe. {To improve computational efficiency, bi-fidelity Monte Carlo (BFMC) has been proposed as a variance reduction methods~\cite{karen2016mfmc}, with the unbiased estimator defined as:}
\begin{equation} \label{eq:bfmc}
    \Fbar^{\sfbf}(\BFx,n,v,c) = \frac{1}{n} \sum_{i=1}^{n} F^h(\BFx,{\xi^h_i}) - c \left(\frac{1}{n} \sum^n_{{j}=1} F^{\ell}(\BFx,{\xi^{\ell}_j}) - \frac{1}{v} \sum^v_{{j}=1} F^{\ell}(\BFx,{\xi^{\ell}_{j}})\right),
\end{equation}
{where $v$ is the sample size for the LF oracle and $c\in\real$. }
The variance of $\Fbar^{\sfbf}(\BFx,n,v,c)$ { is }
\begin{equation} \label{eq:variance-bfmc}
\begin{split}
    & c^2 (\Var(\Fbar^l(\BFx,n)) + \Var(\Fbar^l(\BFx,v))) - 2c \Cov(\Fbar^h(\BFx,n), \Fbar^l(\BFx,n)) \\
    & + 2c \Cov(\Fbar^h(\BFx,n), \Fbar^l(\BFx,v)) - 2c^2 \Cov(\Fbar^l(\BFx,n), \Fbar^l(\BFx,v)) + \Var(\Fbar^h(\BFx,n)).
\end{split}
\end{equation}
Therefore, variance reduction becomes feasible { when strong correlation exists between HF and LF estimators and appropriate values of $n,v,$ and $c$ are chosen. }
{ The effectiveness of such variance reduction has been empirically demonstrated in \cite{peherstorfer2018mf,yao2022mfexam}. }

In our proposed algorithm, we have developed an innovative AS strategy, referred { to } as bi-fidelity Adaptive Sampling (BFAS), that leverages both LF and HF oracles. Our approach dynamically employs BFMC and CMC, guided by { covariance and variance estimates } for the functions. {Although BFAS can, in principle, be applied to a broad range of iterative solvers, } we focus exclusively on stochastic trust region (TR) algorithms { for solving Problem~ \eqref{eq:problem}}.

\subsection{Stochastic Trust Region Algorithms for Derivative-free Stochastic Optimization}
\label{sec:STRO}

The popularity of stochastic TR algorithms has recently surged for addressing~\eqref{eq:problem} due to their robustness, which comes from their ability to self-tune and naturally utilize { approximate } curvature information in determining step lengths. Stochastic TR algorithms~\cite{blanchet2019convergence,Chang2013STRONG,chen2018storm,curtis2023worst} typically { involve } the following four steps in { each }iteration $k$:
\begin{enumerate}
    \item[(a)] (model construction) a local model is constructed to approximate the objective function $f$ by utilizing specific design points and their function estimates within a designated area of confidence, i.e., { the } TR, typically defined as an {$\ell_2$ ball } with radius $\Delta_k$ centered { at } the current iterate $\BFX_k$;
    \item[(b)] (subproblem minimization) a candidate point $\BFX_k^s$ is obtained by approximately minimizing the local model within the TR;
    \item[(c)] (candidate evaluation) the objective function at $\BFX_k^s$ is estimated by querying the oracle, and depending on this evaluation, $\BFX_k^s$ is either accepted or rejected; and
    \item[(d)] (TR management) if $\BFX_k^s$ is accepted, it becomes the { next } iterate $\BFX_{k+1}$, and the TR radius $\Delta_k$ is either enlarged or remains unchanged; conversely, if $\BFX_k^s$ is rejected, $\BFX_{k}$ { remains } as $\BFX_{k+1}$, and $\Delta_k$ { is reduced to enable the construction of a more accurate local model.}
\end{enumerate}

As described in Section \ref{sec:bi-fidelity-TRO}, our proposed algorithm {executes } the aforementioned four steps multiple times within a single iteration to address Q1, utilizing both HF and LF simulation oracles. Specifically, when { a strong } correlation between the LF and HF function is {estimated within the TR based on the } optimization history, the local model { are first constructed } for the LF function. {If the LF-based local model fails to produce a candidate with a lower HF objective estimate, a local model for the HF function is then constructed.}

\subsection{Summary of Results and Insight.}
{In this paper, } we propose a novel stochastic TR method, { which employs two separate TRs to handle HF and LF functions, and incorporates an adaptive sampling scheme specifically designed for BFSO, } named ASTRO-BFDF. {Our key contributions are as follows:}
\begin{enumerate}
    \item[(a)] Addressing Q1, {we introduce an adaptive correlation constant ($\alpha_k$ in Algorithm \ref{alg:TRO-MFDF}), capturing local correlation between the HF and LF functions.} 
    \item[(b)] {Addressing Q2, } we suggest a new adaptive sampling algorithm, named BFAS, {which leverages both HF and LF oracles and dynamically controls $n, v,$ and $c$ through updates from each new sample.}
    \item[(c)] We prove the almost sure convergence, i.e., $\lim_{k\rightarrow\infty}\| \nabla f^h(\BFX_k) \| = 0 $ w.p.1, of ASTRO-BFDF. 
    \item[(d)] {Extensive numerical experiments demonstrate that ASTRO-BFDF consistently outperforms a range of bi-fidelity and single-fidelity solvers across varying correlation levels between the HF and LF functions, as well as different variances of HF and LF stochastic noise.}
\end{enumerate}

\section{Preliminaries}
In this section, we provide key definitions, standing assumptions, and some useful results that will be invoked in the convergence analysis of the proposed algorithm. 

\subsection{Notation} 

We represent vectors using bold font; for instance, $\BFx = (x_1, x_2, \cdots, x_d) \in \real^d$ refers to a vector in $d$-dimensional space. Sets are denoted with calligraphic fonts, while matrices are shown in sans serif fonts. The default norm, $\|\cdot\|$, is the { $\ell_2$ } norm.  The closed ball of radius $\Delta>0$ centered at $\BFx^{0}$ is $\mcB(\BFx^{0};\Delta)=\{\BFx\in\real^d:\|\BFx-\BFx^{0}\|\leq\Delta\}$. For a sequence of sets ${\mcA_n}$, ${\mcA_n \ \text{i.o.}}$ denotes $\limsup_{n\rightarrow\infty} \mcA_n$, {where ``i.o." stands for ``infinitely often." }
We write $f(\BFx)=\mcO(g(\BFx))$ if there are positive constants $\varepsilon$ and $m$ such that $|f(\BFx)|\le mg(\BFx)$ for all $\BFx$ with $0 < \|\BFx\| < \varepsilon$. Capital letters denote random scalars and vectors. For a sequence of random vectors $\{\BFX_k,k\ge1\}$, $\BFX_k \xrightarrow{w.p.1}\BFX$ denotes almost sure convergence. ``iid" means independent and identically distributed, and ``w.p.1" means with probability 1.
The superscripts $h$ and $l$ indicate that the terminology is related to high-fidelity and low-fidelity simulations, respectively.
The terms $\sigmahat^h(\BFx, n)$ and $\sigmahat^{\ell}(\BFx, n)$ are the standard deviation estimates of HF and LF functions at $\BFx$ with sample size $n$, while $\sigmahat^{h,\ell}(\BFx, n)$ is the covariance estimate between them.

\subsection{Key Definitions}
{We introduce the definition of a stochastic interpolation model, used as the local model in ASTRO-BFDF.
}
\begin{definition}[stochastic interpolation models: {Definition 2.3 in \cite{Sara2018ASTRO}}]
    Given $\BFX_k=\BFX_k^0\in\real^d$ and $\Delta^q_k>0$, 
    let $\Phi(\BFx)=[\phi_0(\BFx),\phi_1(\BFx), \dots, \phi_p(\BFx)]$ be a polynomial basis on $\real^d$. With $p=d(d+3)/2$, $q \in \{h,\ell\}$ and the design set $\mcX_k=\{{\BFX_k^{i}}\}_{i=0}^{p}\subset \mcB(\BFX_k;\Delta^q_k)$, we find $\BFnu^q_k = [\nu^q_{k,0},\nu^q_{k,1}, \dots, \nu^q_{k,p}]^\intercal$ such that 
    \begin{equation}\label{eq:syslineq}
        \sfM(\Phi, \mcX_k) \BFnu^q_k = \left[ \Fbar^q_k(\BFX_k^0,N(\BFX_k^0))), \Fbar^q_k(\BFX_k^1,N(\BFX_k^1)), \dots, \Fbar^q_k(\BFX_k^p,N(\BFX_k^p)) \right]^\intercal,
    \end{equation}
    where \\ 
    \begin{equation*}
        \sfM(\Phi, \mcX_k) = 
        \begin{bmatrix}
        \phi_1(\BFX_k^{0}) & \phi_2(\BFX_k^{0}) & \cdots & \phi_p(\BFX_k^{0})  \\
        \phi_1(\BFX_k^{1}) & \phi_2(\BFX_k^{1}) & \cdots & \phi_p(\BFX_k^{1})  \\
        \vdots & \vdots & \vdots & \vdots \\
        \phi_1(\BFX_k^{p}) & \phi_2(\BFX_k^{p}) & \cdots & \phi_p(\BFX_k^{p})  \\
        \end{bmatrix}.
    \end{equation*}
    {We say $\mcX_k$ is poised provided $\sfM(\Phi, \mcX_k)$ is nonsingular. }
    %The matrix $\sfM(\Phi, \mcX_k)$ is nonsingular if the set $\mcX_k$ is poised in $\mcB(\BFX_k;\Delta^q_k)$. 
    %A set $\mcX_k$ is $\Lambda-$poised in $\mcB(\BFX_k;\Delta^q_k)$ if $\Lambda \ge \max_{i=0,\dots,p}\max_{\BFz\in\mcB(\BFX_k;\Delta^q_k)}|l_i(\BFz)|$,where $l_i(\BFz)$ are the Lagrange polynomials. 
    If there exists a solution to~\eqref{eq:syslineq}, then the function $M^q_k:\mcB(\BFX_k;\Delta^q_k) \to \real$, defined as $M^q_k(\BFx) = \sum_{j=0}^{p} \nu^q_{k,j} \phi_{j}(\BFx)$ is a stochastic polynomial interpolation  of estimated values of $f^q$ on $\mcB(\BFX_k;\Delta^q_k)$. 
    In particular, if $\BFG^q_k=
    \begin{bmatrix}
    \nu^q_{k,1} & \nu^q_{k,2} & \cdots & \nu^q_{k,d} 
    \end{bmatrix}^\intercal$ and $\sfH^q_k$ is a symmetric $d\times d$ matrix with elements uniquely defined by $(\nu^q_{k,d+1},\nu^q_{k,d+2},\ldots, \nu^q_{k,p} 
    )$, then we can define the stochastic quadratic model $M^q_k: \mcB(\BFX_k;\Delta^q_k)\to\real$, as
\begin{equation}
    M^q_k(\BFx) = \nu^q_{k,0} +  (\BFx-\BFX_k)^\intercal \BFG^q_k + \frac{1}{2} (\BFx-\BFX_k)^\intercal \sfH^q_k(\BFx-\BFX_k).\label{eq:mdefn}
\end{equation}
\label{defn:polyintermd}
\end{definition}

{The next definition states that the approximate solution to the subproblem in ASTRO-BFDF ensures a sufficient decrease in the local model (see Theorem 10.1 in \cite{katya:DFObook}).}

\begin{definition}[Cauchy reduction: {Definition 2.5 in \cite{Sara2018ASTRO}}] Given $\BFX_k \in \real^d$,  $\Delta^q_k > 0$, $q\in\{h,\ell\}$, and a function $M_k^q:\mcB(\BFX_k;\Delta^q_k) \to \real$ obtained following Definition~\ref{defn:polyintermd}, $\BFS_k^c$ is called the Cauchy step if 
\begin{equation*}
    M^q(\BFX_k)-M^q(\BFX_k+\BFS_k^c) \ge \frac{1}{2}\|\nabla M^q(\BFX_k)\|\min\left\{ \frac{\|\nabla M^q(\BFX_k)\|}{\|\nabla^2 M^q(\BFX_k)\|}, \Delta^q_k \right\}.
\end{equation*}
When $\|\nabla^2 M^q_k(\BFX_k)\|=0$, we assume $\|\nabla M^q(\BFX_k)\|/\|\nabla^2 M^q(\BFX_k)\|=+\infty$. The Cauchy step is derived by minimizing the model $M^q_k(\cdot)$ along the steepest descent direction within $\mcB(\BFX_k;\Delta^q_k)$, making it easy and quick to compute.
\label{defn:cauchyred}
\end{definition}

{Lastly, we introduce the concepts of filtration and stopping time, which play a crucial role in analyzing the behavior of ASTRO-BFDF and BFAS.}

\begin{definition}[filtration and stopping time: {Section 35 in \cite{billingsley1995probability}}] A filtration $\{\mcF_{k}\}_{k \ge 1}$ on a probability space $(\Omega, \mathbb{P}, \mcF)$ is a sequence of $\sigma$-algebras, each contained within the next, such that for all $k$, $\mcF_{k}$ is a subset of $\mcF_{k+1}$, and all are subsets of $\mcF$. 
A function $N: \Omega \rightarrow \{0, 1, 2, \dots, \infty\}$ is referred to as a stopping time with respect to the filtration $\mcF$ if the set $\{\omega \in \Omega : N(\omega) = n\}$ is an element of $\mcF$ for every $n < \infty$.
\end{definition}

\subsection{Standing Assumptions}

We now { state } the standing assumptions {  underlying our analysis}. Assumption~\ref{assum:fn} { specifies } the characteristics of the functions $f^h$ and $f^{\ell}$, {which } precisely { define } the problem { under consideration}.

\begin{assumption}[function]
    The HF function $f^h$ and the LF function $f^{\ell}$ are continuously differentiable in an open domain $\Omega$, $\nabla f^h$ and $\nabla f^{\ell}$ are Lipschitz continuous in $\Omega$ with constant $\kappa_{Lg}>0$. 
    \label{assum:fn}
\end{assumption}

We make the next assumption on the higher moments of the stochastic noise { following a Bernstein-type condition adapted to a martingale setting}. Random variables fulfilling Assumptions~\ref{assum:martingale} exhibit a subexponential tail behavior. 
 
\begin{assumption}[stochastic noise]
    The Monte Carlo oracles generate iid random variables $F^q(\BFX^{i}_k,\xi^q_j) = f^q(\BFX^{i}_k)+E^{i,q}_{k,j}$ with $E^{i,q}_{k,j}\in\mcF_{k,j}$ for $i \in \{0,1,2,\dots,p,s\}$ and $q \in \{h,\ell\}$, where $\BFX_k^s$ is the candidate iterate at iteration $k$ and $\mcF_k=\mcF_{k,0}\subset\mcF_{k,1}\subset\dots\subset \mcF_{k+1}$ for all $k$. Then the stochastic errors $E_{k,j}^{i,q}$ are independent of $\mcF_{k-1}$, $\mbE[E_{k,j}^{i,q}\mid\mcF_{k,j-1}]=0$, and there exists $(\sigma^q)^2>0$ and $b^q>0$ such that for a fixed $n$, 
    \begin{align}
    \label{eq:stochstic-noise-assumption}
    \frac{1}{n}\sum_{j=1}^n\mbE[|E_{k,j}^{i,q}|^m\mid\mcF_{k,j-1}]\leq\frac{m!}{2}(b^q)^{m-2}(\sigma^q)^2,\ \forall m=2,3,\cdots,\forall k.
    \end{align}
    %Moreover, The Monte Carlo LF oracle also generates iid random variables $F^{\ell}(\BFX^{i}_k,\xi_j) = f^{\ell}(\BFX^{i}_k)+E^{i,l}_{k,j}$ with $E^{i,l}_{k,j}\in\mcF_{k,j}$ for $i \in \{0,1,2,\dots,p,s\}$. Similar to the HF case, the stochastic errors $E_{k,j}^{i,l}$ are independent of $\mcF_{k-1}$, $\mbE[E_{k,j}^{i,l}\mid\mcF_{k,j-1}]=0$, and there exists $(\sigma^{\ell})^2>0$ and $b^{\ell}>0$ such that for a fixed $n$, 
    %\begin{align*}
    %\frac{1}{n}\sum_{j=1}^n\mbE[|E_{k,j}^{i,l}|^m\mid\mcF_{k,j-1}]\leq\frac{m!}{2}(b^{\ell})^{m-2}(\sigma^{\ell})^2,\ \forall m=2,3,\cdots,\forall k.    
    %\end{align*}
    \label{assum:martingale}
\end{assumption}

\begin{comment}
 Next, to analyze the sums and maxima of the stochastic error estimates, we need to characterize the tail probability of a sequence of dependent random variables. 

\begin{assumption}
    For a given solution at iteration $k$ and constant $c>0$, there exists a large  $c_0>0$ such that for all $\lambda_k\leq n\leq N_k$, $\frac{1}{n-1}\sum_{j=1}^{n-1}|E_{k,j}|$ is stochastically decreasing in $|E_{k,n}|$. In other words, 
    the ratios below are $\mcO(1)$ for all $t\in[c_0,c]$. 
    \begin{align}
        \limsup_{c\to\infty}\sup_{c_0\leq t\leq c}\frac{\mbP\{\frac{1}{n-1}\sum_{j=1}^{n-1}|E_{k,j}|>c-t\ \mid\ |E_{k,n}|=t\}}{\mbP\{\frac{1}{n-1}\sum_{j=1}^{n-1}|E_{k,j}|>c-t\}} & <\infty.
    \end{align}
    \label{assum:dependence}
\end{assumption}

Assumption \ref{assum:dependence} allows for dependence among consecutive stochastic error estimates with a specific tail independence structure, enabling us to apply a similar tail probability result for heavy-tailed (sub-exponential) random variables. The non-restrictive dependence structure leads to the subexponentiality of the summands, eliminating the impact of dependence on the tail behavior of the sums. See \cite{ha2023} for more details.
\end{comment}

\subsection{Useful Results}
%\begin{lemma} \label{lem:minkowski} (Minkowski's Inequality)
%    Suppose that $E^h$ and $E^{\ell}$ are two random variables. Then 
%    \begin{equation*}
%        \mbE[|E^h+E^{\ell}|^p]^{1/p} \le \mbE[|E^h|^p]^{1/p} + %\mbE[|E^{\ell}|^p]^{1/p}.
%    \end{equation*}
%\end{lemma}

In this section, we present useful results that will be invoked to prove the almost sure convergence of ASTRO-BFDF. {We begin by introducing Bernsetin inequality for martingales.} %(see Theorem A under conditional Bernstein condition by \cite{2015fangraliu}).}

{
\begin{lemma}[Bernsetin inequality for martingales: Lemma 2.5 by \cite{ha2023}]
\label{lem:bernstein}
Let $(\xi_i,\mcF_i)_{i=0,1,\dots}$ be a martingale difference sequence on some probability space $(\Omega,\mcF,P)$ with $\mbE\left[\xi_i \vert \mcF_{i-1}\right] = 0$, where $\xi_0=0$ and $\{\Omega,\emptyset\} = \mcF_0 \subseteq \mcF_1 \subseteq \mcF_2 \subseteq \cdots \subseteq \mcF_n \subseteq \mcF$ is a sequence of increasing filtrations.
Furthermore, assume that there exist constants $b>0$ and  $\sigma_{\xi}^2>0$ such that for any $m \in \{2,3,\dots\}$, and any $i \in \{0,1,\dots\}$, $$\mbE\left[ |\xi_i|^m \vert \mcF_{i-1} \right] \leq \frac{1}{2}\, m! \, b^{m-2} \sigma_{\xi}^2.$$ 
%for some constants $b>0$, $\sigma_{\xi}^2>0$, any $m \in \{2,3,\dots\}$, and any $i \in \{0,1,\dots\}$.
Then, for any $c>0$ and any $n \in \mbN$, 
$$\mbP\left\{\sum_{i=1}^n \xi_i \geq nc \right\} \leq \exp\left\{- \frac{nc^2}{2(bc + \sigma_{\xi}^2)} \right\}. $$
%$$\mbP\left\{\sum_{i=1}^n \xi_i \geq nc \mbox{ and } \sum_{i=1}^n \mbE[\xi_i^2|\mcF_{i-1}] \leq n\sigma_{\xi}^2 \right\} \leq \exp\left\{- \frac{nc^2}{2(bc + \sigma_{\xi}^2)} \right\}. $$
\end{lemma}
}

The {next } result demonstrates that, under Assumption \ref{assum:martingale}, the estimate of the stochastic errors is bounded by the square of the TR radius when a specific adaptive sampling rule is applied. {This result matches the guaranteed estimation accuracy given by $\mbP\{| N(\BFX_k^i)^{-1} \sum_{j=1}^{N(\BFX_k^i)} E^{i,q}_{k,j}|\geq c_f(\Delta_k^q)^2)\} \le \alpha_k$ for any given $c_f > 0$, where $\alpha_k$ increases gradually, driven by a  logarithmically increasing sequence $\lambda_k$. This type of inflation factor is a common approach in sequential estimation settings \cite{ghosh:sequential1997}.} %Consequently, after a sufficiently large number of iterations, the true function $f$ decreases whenever the candidate solution is accepted.

\begin{theorem}[Stochastic noise: {Theorem 3.2 with case (A-0) by } \cite{ha2023}]\label{thm:asfinitedelta2}
Let $c_f>0$ and $\Delta_k^q > 0$ be given and $E^{i,q}_{k,j}$ denotes the stochastic noise following Assumption \ref{assum:martingale}. {Given any $\epsilon_\lambda \in (0,1)$ and $\lambda_0 \ge 2$, suppose the adaptive sample size $N(\BFX_k^i)$ is a stopping time satisfying } $N(\BFX_k^i) \ge (\sigma_0^q)^2 \lambda_k \kappa^{-2} (\Delta_k^q)^{-4},$ {where $\sigma_0 > 0$, $\kappa > 0$, and $\lambda_k = \lambda_0 (\log k)^{1+\epsilon_\lambda}$. Then, the following series is summable:}
\begin{equation*}
    \sum_{k=1}^\infty\mbP\left\{\left|\frac{1}{N(\BFX_k^i)}\sum_{j=1}^{N(\BFX_k^i)} E^{i,q}_{k,j}\right|\geq c_f(\Delta_k^q)^2 \right\} < \infty.
\end{equation*}
\end{theorem}

Although Theorem  \ref{thm:asfinitedelta2} has been proven with the following adaptive sampling rule (See Section 3.2 of \cite{ha2023}) 
\begin{equation} \label{eq:as-origin}
    N(\BFX_k^i) = \min\biggl\{ n \in \mbN:\frac{\max\{\sigma_0,\sigmahat^h(\BFX_k^i,n)\}}{\sqrt{n}}\leq\frac{\kappa(\Delta^q_{k})^{2}}{\sqrt{\lambda_k}}\biggr\},
\end{equation}
it can also be trivially established with a stopping time $N(\BFX_k^i) \ge \mcO(\lambda_k(\Delta_k^q)^{-4})$ by employing the same logical framework. {The detailed proof is provided in Appendix \ref{apdx:proofasfinitedelta2}. }
The next result provides an upper bound for the gradient error norm at any design point within the TR when a stochastic linear or quadratic interpolation model is used. Combined with Theorem \ref{thm:asfinitedelta2}, it indicates that the gradient error norm will be bounded by the order of the TR radius after sufficiently many iterations. 

\begin{lemma}[Stochastic Interpolation Model: {Lemma 2.9 by } \cite{Sara2018ASTRO}]\label{lem:stoch-interp}      
    {Let Assumptions \ref{assum:fn} and \ref{assum:martingale} hold. }
    If $M^q_k(\BFz)$ is a stochastic linear interpolation model or a stochastic quadratic interpolation model of $f^q$ with the design set $\mcX_k=\{\BFX_k^{i}\}_{i=0}^{p}\subset \mcB(\BFX_k;\Delta_k^q)$ and corresponding function estimates $\Fbar^q(\BFX_k^i,N(\BFX_k^i))= f^q(\BFX_k^i) + \Ebar^{i,q}_k(N^i_k)$, there exist positive constants $\kappa_{eg1}$ and $\kappa_{eg2}$ such that for any $\BFz \in \mcB(\BFX_k;\Delta_k^q)$, 
    \begin{equation} \label{eq:gradient-error-ub}
        \| \nabla M^q(\BFz) - \nabla f^q(\BFz)\| \le \kappa_{eg1}\Delta^q + \kappa_{eg2}\frac{\sqrt{\sum_{i=1}^p(\Ebar^{i,q}_k(N^i_k)-\Ebar^{0,q}_k(N^0_k))}}{\Delta^q},
    \end{equation}
    where $\Ebar^{i,q}_k(N^i_k) = N(\BFX_k^i)^{-1}\sum_{j=1}^{N(\BFX_k^i)} E^{i,q}_{k,j}$.
\end{lemma}

Lastly, we present the variance of BFMC estimator { under the Common Random Numbers (CRN) scheme. In the CRN setting, the same stochastic realization $\omega_i$ drives both HF and LF simulations, yielding paired samples $\xi_i^h$ and $\xi_i^{\ell}$. The induced positive correlation reduces the variance of the BFMC function estimates and the estimated function reductions between successive design points.}

\begin{lemma}[Variance of BFMC under CRN: {Lemma 3.3 by } \cite{karen2016mfmc}]
\label{lem:variance-bfmc}
{
%Define $E_{k,j}^{i,q}$ is the noise term specified in Assumption~\ref{assum:martingale}.  
Let $E_{k,j}^{i,q}$ denote the stochastic noise satisfying Assumption \ref{assum:martingale}.
The index $j$ corresponds to the $j$-th common random number $\omega_j$, so the pair
$(E_{k,j}^{i,h},E_{k,j}^{i,\ell})$ shares the same realization and has covariance $\sigma^{h,\ell}(\BFX_k^i)\;=\;
\operatorname{Cov}\!\bigl(E_{k,j}^{i,h},E_{k,j}^{i,\ell}\bigr).$
Then the variance of the BFMC estimator
$\Fbar^{\sfbf}(\BFX_k^i,n,v,c)$ is
\begin{equation} \label{eq:variance-bfmc-main}
    \begin{split}    
        \Var(\Fbar^{\sfbf}({\BFX_k^i},n,v,c)) = \frac{(\sigma^h({\BFX_k^i}))^2}{n} &+ c^2\left(\frac{1}{n} - \frac{1}{v} \right) (\sigma^{\ell}({\BFX_k^i}))^2 \\
        & +2c \left(\frac{1}{v}-\frac{1}{n}\right) \sigma^{h,\ell}({\BFX_k^i}),
    \end{split}
\end{equation}
for any $n,v\in\mathbb N$ with $n<v$, $c\in\mathbb R$, and
$i\in\{0,1,\dots,p,s\}$.
}

\end{lemma}

{We note that \eqref{eq:variance-bfmc} is identical to \eqref{eq:variance-bfmc-main} under Assumption \ref{assum:martingale} with CRN, since $\Cov(\Fbar^h(\BFX_k^i,n),\Fbar^{\ell}(\BFX_k^i,v)) = (\max\{n,v\})^{-1} \sigma^{h,\ell}(\BFX_k^i)$ and $\Var(\Fbar^q(\BFX_k^i,n)) = n^{-1}(\sigma^q(\BFX_k^i))^2$ for any $q \in \{h,\ell\}$ (see Lemma 3.2 by \cite{karen2016mfmc}). } 
{Since } $\sigma^h({\BFX^i_k}), \sigma^{\ell}({\BFX^i_k}),$ and $\sigma^{h,\ell}({\BFX^i_k})$ are usually unknown in reality, { we use their empirical estimates, such as } $\sigmahat^h({\BFX^i_k},n),\sigmahat^{\ell}({\BFX^i_k,v}),$ and $\sigmahat^{h,\ell}({\BFX^i_k,n})$. { Using these estimates in place of the true values does not affect the convergence of ASTRO-BFDF, but instead represents a more realistic implementation of the sampling scheme in practice. } { Lastly, we present Borel-Cantelli's First Lemma for martingales  that repeatedly invoke.
\begin{lemma}[Borel–Cantelli for martingales: Chapter 12.15 in \cite{1991wil}]
Let $\{A_n\}_{n\ge1}$ be a sequence of events on a probability space $(\Omega, \mathcal{F}, \mathbb{P})$, and let $\{\mathcal{F}_n\}$ be a filtration. If
$ \sum_{n=1}^\infty \mathbb{P}(A_n \mid \mathcal{F}_{n-1}) < \infty \quad \text{a.s.}, $
then $\mathbb{P}(A_n \text{ i.o.}) = 0.$
\end{lemma}}

\section{Adaptive Sampling Bi-fidelity Trust Region Optimization}
Similar to other stochastic TR algorithms, {our proposed adaptive sampling bi-fidelity TR algorithm (ASTRO-BFDF), which build upon } ASTRO-DF~\cite{Sara2018ASTRO}, generates { the sequence } $\{\BFX_k\}$ through the four steps outlined in Section \ref{sec:STRO}. { ASTRO-BFDF differs from ASTRO-DF in two key components, primarily due to the presence of an LF simulation oracle: }

\begin{itemize}
    \item[(1)] Sample sizes are { carefully } managed through adaptive sampling using BFMC or CMC. 
    Within this approach, two critical decisions are made. {First}, as samples { are collected}, the { algorithm determines whether to employ CMC or BFMC.} Second, it dynamically adjusts the sample sizes for both HF and LF simulation oracles { and determines } the coefficient $c$ in~\eqref{eq:bfmc}, all in real time as { sampling progresses.}
    \item[(2)] At each iteration $k$, two local models can be constructed using HF and LF simulation oracles, { each associated with } its own TR: $\Delta_k^{\ell}$ for the LF function and $\Delta_k^h$ for the HF function. The { LF-based } local model serves two purposes: (i) identifying { a } candidate solution for the next iterate, and (ii) updating the adaptive correlation constant {($\alpha_k$ in Algorithm~\ref{alg:TRO-MFDF} and~\ref{alg:TRO-LFDF}).}
\end{itemize}

We first introduce the bi-fidelity adaptive sampling (BFAS) strategy, which corresponds to the first feature.

\begin{algorithm}[htp]  %\myalgsize
\caption{[$N_k(\BFx),V_k(\BFx),C_k(\BFx),\Ftilde_k(\BFx)$]=\texttt{BFAS}$(\BFx,\Delta_k,\lambda_k,\kappa, {s^h, s^\ell,w^h,w^{\ell}})$}
\label{alg:BFAS}
\begin{algorithmic}[1]
\REQUIRE $\BFx \in \real^d$, TR radius $\Delta_{k}$, sample size lower bound sequence $\{\lambda_k\}$, batch size $s^h < s^{\ell}$ for HF and LF oracles, adaptive sampling constant $\kappa>0$, lower bound of an initial variance approximation $\sigma_0>0$, and {costs $w^h \ge w^{\ell}$ of calling HF and LF oracles.}
{\ENSURE HF sample size $N_k(\BFx)$, LF sample size $V_k(\BFx)$, MFMC coefficient $C_k(\BFx)$, and function estimate $\Ftilde_k(\BFx)$}

\STATE \label{ASBF:set-n} Set $n = (\sigma_0)^2\lambda_k(\kappa^2\Delta_k^4)^{-1}$ and $v = n+1$.
\STATE \label{ASBF:initial-estimate} Estimate $\sigmahat^h(\BFx,n)$, $\sigmahat^{h,\ell}(\BFx,n)$, and $\sigmahat^{\ell}(\BFx,v)$.

\STATE \label{ASBF:start} Obtain the {CMC}-predicted sample sizes {$N^p(\BFx)$}, where
\begin{equation} \label{eq:as}
    N^p(\BFx)=\min\biggl\{ n^p \in \mbN:\frac{\sigmahat^h\left(\BFx,n\right)}{\sqrt{n^p}}\leq\frac{\kappa\Delta_{k}^{2}}{\sqrt{\lambda_k}}\biggr\}
\end{equation}

\ALOOP{}

\STATE \label{ASBF:solve-problem} Approximately compute $C^*, N^*$, and $V^*$ by solving the problem \eqref{eq:bfmc-problem} and set $c = C^*$.

\IF{$w^h N^* + w^{\ell} V^* \le w^h {N^p(\BFx)}$} \label{ASBF:check-bf-cmc}
\STATE \label{ASBF:bfmc-start} Set $v = \max\{n+1,v\}$ and update $\sigmahat^{h,\ell}(\BFx,n)$ and $\sigmahat^{\ell}(\BFx,v)$ by calling the LF oracle.
\IF{$\Var(\Fbar^{\sfbf}(\BFx,n,v,c)) \le \kappa^2 \Delta_k^4 \lambda_k^{-1} $ {(Condition 1)}} \label{ASBF:bfmc-test}
\RETURN $[n,v,c,\Fbar^{\sfbf}(\BFx,n,v,c)]$ (BFMC) 
\ENDIF

\IF{$n \ge N^*-1$}
\STATE Set $v = v + s^{\ell}$ and get $s^{\ell}$ additional replications of the LF oracle and update $\widetilde{\sigma}^{\ell}(\BFx,v)$.
\ELSE
\STATE Set $n = n+s^h$ and update $\sigmahat^h(\BFx,n)$ and $\sigmahat^{h,\ell}(\BFx,n)$ by calling the LF and HF oracles.
\ENDIF

\ELSE
\IF{$n \ge N^p(\BFx)$ {(Condition 2)}}
{\IF{$\Var(\Fbar^{\sfbf}(\BFx,n,v,c)\le\Var(\Fbar^h(\BFx,n))$ and $n<v$}
    \RETURN $[n,v,c,\Fbar^{\sfbf}(\BFx,n,v,c)]$ (MFMC)
\ELSE
    \RETURN $[n,v,c,\Fbar^h(\BFx,n)]$ (CMC)
\ENDIF}

\ENDIF
\STATE \label{ASBF:cmc-n-update} Set $n = n+s^h$ and update $\sigmahat^h(\BFx,n)$ and $N^p(\BFx)$ by calling the HF oracle.

\ENDIF

\ENDALOOP
 
\end{algorithmic}
\end{algorithm}

\subsection{Adaptive Sampling for Bi-Fidelity Stochastic Optimization}
While BFMC { can } reduce the variance of function { estimates, } blindly employing BFMC may not always be advantageous. 
For example, when the inherent variance of the LF simulation significantly exceeds that of the HF simulation, { a substantial number of LF samples are required to reduce BFMC variance (see Lemma~\ref{lem:variance-bfmc}). In such cases, } if { the LF oracle } is { only marginally cheaper } than the HF { oracle}, { the cost advantage of BFMC diminishes, making CMC potentially more efficient}. Therefore, it is essential to { determine } which Monte Carlo method to employ at a given design point based on the variance of the LF and HF simulation { outputs } and the covariance between them. 
However, { the true variances of the HF and LF simulations and their covariance are unknown. } { Consequently, }  when adaptive sampling is used, the choice of the MC method needs to be dynamically determined based on variance and covariance estimates, which are sequentially updated { from } simulation results.
In summary, { it is necessary to dynamically determine } $N,V$, and $C$ while { collecting } simulation replications, where $N$ and $V$ are the sample sizes for the HF and LF oracles and $C$ represents the coefficient in the BFMC estimate, (denoted as $c$ in~\eqref{eq:bfmc}). To achieve this for any $\BFx \in \real^d$ at iteration $k$, we suggest BFAS, as listed in Algorithm~\ref{alg:BFAS}.

Algorithm~\ref{alg:BFAS} starts by sampling $n$ HF oracle { calls } and $v$ LF oracle { calls } to estimate the variance and covariance terms in~\eqref{eq:variance-bfmc}.
{ Using } the variance estimate $\sigmahat^h(\BFx,n)$, we { compute } a predicted minimum sample size $N^p(\BFx)$ for CMC, adhering to the adaptive sampling rule~\eqref{eq:as-origin}. 
Then the predicted computational cost of CMC at $\BFx$ is $w^h N^p(\BFx)$, where $w^h$ is the cost of { one HF oracle call. }  
{Next, this cost is compared to the projected computational cost of BFMC. } To predict the { minimum cost for BFMC}, we solve~\eqref{eq:bfmc-problem} { using the } variance estimates $\sigmahat^h(\BFx,n)$, $\sigmahat^{\ell}(\BFx,v),$ and $\sigmahat^{h,\ell}(\BFx,n)$ for $\Var(\bar{F}^{\sfbf}(\BFx,\widetilde{n},\widetilde{v},\widetilde{c}))$ (See Lemma~\ref{lem:variance-bfmc}).
\begin{equation} \label{eq:bfmc-problem}
\begin{split}
    [N^*,V^*,C^*]  \in \argmin_{\widetilde{n},\widetilde{v},\widetilde{c} \in \real} \quad w^h \widetilde{n} + w^{\ell}\widetilde{v}  \\
    \text{subject to} \quad
     \Var(\bar{F}^{\sfbf}(\BFx,\widetilde{n},\widetilde{v},\widetilde{c})) &\le \kappa^2 \Delta_k^4 \lambda_k^{-1} \\
    \widetilde{n}-\widetilde{v} &\le 0 \\
     n \le \widetilde{n} &\le \infty \\ 
     v \le \widetilde{v} &< \infty,
\end{split}
\end{equation}
where $w^{\ell}$ is the cost { of a single LF oracle call}.
The first constraint { corresponds to } the adaptive sampling rule from~\eqref{eq:as-origin}, ensuring that { the } BFMC { estimator } achieves the required accuracy. {Problem~\eqref{eq:bfmc-problem} is a three-dimensional deterministic continuous optimization problem that can be efficiently solved in practice. } We { next compare } the predicted computational costs between CMC and BFMC. 

If the { predicted } cost of BFMC is lower than that of CMC, BFMC { is expected to be more } cost-effectiveness, { and } the algorithm proceeds to Step \ref{ASBF:bfmc-start}. We first set $v = \max\{v,n+1\}$ with the updated $n$ in Step \ref{ASBF:cmc-n-update} to ensure that $v > n$. Following this adjustment, additional { LF oracle } replications may { be required, and we } update both $\sigmahat^{l}(\BFx,v)$ and $\sigmahat^{h,\ell}(\BFx,n)$ { accordingly. Next, } if the variance of the BFMC { estimator satisfies the accuracy condition} , i.e., 
\begin{equation} \label{eq:bfmc-test}
    \Var(\Fbar^{\sfbf}(\BFx,n,v,c))  \le \kappa^2 \Delta_k^4 \lambda_k^{-1},
\end{equation}
the algorithm returns $\Fbar^{\sfbf}(\BFx,n,v,c)$. { Otherwise, the algorithm determines } whether to increase $n$ or $v$. 
{Since $n$ has a stronger effect on reducing the variance, the algorithm first checks whether $n < N^* - 1$. If so, $n$ is increased by $s^h$, and the relevant variance and covariance estimates are updated. } If $n \ge N^* - 1$, { additional LF oracle samples are obtained, and $\sigmahat^{\ell}(\BFx, v)$ is updated. The algorithm then returns to Step~\ref{ASBF:solve-problem}.}

{After solving the subproblem, if the projected cost of achieving the required accuracy using CMC is lower than that of BFMC, the algorithm switches to CMC to avoid unnecessary LF sampling. }
If $n \ge N^p$, { the function estimate is sufficiently accurate for optimization, and the algorithm returns the estimator (CMC or BFMC) with the lower estimated variance, ensuring the most accurate result. } If not, { $n$ is increased, the variance estimate is updated, and the process continues. Note that the LF variance estimate and covariance estimate are not updated at this stage. }

Since $n,v,$ and $c$ are dynamically determined based on { simulation realizations, } three outputs { of BFAS } are the stopping times determined by the filtration. Hence, we refer the output of Algorithm~\ref{alg:BFAS} as $[N_k(\BFx),V_k(\BFx),C_k(\BFx),\Ftilde_k(\BFx)]$. 

\begin{remark}[{Computational Costs of BFAS}]
    {The total cost of BFAS is given by $w^h N_k(\BFx)+ w^{\ell} V_k(\BFx)$. If either Condition 1 or Condition 2 is met based on the initial estimates, this cost reduces to $w^h n + w^{\ell} v$, which represents the minimum possible cost. If the final output is the BFMC estimate, the computational cost is strictly lower than that of CMC becuase of Step~\ref{ASBF:check-bf-cmc}. Conversely, when the output is CMC, the $V_k(\BFx)$ LF replications do not contributed to the final estimate. The adaptive sampling scheme mitigates this redundancy by ensuring that only the HF oracle is called when CMC appears more cost-effective but Condition 2 is not yet satisfied (Step~\ref{ASBF:cmc-n-update} in Algorithm~\ref{alg:BFAS}). Furthermore, LF oracle replications are reused to construct the LF-based local model, as detailed in Section \ref{sec:bi-fidelity-TRO}.
    }
\end{remark}

%When $n$ is less than $N^*$ which implies that the replications of the HF oracle we already have are insufficient, we increase $n$

%\begin{remark}
%    Problem \eqref{eq:bfmc-problem} is a continuous problem, inspired from \cite{karen}.
%\end{remark}

%Then we update $\sigmahat^{h}(\BFx,\lambda_k+b)$ and $N^p(\BFx)$ by calling $b$ additional replications of the HF oracle. After updating both $N^p$ and $n$, we proceed to reassess whether $n$ replications provides adequate information to derive a sufficiently accurate estimation of the function. If it does, the algorithm then return $\overline{F}(\mathbf{x}, n)$. Otherwise, we go from Step \ref{ASBF:cmc-end} to Step \ref{ASBF:} and update $\sigmahat_{h,\ell}(\BFx,n)$ $\sigmahat_{h,\ell}(\BFx,n)$ by obtaining $b$ additional replications of the LF oracle. 

\subsection{Bi-fidelity Stochastic Trust Region Method}
\label{sec:bi-fidelity-TRO}

We now delve into the bi-fidelity stochastic TR method with adaptive sampling (ASTRO-BFDF). { The core idea of ASTRO-BFDF } is to { primarily exploit } the LF oracle until it no longer yields { a solution that improves upon the current incumbent solution}. Hence, unlike the stochastic TR methods discussed in Section \ref{sec:STRO}, { ASTRO-BFDF constructs two local models within separate TRs } tailored { to the } HF and LF functions, { denoted by $\Delta_k^h$ and $\Delta_k^{\ell}$, respectively. Note that } $\Delta_k^h$ is { always larger } than $\Delta_k^{\ell}$ for { all } iterations to maintain large steps for the { HF-based }local model and save computational budget in~\eqref{eq:as} and~\eqref{eq:bfmc-test}.

\begin{algorithm}[htp]  %\myalgsize
\small
\caption{\texttt{ASTRO-BFDF}}
\label{alg:TRO-MFDF}
\begin{algorithmic}[1]
\REQUIRE Initial incumbent $\BFx_{0}\in\real^d$, initial and maximum TR radius $\Delta^{\ell}_{0}, \Delta^h_{0}, \Delta_{\max}>0$, model fitness thresholds $0<\eta<1$ and certification threshold $\mu>0$, expansion and shrinkage constants $\gamma_1>1$ and $\gamma_2\in(0,1)$, sample size lower bound sequence $\{\lambda_k\} = \{\mcO(\log k)\}$, adaptive sampling constant $\kappa>0$, correlation constant $\alpha_k > 0$, and lower bound of an initial variance approximation $\sigma_0>0$, {and } sufficient reduction constant $\zeta>0$. %and gradient norm of the LF model lower bound $\hat{\epsilon}>0$.
{\ENSURE Solution sequence \{$\BFX_k$\}}

\FOR{$k=0,1,2,\hdots$}
\STATE {Set $I_k^h=$True.}
\WHILE{$\alpha_k < \alpha_{th}$}
\STATE \label{ASBFTRO:callingLF} Obtain { $\check{\BFX}_k^{\ell}$ and $I_k^h$ } by calling Algorithm \ref{alg:TRO-LFDF}.
\IF{$I_k^h$ is True}
{\STATE Set $\Delta_k^{\ell} = \gamma_2\Delta_k^{\ell}$ and $\alpha_k = \gamma_2 \alpha_k$}  
\ELSE
\STATE Set $(\BFX_{k+1},\Delta^{\ell}_{k+1})= ({\check{\BFX}_k^{\ell}}, \gamma_1\Delta^{\ell}_k)$ and $\alpha_{{k+1}} = \min\{\gamma_1 \alpha_k,1\}$.
\STATE \label{LF:delta-h-update} Set $\Delta_{k+1}^h = \max\{\Delta_{k+1}^{\ell}, \Delta_k^h\}$ and  $k = k+1$.
\BREAK
\ENDIF
\ENDWHILE
\IF{$I_k^h$ is True}
\STATE \label{ASBFTRO:designsetselect} Select $\mcX_{k}=\{ \BFX_{k}^{i}\}_{i=0}^{2d}\subset\mcB(\BFX_{k};\Delta^h_{k})$.

\STATE \label{ASBFTRO:estimate-hf} Estimate the HF function at $\{ \BFX_{k}^{i}\}_{i=0}^{2d}$ by calling Algorithm \ref{alg:BFAS} with $\Delta_k = \Delta_k^h$.

\STATE Estimate the LF function $\Fbar^{\ell}(\BFX_k^i,T_k^i)$ at $\{ \BFX_{k}^{i}\}_{i=0}^{2d}$, satisfying 
\begin{equation} \label{eq:as-lf}
    T_k^i = \min\biggl\{t \in \mbN:\frac{\max\{\sigma_0,\sigmahat^{\ell}\left(\BFX_k^i,t\right)\}}{\sqrt{t}}\leq\frac{\kappa(\Delta^h_{k}
    )^{2}}{\sqrt{\lambda_k}}\biggr\}.
\end{equation}

\STATE Construct local models $M_k^{\ell}(\BFX)$ and $M_k^h(\BFX)$.
\STATE Approximately compute the local model minimizers $$\BFX^{s,h}_k\in\argmin_{\left\| \BFX-\BFX_k \right\| \leq \Delta^h_{k}}M^h_{k}(\BFX) \text{ and } \BFX^{s,\ell}_k\in\argmin_{\left\| \BFX-\BFX_k \right\| \leq \Delta^h_{k}}M^{\ell}_{k}(\BFX).$$
\STATE Estimate $\Ftilde(\BFX^{s,h}_k)$ and $\Ftilde(\BFX^{s,\ell}_k)$ by calling Algorithm \ref{alg:BFAS} with $\Delta_k = \Delta_k^h$.
\STATE Set the candidate point $\BFX^{s}_k \in \argmin_{\BFx\in\{\BFX^{s,h}_k,\BFX^{s,\ell}_k\}} \Ftilde(\BFx).$

\STATE \label{eq:success-ratio-bf} Compute the success ratio $\rhohat_k$ and $\rhohat^{\ell}_k$ as
\begin{equation*}
    \rhohat_k = \frac{\Ftilde_k(\BFX_k^0)-\Ftilde_k(\BFX_k^{s,{h}})}{M_k^h(\BFX_k^0)-M_k^h(\BFX_k^{s,{h}})}
    \text{ and }    
    \rhohat^{\ell}_k = \frac{\Ftilde_k(\BFX_k^0)-\Ftilde_k(\BFX_k^{s,\ell})}{\max\{\zeta(\Delta_k^h)^2,M_k^{{h}}(\BFX_k^0)-M_k^{{h}}(\BFX_k^{s,\ell})\}}.
\end{equation*}
\STATE If $\rhohat^{\ell}_k \ge \eta$, set ${\alpha_{k+1}}=\gamma_1 \alpha_k$; otherwise set ${\alpha_{k+1}}=\gamma_2 \alpha_k$.
\STATE \label{HF:delta-l-update} Set $(\BFX_{k+1},\Delta^h_{k+1})=$ \[
\begin{cases}
    (\BFX_k^s,\min\{\gamma_1\Delta^h_{k}, \Delta_{\max}\})& \text{if } \rhohat_k \ge \eta \text{ and }\mu\|\nabla M^h_{k}(\BFX_{k})\| \ge \Delta^h_{k}, \\
    (\BFX_{k},\gamma_2\Delta^h_{k})              & \text{otherwise},
\end{cases}
\] $\Delta^{\ell}_{{k+1}} = \min\left\{\Delta^{\ell}_k,\Delta^h_k\right\}$, and $k=k+1$.

\ENDIF
\ENDFOR
\end{algorithmic}
\end{algorithm}

\begin{algorithm}[htp]  %\myalgsize
\small
\caption{{[$\check{\BFX}_k^{\ell},I_k^h$]} = \texttt{ASTRO-LFDF}$(\BFX_k)$}
\label{alg:TRO-LFDF}
\begin{algorithmic}[1]
\REQUIRE $\BFX_k$, $\Delta^{\ell}_{k}$, model fitness thresholds $0<\eta<1$ and certification threshold $\mu>0$, sufficient reduction constant $\theta>0$, expansion and shrinkage constants $\gamma_1>1$ and $\gamma_2\in(0,1)$, sample size lower bound sequence $\{\lambda_k\}=\{\mcO(\log k)\}$, adaptive sampling constant $\kappa>0$, correlation constant $\alpha_k > 0$, correlation threshold $\alpha_{th}>0$, lower bound of an initial variance approximation $\sigma_0>0$, {and } sufficient reduction constant $\zeta>0$. %gradient norm of the LF model lower bound $\hat{\epsilon}>0$.
{\ENSURE candidate solution $\check{\BFX}_k^{\ell}$ and indicator $I_k^h$ for HF-based local model construction.}

\STATE \label{ASLFTRO:designsetselect} Select $\mcX^{\ell}_{k}=\{ \BFX_{k}^{i}\}_{i=0}^{p}\subset\mcB(\BFX_{k};\Delta^{\ell}_{k})$.
\STATE Estimate $\Fbar^{\ell}(\BFX_k^i,T_k^i)$ at $\{ \BFX_{k}^{i}\}_{i=0}^{2d}$, satisfying \eqref{eq:as-lf} with $\Delta_k^{\ell}$ instead $\Delta_k^h$.

\STATE Construct local model $M_k^{\ell}(\BFX)$.
%\STATE Estimate LF function and construct a local model $M^{\ell}_k(\BFX)$
\STATE Approximately compute the local model minimizer $$\BFX^{s,\ell}_{k}=\argmin_{\left\| \BFX-\BFX_k \right\| \leq \Delta^{\ell}_{k}}M^{\ell}_{k}(\BFX).$$
\STATE Estimate $\Ftilde_k(\BFX_k^{s,\ell})$ and $\Ftilde_k(\BFX_k^0)$ by calling Algorithm \ref{alg:BFAS} with $\Delta_k = \Delta_k^{\ell}$.
\STATE \label{LF:success-ratio} Compute the success ratio $\rhohat_k$ as
\begin{equation} \label{eq:success-ratio-lf}
    {\rhohat_k^{\ell}} = \frac{\Ftilde_k(\BFX_k^0)-\Ftilde_k(\BFX_k^{s,\ell})}{\max\{\zeta(\Delta_k^h)^2,M_k^{\ell}(\BFX_k^0)-M_k^{\ell}(\BFX_k^{s,\ell})\}}.
\end{equation}
\IF{$\rhohat_k \ge \eta$} \label{LF:sufficient-test} %Use LF one
{\RETURN $[\check{\BFX}_k^{\ell} = \BFX_k^{s,\ell}, I_k^h = \text{False}]$}
\ELSE
{\RETURN $[\check{\BFX}_k^{\ell} = \BFX_k^{0}, I_k^h = \text{True}]$}
\ENDIF

\end{algorithmic}
\end{algorithm}

With two TRs, { it is essential to } determine when and how to construct the { LF-based } local model, { as uncontrolled use of the LF oracle can lead to inefficiency~\cite{toal2015some}. The main difficulty is assessing the local usefulness of the LF model for optimization without performing costly global comparisons with the HF model. }  Previous studies have attempted to address this by quantifying the correlation between the HF and LF functions through sampling various design points and their corresponding function estimates~\cite{lv2021multi,muller2020bfgo,song2019radial}. However, as { will be shown } in Section \ref{sec:numerical}, low { global } correlation does not { always mean that the LF function is unhelpful for optimization}. { In } the TR method, { it is necessary to ascertain the utility of the LF oracle within the TR rather than across the entire domain. Therefore, we introduce an adaptive correlation constant, $\alpha_k$, to decide whether constructing the LF-based local model $M_k^{\ell}$ is worthwhile. }

The adaptive correlation constant is dynamically updated leveraging results from previous iterations, { similar to the TR radius}. This enables { the algorithm } to { evaluate } whether the LF function contributes to optimization within the current TR. { At } each iteration $k$, { the algorithm checks if } $\alpha_k$ { exceeds } a user-defined threshold $\alpha_{th} \in (0,1)$. 
If { so, it } constructs $M_k^{\ell}$ using { the LF } design set $\mathcal{X}^{\ell}_k$ and corresponding estimates, { then minimizes $\mathcal{X}^{\ell}_k$ within the TR~$\mcB(\mathbf{X}_k;\Delta_k^{\ell})$}. If the candidate { achieves } a sufficient decrease in the HF function and { has a sufficiently large model gradient norm}, it is accepted, the TR expands, and $\alpha_k$ increases. 
{ Otherwise}, the candidate is rejected, leading to the contraction of $\Delta_k^{\ell}$, a decrease in $\alpha_k$, and progression to Step \ref{ASLFTRO:designsetselect} in Algorithm \ref{alg:TRO-LFDF} to identify a superior candidate within the shrunken TR. This process { repeats } until $\alpha_k < \alpha_{th}$, { indicating that the LF oracle is no longer helpful. } { We note that the initial value $\alpha_0$ should be chosen carefully, as poor initialization can slow early progress. However, its impact diminishes as the algorithm progresses, similar to the effect of an improperly chosen initial TR radius.}%the sensitivity is similar to that of a poor initial TR radius: it may slow early progress but becomes less significant as the algorithm proceeds.}

\begin{remark}
\label{remark:suff-red-lf}    
    The sufficient reduction test { for the LF model, involving $\hat{\rho}_k^{\ell}$, differs from the test involving $\hat{\rho}_k$ } in Algorithm \ref{alg:TRO-MFDF}. {Specifically}, for a successful iteration, the reduction in function estimates must be larger than $\zeta (\Delta_k^h)^2$ for some $\zeta > 0$ (See Step \ref{LF:success-ratio} in Algorithm~\ref{alg:TRO-LFDF}). 
    {This condition prevents the acceptance of a candidate solution based on an insignificant reduction in the local model value in } \eqref{eq:success-ratio-lf}, { and is also reflected in the convergence analysis of ASTRO-BFDF.}
\end{remark}

When { the LF-based search } fails to identify the next iterate, { the algorithm } constructs the { HF-based } local model $M_k^h$. The design set $\mathcal{X}_k$ { is formed by reusing previously visited points whenever possible, including those from the LF-based models in the current iteration. } HF function { values are then estimated using BFAS, which also provides LF estimates. } Then, we can additionally derive estimates for the LF function $\bar{F}^{\ell}(\mathbf{X}^i_k,T_k^i)$, aligning with the adaptive sampling rule \eqref{eq:as-lf}, { by obtaining $\max\{0,T_k^i-V_k^i\}$ more LF samples, which is typically small in practice and thus incurs minimal computational burden. Both HF-based and LF-based local models are then constructed, their minimizers evaluated, and the better candidate is chosen as the next iterate. Finally, the TR radius $\Delta_k^h$ and the adaptive correlation constant are updated based on the sufficient reduction test.}

In Algorithm \ref{alg:TRO-MFDF}, the {LF-based } local models { are constructed } at various points, each serving distinct purposes. 
Specifically, within Algorithm \ref{alg:TRO-LFDF}, { the model is employed } to seek an improved solution for the HF function. This decision stems from the belief that, {when $\alpha_k > \alpha_{th}$, the LF model will provide sufficiently informative local approximations that can support the HF optimization. }
Thus, { HF oracle usage } is minimized. In the outer loop, %{ LF-based models help update $\alpha_k$, leveraging LF samples obtained via BFMC at negligible extra cost. }
the primary objective is to update the adaptive correlation constant, even in cases where the LF function has not proven beneficial in preceding iterations. In this case, our aim is to minimize reliance on the LF function, a goal achievable through the adoption of BFMC. When the HF function values are estimated at Step \ref{ASBFTRO:estimate-hf} in ASTRO-BFDF, { a substantial number of LF samples } are already { available from }  BFMC, { enabling the construction of } $M_k^{\ell}$ { at negligible additional cost}. {A drawback is that when the LF-based model fails to contribute meaningfully, Algorithm~\ref{alg:TRO-MFDF} may incur higher computational cost than HF-only solvers. This additional cost is an inherent trade-off for assessing the correlation, which is essential for determining its role in optimization. However, BFAS mitigates this overhead in practice by limiting unnecessary LF oracle evaluations.}

\section{Convergence Analysis}
\label{sec:convergence}
In this section, we { present the convergence analysis of ASTRO-BFDF}. We first introduce two additional assumptions concerning the local model. First, the minimizer of the local model { is assumed to satisfy } a certain degree of function reduction, known as the Cauchy reduction (See Definition \ref{defn:cauchyred}). Second, the Hessian of the local model { is assumed to be } uniformly bounded. Both of these assumptions are essential to validate the quality of the candidate point { generated at each iteration}. 

\begin{assumption}[Reduction in Subproblem]
    \label{assum:fcd}
    %($\kappa_{fcd}$: fraction of the Cauchy decrease) 
    %Let $M:\mcB(\BFx,\Delta) \rightarrow \real$ be a function obtained following Definition \ref{defn:polyintermd}. 
    For some $\kappa_{fcd}\in(0,1]$, $q \in \{h,\ell\}$, and all $k$, $M^q_k(\BFX^0_k)- M^q_k(\BFX^{s,q}_k)\ge \kappa_{fcd}\left(M^q_k(\BFX^0_k) - M^q_k(\BFX_k^0+\BFS^c_k)\right)$,
    where $\BFS_k^c$ is the Cauchy step.
\end{assumption}

\begin{assumption}[Bounded Hessian in Norm]
In ASTRO-BFDF, the local model Hessians $\sfH_k^q$ are bounded by $\kappa^q_\sfH$ for all $k$ and $q\in\{h,\ell\}$ with $\kappa^q_\sfH \in (0, \infty)$ almost surely.\label{assum:hessian-norm}
\end{assumption}

\subsection{{Main Result}}
The convergence analysis of { ASTRO-DF } has received considerable attention in prior works such as \cite{ha2023,Sara2018ASTRO}. { Although our analysis follows a similar framework, } there are two crucial considerations we must address.

\begin{enumerate}
    \item[(a)] (Stochastic noise) BFMC { must satisfy conditions analogous to those in Theorem~\ref{thm:asfinitedelta2}, ensuring }  that the stochastic error in BFMC { is } less than $\mcO((\Delta^q_k)^2)$ for $q\in \{h,\ell\}$ after sufficiently large $k$. To achieve this, a crucial prerequisite (See Assumption~\ref{assum:martingale}) is ensuring that $F^h(\BFx,{\xi^h}) - c F^{\ell}(\BFx,{\xi^{\ell}})$ exhibits similar properties to $F^h(\BFx,{\xi^h})$ for any ${\xi^h} \in {\Xi^h}$, $c > 0$, and $\BFx \in \real^d$. %This will be established as the initial step towards proving the convergence theory.
    \item[(b)] (Trust-region) The TR sizes for both HF and LF functions need to converge to zero. {  Similar to other stochastic TR methods~\cite{chen2018storm,ha2023jsim,Sara2018ASTRO}, this condition is essential } because function { estimation } errors { are bounded, even in the worst-case, on the order of $\mathcal{O}(\Delta_k^p)$ for some $p > 0$ under } specific sampling rules and assumptions. Consequently, the estimation errors will also converge to zero, ensuring the accuracy of the estimates. Therefore, within { BFSO}, we also need the same result for $\Delta_k^h$. Furthermore, since $\Delta_k^h \ge \Delta_k^{\ell}$ for all $k \in \mbN$, the convergence of $\Delta_k^h$ implies the convergence of $\Delta_k^{\ell}$ as well.
\end{enumerate}

Taking into account the aforementioned considerations, we are now poised to present the convergence theory of ASTRO-BFDF.

\begin{theorem}[Almost Sure Convergence] \label{thm:almostsureconvergence}
    Let Assumptions \ref{assum:fn}-\ref{assum:hessian-norm} hold. Then, 
    \begin{equation} \label{eq:lim-conv}
        \lim_{k\rightarrow\infty} \|\nabla f^h(\BFX_k) \| \xrightarrow[]{w.p.1} 0.
    \end{equation}
\end{theorem}

Theorem \ref{thm:almostsureconvergence} guarantees that a sequence $\{\BFX_k(\omega)\}$ generated by Algorithm \ref{alg:TRO-MFDF} converges to the first-order stationary point for any {solution } sample path $\omega$. 

\subsection{{Proof of Theorem \ref{thm:almostsureconvergence}}} %We now establish the proof of the convergence theory (Theorem \ref{thm:almostsureconvergence}). 
We start by demonstrating that the iid random variables $E_{k,j}^{i,h} - cE_{k,j}^{i,l}$ also fulfill Assumption \ref{assum:martingale} for any $k \in \mbN$, $c \in \real$, and $i \in \{0,1,2,\dots,p,s\}$, indicating their adherence to {a } sub-exponential distribution. 

\begin{lemma} \label{lem:mfmc-supexp}
    Let Assumption \ref{assum:martingale} holds. Then there exist $\sigma^2>0$ and $b>0$ such that for a fixed $n$ and $c \in \real$, 
    \begin{equation} \label{eq:martingale-mfmc}
    \frac{1}{n}\sum_{j=1}^n\mbE[|E_{k,j}^{i,h}-c E_{k,j}^{i,l}|^m\mid\mcF_{k,j-1}]\leq\frac{m!}{2}b^{m-2}\sigma^2,\ \forall m=2,3,\cdots,\forall k.
    \end{equation}
\end{lemma}
\begin{proof}
    {Let us first arbitrary choose $\sigma^h$, $\sigma^{\ell}$, $b^h$, and $b^{\ell}$ such that $\sigma^h > \sigma^{\ell} > 0$, $b^h > b^{\ell}$, and Assumption 2 holds. } 
    We obtain from the Minkowski inequality and Assumption \ref{assum:martingale} that for a any $k,j \in \mbN$, $c \in \real$, and any $m \in \{2,3,\cdots\}$, there exist $b^h, b^{\ell}, (\sigma^h)^2,(\sigma^{\ell})^2>0$ such that 
    \begin{equation} \label{eq:martingale-mfmc-proof}
    \begin{split}
        \mbE[|E_{k,j}^{i,h} -c E_{k,j}^{i,l}|^m\mid\mcF_{k,j-1}]         &\le \left(\mbE[|E_{k,j}^{i,h}|^m\mid\mcF_{k,j-1}]^{\frac{1}{m}} + \mbE[c|E_{k,j}^{i,l}|^m\mid\mcF_{k,j-1}]^{\frac{1}{m}}\right)^m\\
        &\le \left((\frac{m!}{2}(b^h)^{m-2}(\sigma^h)^2)^{\frac{1}{m}} + c(\frac{m!}{2}(b^{\ell})^{m-2}(\sigma^{\ell})^2)^{\frac{1}{m}}\right)^m.
    \end{split}
    \end{equation}
    Then there must exist some constant $\alpha_{\sigma}, \alpha_b >1$ such that $\alpha_{\sigma} (\sigma^{\ell})^2 = (\sigma^h)^2$ and $\alpha_b b^{\ell} = b^h$. Then the right-hand side of \eqref{eq:martingale-mfmc-proof} becomes $((\alpha_\sigma^2 \alpha_b^{m-2})^{1/m}+c)^m(2^{-1}m!(b^{\ell})^{m-2}(\sigma^{\ell})^2).$ Since $(\alpha_\sigma^2 \alpha_b^{m-2})^{1/m} \le \alpha_{\sigma} \alpha_b$ for all $m \in \{2,3,\cdots\}$, we obtain
    \begin{equation*}
        \frac{1}{n}\sum_{j=1}^n\mbE[|E_{k,j}^{i,h}-c E_{k,j}^{i,l}|^m\mid\mcF_{k,j-1}] \le \frac{m!}{2}((\alpha_\sigma \alpha_b+c)b^{\ell})^{m-2}((\alpha_\sigma \alpha_b+c)\sigma^{\ell})^2.
    \end{equation*}
    Hence, the statement of the theorem holds with $\sigma = \sigma^{\ell} (\alpha_\sigma \alpha_b+c)$ and $b = (\alpha_\sigma \alpha_b+c)b^{\ell}.$ {If $\sigma^h \le \sigma^{\ell}$ or $b^h \le b^{\ell}$, we set $\check{\sigma}^h = 2\sigma^{\ell}$ and $\check{b}^h = 2b^{\ell}$. With this choice, \eqref{eq:stochstic-noise-assumption} continues to hold with $\check{\sigma}^h$ and $\check{b}^h$ in place of $\sigma^h$ and $b^h$, and the proof follows from the same steps as above.} %with Assumption \ref{assum:martingale} remains valid with $\check{\sigma}^h$ and $\check{b}^h$ instead of $\sigma^h$ and $b^h$, and the proof follows from the same steps as above. }

\end{proof}

Now let us prove that the function estimate error from BFAS is bounded by $\mcO((\Delta^q_k)^2)$ {with probability one after sufficiently large $k$}, aligning with the outcome stated in Theorem \ref{thm:asfinitedelta2}. %{This result matches the guaranteed estimation accuracy given by $\mbP\{|\Etilde_k^i|\geq c_f(\Delta_k^q)^2)\} \le \alpha_k$ for any given $c_f > 0$, where $\alpha_k$ increases gradually, driven by a  logarithmically increasing sequence $\lambda_k$.} %This finding {allows us to achieve getting more accurate local model by } %This finding not only enables us to attain the stochastic fully linear model (See Definition \ref{defn:fullylinear}) but also leads to the crucial observation that $\Delta^h_k$ converges to $0$ almost surely as $k$ tends to infinity. 

%It is necessary to show that the stochastic error in BFMC from BFAS is bounded by $\mcO(\Delta_k^2)$, which is the same result of Theorem \ref{thm:asfinitedelta2}.

\begin{lemma} \label{lem:mfmc-asfinite}
Let Assumption \ref{assum:martingale} holds and $\BFX_k^i$ for $i \in \{0,1,2,\dots,p,s\}$ be the design points generated by Algorithm~\ref{alg:TRO-MFDF} at iteration $k$. Let $\Ftilde(\BFX^i_k) = f(\BFX^i_k) + \Etilde_k^i$ be the HF function estimate obtained from Algorithm \ref{alg:BFAS} with $\Delta_k = \Delta_k^q$ for $q \in \{h,\ell\}$. Then, given $c_f>0$,   
\begin{equation}
    \mbP\{|\Etilde_k^i|\geq c_f(\Delta^q_k)^2 \text{ i.o.}\}=0.\label{eq:f-est-diff}
\end{equation}
\end{lemma}
\begin{proof}
    Let $\omega \in \Omega$. Firstly, if the function estimate from BFAS was obtained by CMC, we know from Theorem \ref{thm:asfinitedelta2} and Borel-Cantelli’s first lemma for martingales \cite{1991wil} that the statement of the lemma is satisfied. Now, we assume that the function estimate $\Ftilde(\BFX_k^i{(\omega)})$ is obtained using BFMC, implying that $$|\Etilde_k^i(\omega)| = |\Ebar_k^{i,h}(N_k^{i}(\omega)) - C_k(\omega) \Ebar_k^{i,l}(N_k^{i}(\omega)) + C_k(\omega) \Ebar_k^{i,l}(V_k^i(\omega))|,$$ where $\Ebar_k^{i,q}(N_k^{i}{(\omega)}) = {N(\BFX_k^i{(\omega)})}^{-1}\sum_{j=1}^{N(\BFX_k^i{(\omega)})} E^{i,q}_{k,j}{(\omega)}$ for $q \in \{h,\ell\}$, $N_k^i{(\omega)} = N(\BFX_k^i{(\omega)})$, and $V_k^i{(\omega)} = V(\BFX_k^i{(\omega)})$]. To simplify notation, we will omit $\omega$ from this point forward. Then we have 
    \begin{equation} \label{eq:etilde}    
    \begin{split}
        \mbP\{|\Etilde_k^i| \geq c_f(\Delta^q_k)^2|C_k=c\}\leq \mbP\{|\Ebar_k^{i,h}(N_k^{i}) & - c \Ebar_k^{i,l}(N_k^{i})|\geq \frac{c_f}{2}(\Delta^q_k)^2{|C_k=c}\} \\ 
        &+\mbP\{|c \Ebar_k^{i,l}(V_k^i)|\geq \frac{c_f}{2}(\Delta^q_k)^2{|C_k=c}\}.
    \end{split}
    \end{equation}
    %Note that $C_k$ is already determined before Step \ref{ASBF:bfmc-test} in Algorithm \ref{alg:BFAS}. 
    We know from Step \ref{ASBF:set-n} in Algorithm \ref{alg:BFAS} that $N_k^i$ and $V_k^i$ are greater than or equal to $\sigma_0^2 \lambda_k \kappa^{-2}(\Delta_k^q)^{-4}$.
    We also know from Lemma \ref{lem:mfmc-supexp} that Assumption \ref{assum:martingale} holds for $E_{k,0}^{i,h}-c E_{k,0}^{i,l}$, implying that Theorem \ref{thm:asfinitedelta2} also applies to $E_{k,j}^{i,h}- c E_{k,j}^{i,l}$ for any $j \in \mbN$. Hence, the right-hand side of \eqref{eq:etilde} is summable, from which we obtain $\mbP\{|\Etilde_k^i| \geq c_f(\Delta_{k}^q)^2\}$ is also summable based on $\mbP\{|\Etilde_k^i| \geq c_f(\Delta_{k}^q)^2\} = \mbE[\mbP\{|\Etilde_k^i| \geq c_f(\Delta_{k}^q)^2|C_k=c\}]$.
    As a result, the statement of the theorem holds with Borel-Cantelli’s first lemma for martingales \cite{1991wil}.

\end{proof}

Next, we demonstrate that as $k$ goes to infinity, both TR radii converge to zero almost surely. Despite the main framework of our proof differing trivially from the one presented in \cite{ha2023jsim}, we opt to provide a comprehensive proof to facilitate understanding in Appendix \ref{apdx:proofdeltaconverge}. 

\begin{lemma} \label{lem:deltaconverge}
    Let Assumptions \ref{assum:fn}-\ref{assum:hessian-norm} hold. Then, $$\Delta^h_k \xrightarrow[]{w.p.1} 0 \text{ and }\Delta^{\ell}_k \xrightarrow[]{w.p.1} 0 \text{ as } k \rightarrow \infty.$$
\end{lemma}

\begin{proof}
    See Appendix \ref{apdx:proofdeltaconverge}.
\end{proof}

Relying on Lemma \ref{lem:deltaconverge}, we show through Lemma \ref{lem:Gk-asconverge} that the gradient of the model for the HF function converges to { the } true gradient almost surely. It is worth highlighting that the local model for the HF function is not constructed at every iteration, as sometimes the local model for the LF function can discover a better solution.

%Moreover, we prove that 

\begin{lemma} \label{lem:Gk-asconverge}
    Let Assumptions \ref{assum:fn}-\ref{assum:hessian-norm} hold. Let $\{k_j\}$ be the subsequence such that $I_{k_j}^h  =$ True. Then, $$\|\nabla M^h_{k_j}(\BFX_{k_j}^0) - \nabla f^{{h}}(\BFX^0_{k_j}) \| \xrightarrow[]{w.p.1} 0 \text{ as } j \rightarrow \infty.$$
\end{lemma}
\begin{proof}
    %We first note that $\{k_j\}$ is infinite subsequence 
    We know from Lemma \ref{lem:mfmc-asfinite} that given $c_f>0$, there exists sufficiently large $J$ such that $|\Etilde_{k_j}^i| < c_f(\Delta^h_{k_j})^2$ for any $i \in \{0,1,\cdots,p,s\}$ and $j>J$. Then from Lemma \ref{lem:stoch-interp}, we have,
    \begin{equation} \label{eq:gradient-gap}
    \begin{split}
        \| \nabla M^h_{k_j}(\BFX_{k_j}^0) - \nabla f^{{h}}(\BFX_{k_j}^0)\| &\le \kappa_{eg1}\Delta_{k_j}^h + \kappa_{eg2}\frac{\sqrt{\sum_{i=1}^p(\Etilde_{k_j}^i-\Etilde_{k_j}^0)^2}}{\Delta_{k_j}^h}\\
        &\le \kappa_{eg1}\Delta_{k_j}^h + \kappa_{eg2}\frac{{\sum_{i=1}^p}|\Etilde_{k_j}^i - \Etilde_{k_j}^0|}{\Delta_{k_j}^h} \\
        &\le (\kappa_{eg1}+2{p}\kappa_{eg2}c_f) \Delta_{k_j}^h.        
    \end{split}
    \end{equation}
    Given that Lemma \ref{lem:deltaconverge} ensures $\Delta_{k_j}^h$ converges to 0 w.p.1, the statement of the theorem holds.

\end{proof}

In the following lemma, we demonstrate that after a sufficient number of iterations, if the TR for the HF function is relatively smaller than the model gradient, the iteration is successful with probability one. 
%{The proof trivially follows from Lemma 4.4 with the adaptive sampling rule (A-0) in \cite{ha2023}.}

\begin{lemma} \label{lem:successful-iter}
Let Assumptions \ref{assum:fn}-\ref{assum:hessian-norm} hold. Then there exists $c_d>0$ such that 
\begin{equation*}
    \mbP\left\{ \left( \Delta^h_{k}\le c_d \|\nabla M^h_{k}(\BFX_{k}^0)\| \right) \bigcap \left(\rhohat_k<\eta\right)  \bigcap \left(I^h_k \text{ is True} \right) \text{ i.o.} \right\}=0.
\end{equation*}
\end{lemma}
\begin{proof}
    {See Appendix \ref{apdx:proof-suc-iter}.}
    %We first note that for any $k \in \mbN$, when the minimizer of the low fidelity local model in Algorithm \ref{alg:TRO-LFDF} is accepted as a next iterate, $I_k^h$ is already False. Otherwise, the HF local model is constructed in Algorithm \ref{alg:TRO-MFDF}. Then the rest of the proof trivially follows from  Lemma 4.4 with the adaptive sampling rule (A-0) in \cite{ha2023}.

\end{proof}

Given Lemma~\ref{lem:Gk-asconverge} { and~\ref{lem:successful-iter}, next result } suggests that %in cases where the true gradient is greater than zero, 
if the TR radius for the HF function is comparatively smaller than the true gradient { norm}, the candidate solution is accepted and the TR is expanded. This ensures that the TR for the HF function will not converge to zero before the true gradient does. 

{
\begin{lemma}\label{thm:delta_epsilon}
    Let Assumptions \ref{assum:fn}-\ref{assum:hessian-norm} hold. Then there exists a constant $c_{lb}>0$ such that with probability 1 
    \begin{equation*}
        \Delta_k^h < c_{lb}\|\nabla f^h(\BFX_k^0)\| \text{ for sufficiently large $k$ } \Rightarrow \text{Iteration $k$ is successful}.
    \end{equation*}
\end{lemma}
\begin{proof}
Let the solver sample path $\omega$ be fixed. When $I^h_k$ is False, the iteration $k$ is already successful. Hence, let us assume $I^h_k = \text{True}$, implying that the local model $M_k^h$ must exist. We know from \eqref{eq:gradient-gap} that for sufficiently large $k$, $\|\nabla M_k^h(\BFX_k^0) - \nabla f^h(\BFX_k^0)\| \le c_{gd} \Delta_k^h,$ where $c_{gd} = \kappa_{eg1}+ 2p\kappa_{eg2} c_f$.  Define $c_{lb} = \frac{\min\{c_d,\mu\}}{1+\min\{c_d,\mu\}c_{gd}},$ where $c_d$ is as defined in Lemma \ref{lem:successful-iter}.
Then we have
\begin{equation}
\begin{split}    
    \Delta^h_k \le c_{lb}\|\nabla f(\BFX_k^0)\| & \le c_{lb}(\|\nabla M_k^h (\BFX_k^0 )\| + \|\nabla f^h(\BFX_k^0) - \nabla M_k^h(\BFX_k^0)\|),
\end{split}
\end{equation}
from which we obtain
\begin{equation}
    \Delta_k^h \le \frac{c_{lb}}{1-c_{lb}c_{gd}} \|\nabla M_k^h (\BFX_k^0 )\| = \min\{c_d,\mu\} \|\nabla M_k^h (\BFX_k^0 )\|.
\end{equation}
We also know from Lemma \ref{lem:successful-iter} that when $\Delta_k^h \le c_d \|\nabla M_k^h (\BFX_k^0 )\|$, $\rhohat_k \ge \eta$ for sufficiently large $k$. As a result, when $\Delta_k^h \le c_{lb} \|\nabla f(\BFX_k^0)\|$, we have $\Delta_k^h \le \mu \|\nabla M_k^h (\BFX_k^0 )\| $ and $\rhohat_k \ge \eta$ for sufficiently large $k$, implying that the iteration $k$ is successful.

\end{proof}
}

\begin{lemma} \label{lem:lim-inf-convergence}
Let Assumptions \ref{assum:fn}-\ref{assum:hessian-norm} hold. Then 
\begin{equation} \label{eq:liminf-conv}
  \liminf \|\nabla f^h(\BFX_k) \| \xrightarrow[]{w.p.1} 0 \text{ as } k \rightarrow \infty.  
\end{equation}
\end{lemma}
\begin{proof}
    We begin by noting that when $I_k^h$ is False, iteration $k$ must be successful. Using this fact, along with Lemma \ref{lem:Gk-asconverge} and \ref{lem:successful-iter}, the proof can be completed by straightforwardly following the steps outlined in Theorem 4.6 of \cite{ha2023jsim}.

\end{proof}

We have reached a point where we can establish the almost sure convergence of ASTRO-BFDF. {The proof follows a similar approach to that of the classical TR method. For further details, see \cite{conn2000tr,katya:DFObook,Sara2018ASTRO}. We now prove Theorem 2.} 

\begin{proof}\textit{of Theorem \ref{thm:almostsureconvergence}.}
    We first need to assume that there is a subsequence that has gradients bounded away from zero for contradiction. Particularly, suppose that there exists a set, $\hat{\mcD}$, of positive measure, $\omega \in \hat{\mcD}$, $\epsilon_0>0$, and a subsequence of successful iterates, $\{t_{j}(\omega)\}$, such that 
    $\|\nabla f^h(\BFX_{t_j(\omega)}(\omega))\| > 2\epsilon_{0}$, for all $j \in \mathbb{N}$. 
    We denote $t_{j} = t_{j}(\omega)$ and suppress $\omega$ in the following statements for ease of notation.
    Due to the $\lim$-$\inf$ type of convergence just proved in Lemma \ref{lem:lim-inf-convergence}, for each $t_{j}$, there exists a first successful iteration, $\ell_{j} = \ell(t_{j}) > t_{j}$, such that, for large enough $k$, 
    \begin{equation}\label{eq:Subsequenceong1}
    \| \nabla f^h(\BFX_{k}) \| > 2\epsilon_{0}, \  \ t_{j} \leq k < \ell_{j}, 
    \end{equation}
    and
    \begin{equation}\label{eq:Subsequenceong2}
    \| \nabla f^h(\BFX_{\ell_{j}}) \| < 1.5\epsilon_{0}.
    \end{equation}
    Define $\mcA^h_{j} = \big \{ k \in \mcH: t_{j} \leq k < \ell_{j} \big \}$ and $\mcA^{\ell}_{j} = \big \{ k \in \mcL: t_{j} \leq k < \ell_{j} \big \}$, where
    \begin{equation*}
    \begin{split}
        \mcH &= \{k\in \mbN:(\rhohat_k > \eta) \cap (
        \mu\|\nabla M^h_k(\BFX_k^0) \| \ge \Delta_k^h) \cap (I_k^h \textit{ is True})\}, \\
        \mcL &= \{k\in \mbN:I_k^h \textit{ is False}\}.
    \end{split}    
    \end{equation*}
    Let $j$ be sufficiently large and $k \in \mcA^h_{j}$. We then obtain from Lemma \ref{lem:Gk-asconverge} 
    \begin{equation}\label{eq:LowerBoundGn1}
    \| \nabla M^h_k(\BFX_k) \| > \epsilon_{0}.
    \end{equation}
    Since $k$ is a successful iteration, $\hat{\rho}_{k} \geq \eta$. Furthermore, Assumption \ref{assum:fcd} and \eqref{eq:LowerBoundGn1} imply that
    \begin{equation}\label{eq:eq8}
    \begin{split}
    f^h(\BFX_{k})-f^h(\BFX_{k+1}) + \Etilde_k^0 - \Etilde_k^s %&\geq \eta  [M^h_{k}(\BFX_{k}) - M^h_{k}(\BFX_{k+1})]  \\ 
    &\geq \frac{1}{2}\eta {\kappa_{fcd} \| \nabla M_k^h(\BFX_k) \| \min \Bigg \{\frac{\| \nabla M_k^h(\BFX_k)\|}{\|\sfH^h_{k} \|},\Delta^h_{k} \Bigg \}}\\ 
    &> c_{fd} \Delta_k^h, 
    \end{split}
    \end{equation}
    where $c_{fd} = {\min\{\frac{1}{2}\eta \kappa_{fcd}, 2\zeta \gamma_2 c_{lb} \} \epsilon_0}.$
    When $k \in \mcA_j^{\ell}$, we also obtain
    %with \eqref{eq:eq8} using $\| \nabla M_k^{\ell} (\BFX_k)\| > \hat{\epsilon}$:
    \begin{equation} \label{eq:eq9}
    \begin{split}    
        f^h(\BFX_{k})-f^h(\BFX_{k+1}) + \Etilde_k^0 - \Etilde_k^s > {\zeta(\Delta_k^h)^2} & {> \zeta \gamma_2 c_{lb} \|\nabla f(\BFX_k)\| \Delta_k^h}\\
        & {> 2\zeta \gamma_2 c_{lb} \epsilon_0 \Delta_k^h \ge c_{fd}\Delta_k^h,} 
    \end{split}
    \end{equation}
    {where the second inequality is obtained from Lemma \ref{thm:delta_epsilon}, implying that $\Delta_k^h > \gamma_2 c_{lb} \|\nabla f(\BFX_k)\|$ for sufficiently large $k$ and the third inequality follows from $\|f(\BFX_k)\| > 2\epsilon_0$. }
    Since we know from Lemma \ref{lem:mfmc-asfinite} that 
    \begin{equation} \label{eq:eq10}
    \begin{split}
        |\Etilde_k^0 - \Etilde_{k}^s| < 0.5 c_{fd}\Delta_k^h \text{ for } k \in \mcA_j^h \text{ and }
        |\Etilde_k^0 - \Etilde_{k}^s| < 0.5 c_{fd}\Delta_k^{\ell} \text{ for } k \in \mcA_j^{\ell},
    \end{split}
    \end{equation}
    the sequence $\{f^h(\BFX_k)\}_{k\in\mcA_j}$ is monotone decreasing for sufficiently large $j$. From \eqref{eq:eq9}, \eqref{eq:eq10}, and the fact that { the step size cannot exceed the TR radius}, we deduce that
    \begin{equation}\label{eq:Boundonxks-wocrn}
    \begin{split}
    \| \BFX_{t_{j}} - \BFX_{\ell_{j}} \| \leq \sum_{i \in \mcA_j} \| \BFX_{i} - \BFX_{i+1} \| &\leq \sum_{i \in \mcA^h_j} \Delta^h_{i} + \sum_{i \in \mcA^{\ell}_j} \Delta^{\ell}_{i} \\
    &\leq \dfrac{2(f^h(\BFX_{t_{j}})-f^h(\BFX_{\ell_{j}}))}{c_{fd}}.
    \end{split}
    \end{equation}
    %Now let $\{k_t\}$ be the sequence of the successful iterations. Then we know from Lemma \ref{lem:} that, for $k \in \{k_t\}$, 
    Now define $\mcB_j = \{k \in \mcK : \ell_j \le k < t_{j+1}\}$, where $\mcK = \mcH \cup \mcL$. 
    Let $k \in \mcB_j$ for sufficiently large $j$. From Lemma \ref{lem:mfmc-asfinite}, \eqref{eq:frd-out}, \eqref{eq:frd-in}, and the fact that $\Delta_k^h > \Delta_k^{\ell}$ for any $k\in\mbN$, we obtain $f^h(\BFX_{k})-f^h(\BFX_{k+1}) \geq 0.5 \kappa_R(\Delta^{\ell}_k)^2,$
    implying that the sequence $\{f^h(\BFX_k)\}_{k\in\mcA_j\cup \mcB_j}$ is monotone decreasing for sufficiently large $j$.
    The boundedness of $f^h$ from below then implies that the right-hand side of \eqref{eq:Boundonxks-wocrn} converges to 0 as $j$ goes to infinity, concluding that $\lim_{j\rightarrow\infty} \| \BFX_{t_{j}} - \BFX_{\ell_{j}} \| = 0.$ Consequently, by continuity of the gradient, we obtain that $\lim_{j\rightarrow\infty} \| \nabla f^h(\BFX_{t_{j}}) - \nabla f^h(\BFX_{\ell_{j}}) \| = 0$. However, this contradicts $\| \nabla f^h(\BFX_{t_{j}}) - \nabla f^h(\BFX_{\ell_{j}}) \| > 0.5\epsilon_{0}$, obtained from (\ref{eq:Subsequenceong1}) and (\ref{eq:Subsequenceong2}). Thus, (\ref{eq:lim-conv}) must hold.

\end{proof}

\section{Numerical Experiments}
\label{sec:numerical}
We will now assess and compare ASTRO-BFDF with other simulation optimization solvers. { The experiments consider two categories of test problems: } synthetic problems and problems with discrete event simulation (DES).

Synthetic problems constitute deterministic problems with artificial Gaussian noise. { Since the analytical form of } $f^h$ { is known}, generating numerous problems that adhere to predetermined assumptions becomes relatively straightforward. However, since both the function $f^h$ and the stochastic noises are artificially generated, the performance of the solvers on these problems { may not reflect solver effectiveness on real-world problems. } In particular, when the same random number stream is used, the stochastic noises at different design points will be identical, implying that $F^h(\cdot,{\xi^h}) - f^h(\cdot)$ is a constant function given fixed ${\xi^h} \in {\Xi^h}$. This setting satisfies a stricter assumption than the one posed in this paper.
{ To provide more realistic testing, } we also evaluated the solvers on problems { involving } DES, { which mimics } real-world conditions { and generates }  multiple outputs utilized within the objective function. All experiments have been implemented using SimOpt~\cite{eckman2022simopt}.

{We evaluate the performance of each algorithm through a two-stage experimental design. In the first stage, each solver is tested on every problem instance through 20 independent optimization runs under a problem-specific simulation budget. This budget governs all aspects of the solver’s decision-making process, such as where to evaluate and how many times to sample. In the bi-fidelity setting, simulation costs vary by fidelity level; we assign a unit cost to HF evaluations (i.e., $w^h = 1$), while LF evaluations incur a fractional cost (e.g., $w^{\ell} = 0.5$). For example, querying the HF model 10 times and the LF model 20 times would consume $10 \times w^h + 20 \times w^{\ell} = 20$ units of budget. The reported budget accounts only for simulation costs; computational overhead, such as solving TR subproblems or updating surrogate models in BO, is excluded, as the problem setup (outlined in Section \ref{sec:intro}) assumes simulations to be the dominant computational cost. In the second stage, after obtaining the solution sequences $\{\BFX_k\}$ from all runs, each visited solution is evaluated using 200 HF replications. These evaluations provide unbiased objective estimates and ensure fair comparisons across solvers.}

We compare ASTRO-BFDF { with } ASTRO-DF, {Bi-fidelity Bayesian Optimization (BFBO)~\cite{do2023multi}, Bi-fidelity sample average gradient (BFSAG)~\cite{de2020bi}}, and ADAM~\cite{kingma2017adam}.
Details of the implementation { as well as hyperparameter tuning } can be found in Appendix~\ref{apdx:implementation}. In implementing the solvers, we applied CRN, which involves using the same random { input $\omega_i$ to reduce the variance of BFMC-based function estimates and the estimated function reductions between design points. That is, both HF and LF evaluations use correlated inputs $\xi^h(\omega_i)$ and $\xi^{\ell}(\omega_i)$, and the same random input $\omega_i$ is also shared across different design points. } To integrate CRN into ASTRO-BFDF, each time a local model is constructed, the sample sizes and the coefficient at the center point, obtained through BFAS, are preserved and subsequently utilized for estimating the function values at other design points. { Although incorporating CRN often requires custom coding, SimOpt provides standardized interfaces that make integration straightforward. Without such infrastructure, implementing CRN and adaptive strategies would involve significant solver-specific effort, indicating a critical dependency of the numerical experiments on SimOpt.}

\subsection{Synthetic Problems}
We { consider } four deterministic { HF functions } $f^h$ { and construct LF counterparts based on a correlation parameter $\kappa_{cor}$. The LF function $f^{\ell}(\BFx, \kappa_{cor})$ may closely resembles or differ significantly from $f^h(\BFx)$ depending on $\kappa_{cor}$. Details of these functions are provided in Appendix~\ref{apdx:function-des}. } %creating a total of 108 test problems by altering the LF function and incorporating different {additive } stochastic noises. 
%{The correlation between deterministic HF and LF functions is controlled by positive constant $\kappa_{cor}$. Specifically, the LF deterministic function is defined as $f^{\ell}(\BFx, \kappa_{cor})$, which can take a form similar to or significantly different from $f^h(\BFx)$ depending on the value of $\kappa_{cor}$. Details about four deterministic HF and LF functions can be found in Appendix~\ref{apdx:function-des}.}
We have tested the problems with three different values of $\kappa_{cor}$: 0.1, 0.5, and 0.9. 
Regarding the stochastic noises, a more complex setup is required to determine whether BFAS can enhance computational efficiency. For instance, test problems should include an instance with a high variance of the LF oracle, which can make BFMC undesirable during the optimization. Thus, we examined the configuration in which the stochastic noises for the HF and LF oracles { follow $\mathcal{N}(0,c_{sd} + 0.05X_k[0])$ with $c^h_{sd} \in \{5, 10, 15\}$, where $X_k[0]$ denotes the first coordinate of the current solution $\BFX_k$. The additional term $0.05x[0]$ ensures that noise terms vary across design points under CRN. Combining four HF functions, three LF configurations, and three noise settings for each fidelity yields a total of $4 \times 3 \times 3 \times 3 = 108$ synthetic test problems.}

\begin{figure} [htp]
\centering
\subfloat[{$\kappa_{cor}=0.1$}]{%
\resizebox*{7cm}{!}{\includegraphics{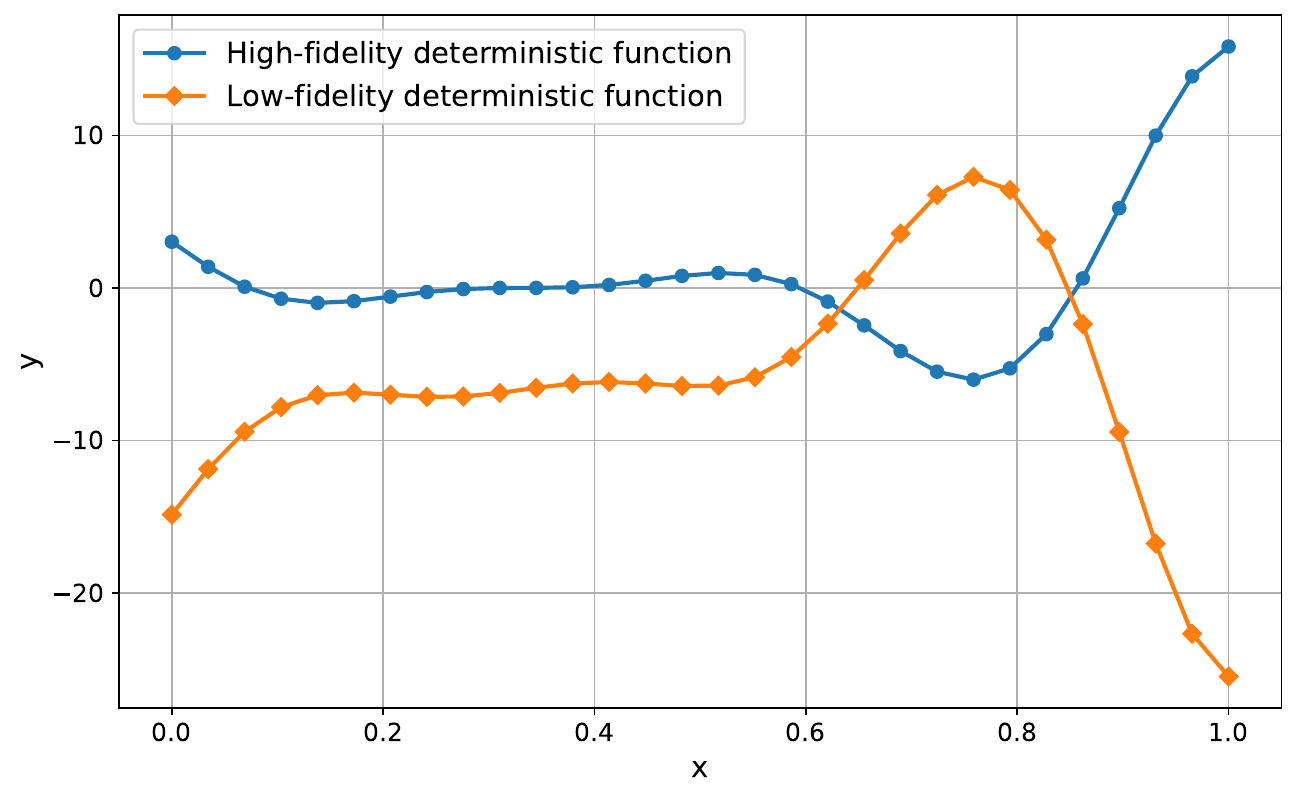}}\label{fig:for-cor-0.1}}%\hspace{2pt}
\subfloat[{$\kappa_{cor}=0.9$}]{%
\resizebox*{7cm}{!}{\includegraphics{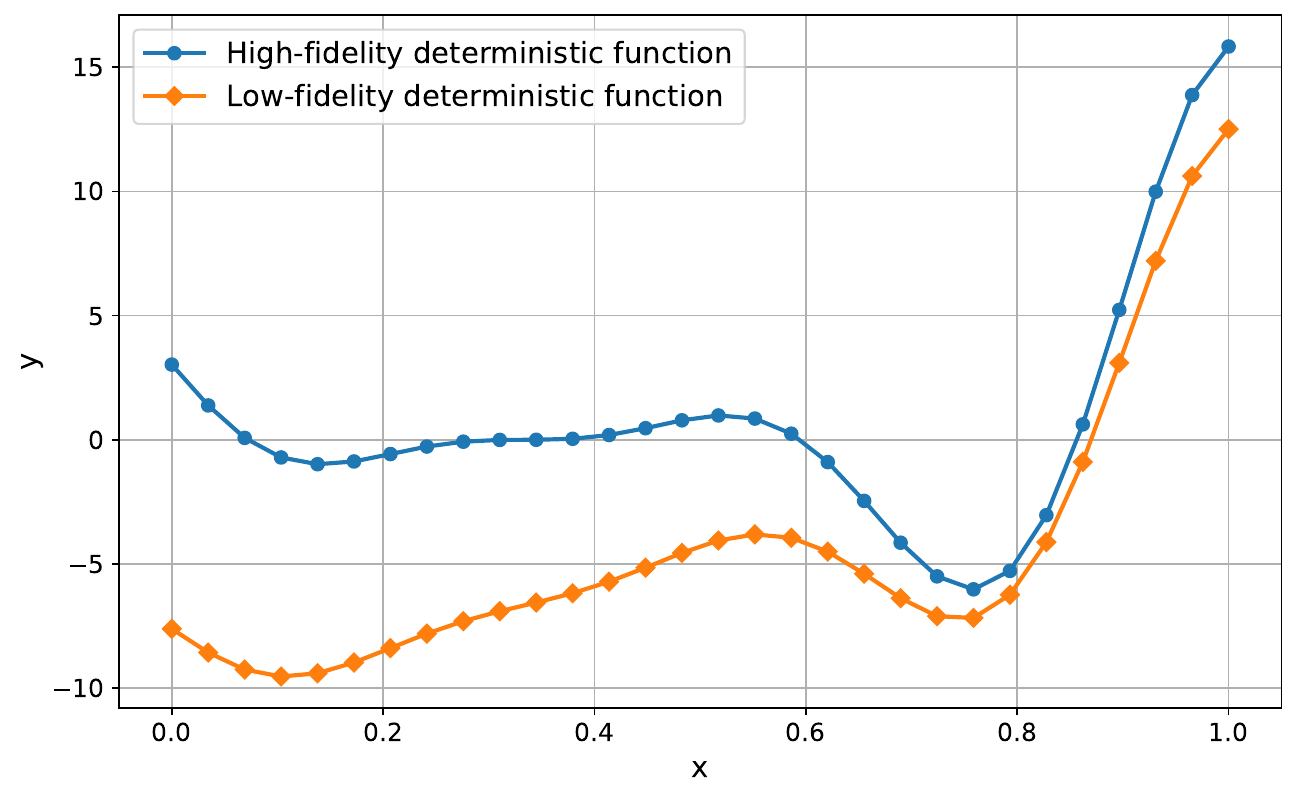}}\label{fig:for-cor-0.9}}
\caption{{Loss landscapes of Forretal functions with varying $\kappa_{cor}$.}} 
\label{fig:for-loss}
\end{figure}

\begin{figure} [htp]
\centering
\subfloat[{$\kappa_{cor}=0.1$}]{%
\resizebox*{7cm}{!}{\includegraphics{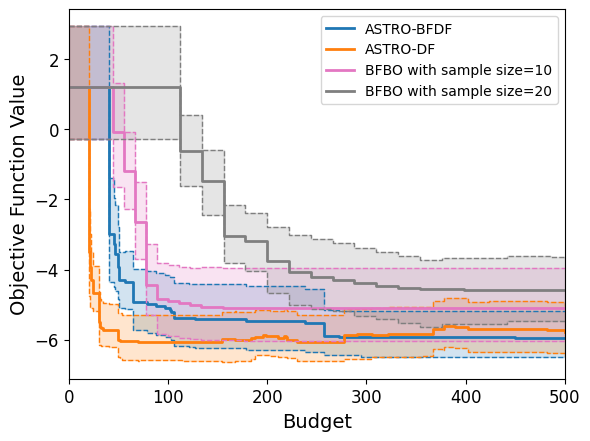}}\label{fig:for-cor-0.1}}%\hspace{2pt}
\subfloat[{$\kappa_{cor}=0.9$}]{%
\resizebox*{7cm}{!}{\includegraphics{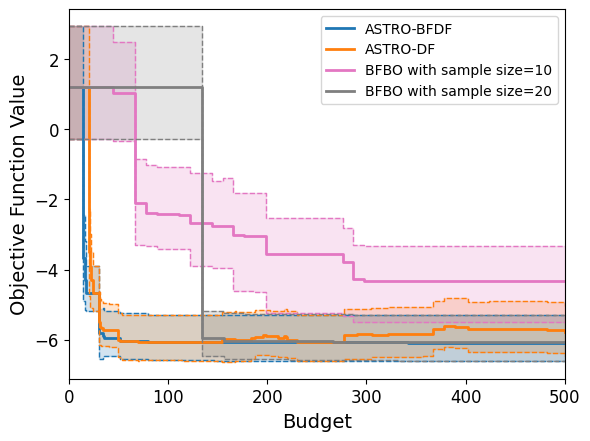}}\label{fig:for-cor-0.9}}
\caption{{Finite-time performance on Forretal functions with cost ratio 1:0.1, $c_{sd}^h=c_sd^{\ell}=15$, and (a) $\kappa_{cor} = 0.1$ and (b) $\kappa_{cor} = 0.9$.}} 
\label{fig:for-finite-performance}
\end{figure}

{We begin by presenting numerical results on synthetic problems with varying $\kappa_{cor}$. Although $\kappa_{cor}$ does not directly represent linear similarity between $f^h$ and $f^{\ell}$, it induces significant changes in LF function behaviors (see Figure~\ref{fig:for-loss}). When $\kappa_{cor} = 0.9$, the LF function provides useful information—such as the sign of the gradient—that supports HF optimization, whereas for $\kappa_{cor} = 0.1$, it provides minimal benefit. To assess the impact of $\kappa_{cor}$ on solver performance, we first tested ASTRO-BFDF, ASTRO-DF, and BFBO on the Forrester functions shown in Figure~\ref{fig:for-loss}, augmented with stochastic noise. See Figure~\ref{fig:for-finite-performance}. ASTRO-BFDF reaches near-optimal solutions faster when $\kappa_{cor} = 0.9$ compared to $\kappa_{cor} = 0.1$. Notably, ASTRO-DF converges faster than ASTRO-BFDF when $\kappa_{cor}=0.1$, but it often accepts suboptimal solutions because of stochastic noise. In contrast, ASTRO-BFDF ultimately attains better solutions by leveraging BFMC for variance reduction in function estimates. 
BFBO with a sample size of 20 fails to reach the optimum when $\kappa_{cor} = 0.1$. These results indicate that BFBO is effective only under strong correlation, as discussed in Section~\ref{sec:intro}. Lastly, when the sample size is reduced to 10 in BFBO, the function estimates are not sufficiently accurate, preventing it from achieving the optimal solution in either case. This further shows that BFBO requires highly accurate function estimates to effectively solve stochastic optimization problems, which makes it inefficient and challenging to determine appropriate sample sizes.}

{We now compare solver performance across all 108 synthetic problems using solvability profiles, a standard tool for evaluating solver effectiveness in simulation optimization. Each curve in a profile represents the fraction of problems that a solver successfully solves within a specified relative optimality gap. For example, ADAM solves about 20\% of the problems to within a 1\% optimality gap after using 30\% of the budget allocated to each problem in Figure~\ref{fig:syn-0.1}. The optimality gap is calculated based on the best solution found among all tested solvers, even when the true optimum is known, to mimic practical conditions and allow fair comparisons.}

\begin{figure} [htp]
\centering
\subfloat[{Cost ratio is 1:0.1.}]{%
\resizebox*{7cm}{!}{\includegraphics{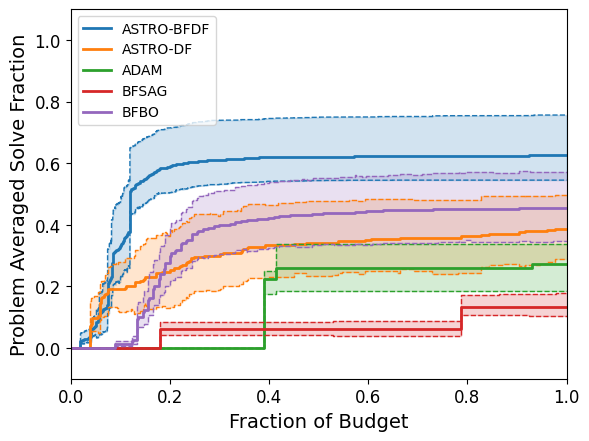}}\label{fig:syn-0.1}}%\hspace{2pt}
\subfloat[{Cost ratio is 1:1.}]{%
\resizebox*{7cm}{!}{\includegraphics{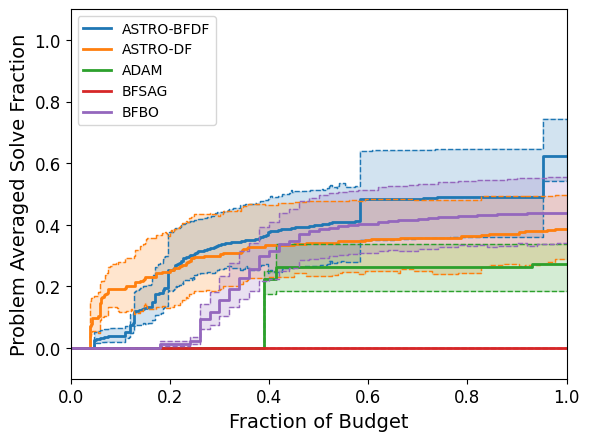}}\label{fig:syn-1}}
\caption{Fraction of 108 synthetic problems solved to { within a 1\% relative optimality gap (i.e., 0.01-optimality), } with 95\% confidence intervals from 20 runs of each algorithm. {In (a) and (b), the cost ratios are 1:0.1 and 1:1, respectively.}} 
\label{fig:solvability-profile}
\end{figure}

{In Figure~\ref{fig:solvability-profile}, } when the cost ratio of calling HF and LF oracles stands at 1:0.1, ASTRO-BFDF emerges as a standout performer, solving over {60}\% of the problems within a mere {20}\% of the budget. { Even when LF evaluations cost the same as HF evaluations (1:1), ASTRO-BFDF finds better solutions than ASTRO-DF (see Figure \ref{fig:syn-1})}. This suggests that utilizing the LF function could be beneficial for optimization, {even when its computational cost is comparable to that of the HF function. } Hence, we will next delve deeper into the specific scenarios where leveraging the LF function proves advantageous for optimization.

{Previous studies}~\cite{muller2020bfgo,song2019radial} { often use } the correlation between LF and HF functions { (or between LF and } a surrogate model) to determine whether { LF evaluations should be employed}. However, even though the LF function may exhibit a high correlation with the HF function in specific feasible regions, its usefulness can vary based on the optimization {  trajectory, } such as the initial design point. Hence, { correlation alone might not be a reliable indicator of LF usefulness. }
Instead of requiring { strong global correlation, it is often sufficient to obtain an accurate gradient approximation at the current iterate. } 
{ By incorporating an adaptive correlation constant and BFAS for variance reduction, ASTRO-BFDF achieves more accurate gradients and outperforms other solvers, even when there is significant bias between $f^h$ and $f^\ell$. To confirm this numerically, we now present results for problems where only the correlation constant $\kappa_{cor}$ varies.} 

\begin{figure} [htp]
\centering
\subfloat[Low correlation ($\kappa_{cor}=0.1$)]{%
\resizebox*{7cm}{!}{\includegraphics{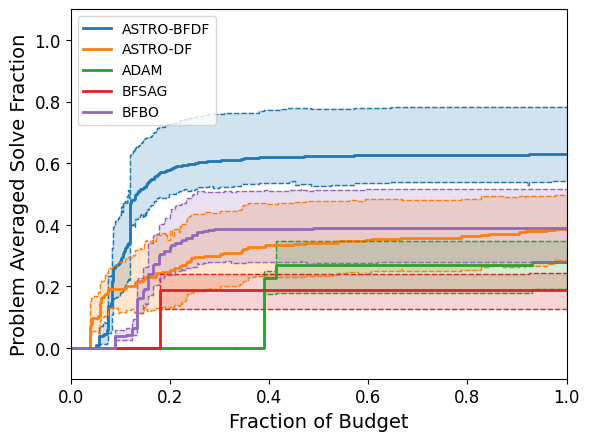}}\label{fig:cor-0.1}}%\hspace{2pt}
\subfloat[High correlation ($\kappa_{cor}=0.9$)]{%
\resizebox*{7cm}{!}{\includegraphics{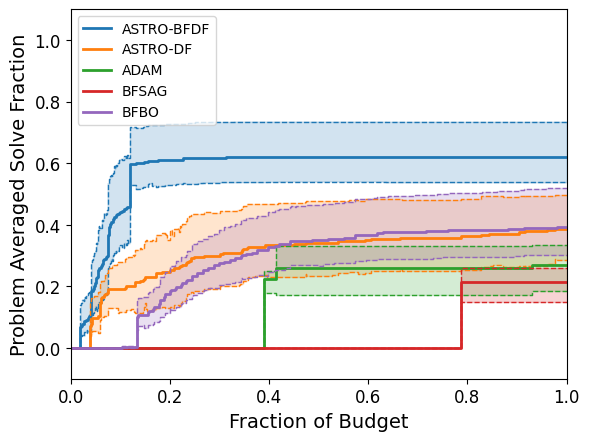}}\label{fig:cor-0.9}}
\caption{Solvability profiles of 36 problems { measured  at a 0.01 optimality gap, } with 95\% confidence intervals from 20 runs of each algorithm, { under } two different correlation settings between the LF and HF functions.} 
\label{fig:correlation-comparison}
\end{figure}

In Figure \ref{fig:correlation-comparison}, ASTRO-BFDF consistently demonstrates superiority regardless of the correlation between the HF and LF functions. {When the correlation is low (see Figure \ref{fig:cor-0.1}), ASTRO-BFDF initially converges more slowly than ASTRO-DF during the first 10\% of the budget, as it evaluates whether $M_k^{\ell}$ provides useful information. Once it prioritizes $M_k^h$, ASTRO-BFDF converges more rapidly to better solutions, benefiting from more accurate gradient approximations enabled by BFAS, compared to ASTRO-DF. }

The usefulness of the LF function in providing accurate gradient estimates can be maximized when it possesses unique structural properties, such as convexity, which { may guide the search toward } the global optimum of the non-convex HF function. In this case, BF optimization remains advantageous { even when the LF oracle has } high variance and cost. However, the opposite scenario is also possible, where the optimum of the LF function is located near a local optimum of the HF function, which { may hinder progress}. 

\begin{figure} [htp]
\centering
\includegraphics[width=0.9\columnwidth]{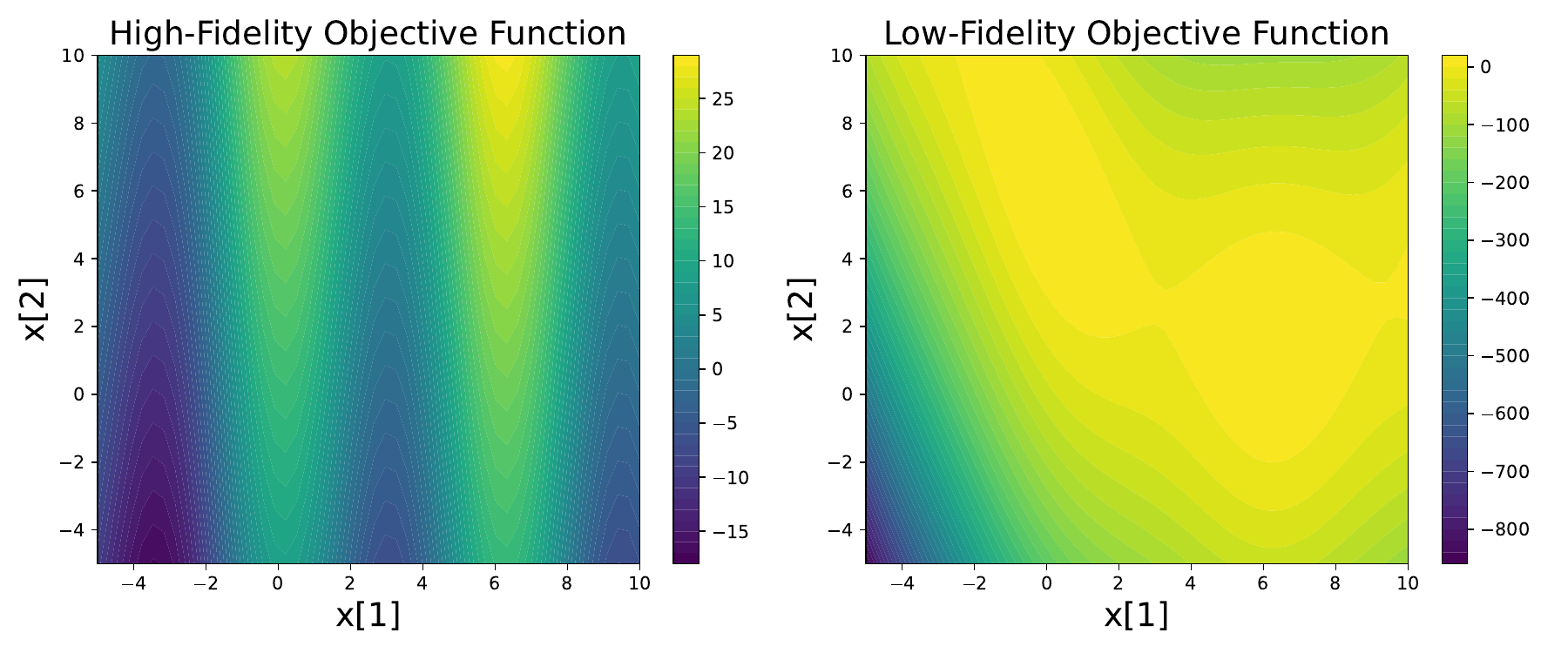}
\caption{The contour maps of the HF and LF function without stochastic noises of the BRANIN problem with {$\kappa_{cor}=0$}. } 
\label{fig:branin-objective}
\end{figure}

\begin{figure} [htp]
\centering
\subfloat[$x_0$ = (0.5,0.5)]{%
\resizebox*{7cm}{!}{\includegraphics{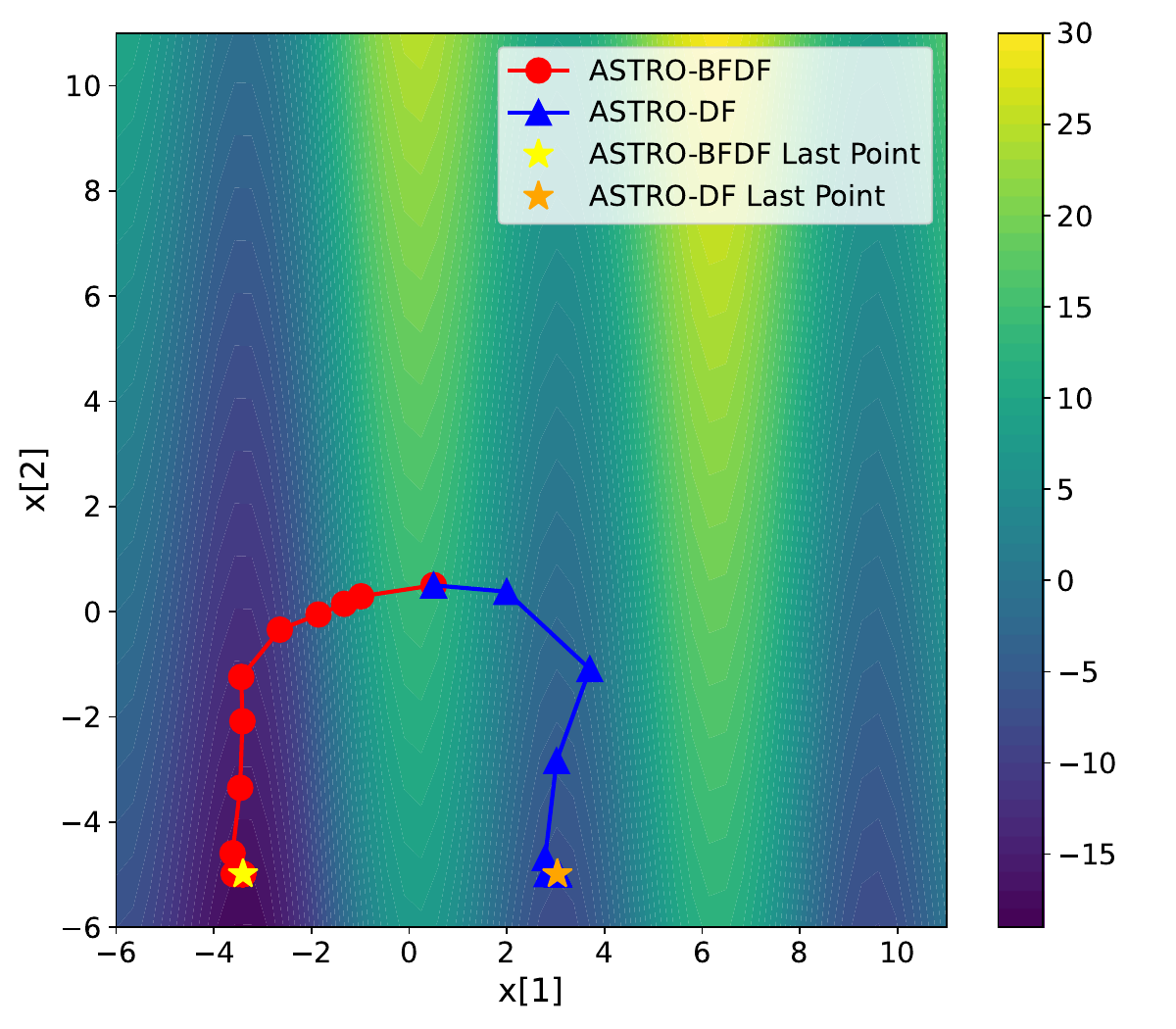}}}%\hspace{2pt}
\subfloat[$x_0$ = (6,6)]{%
\resizebox*{7cm}{!}{\includegraphics{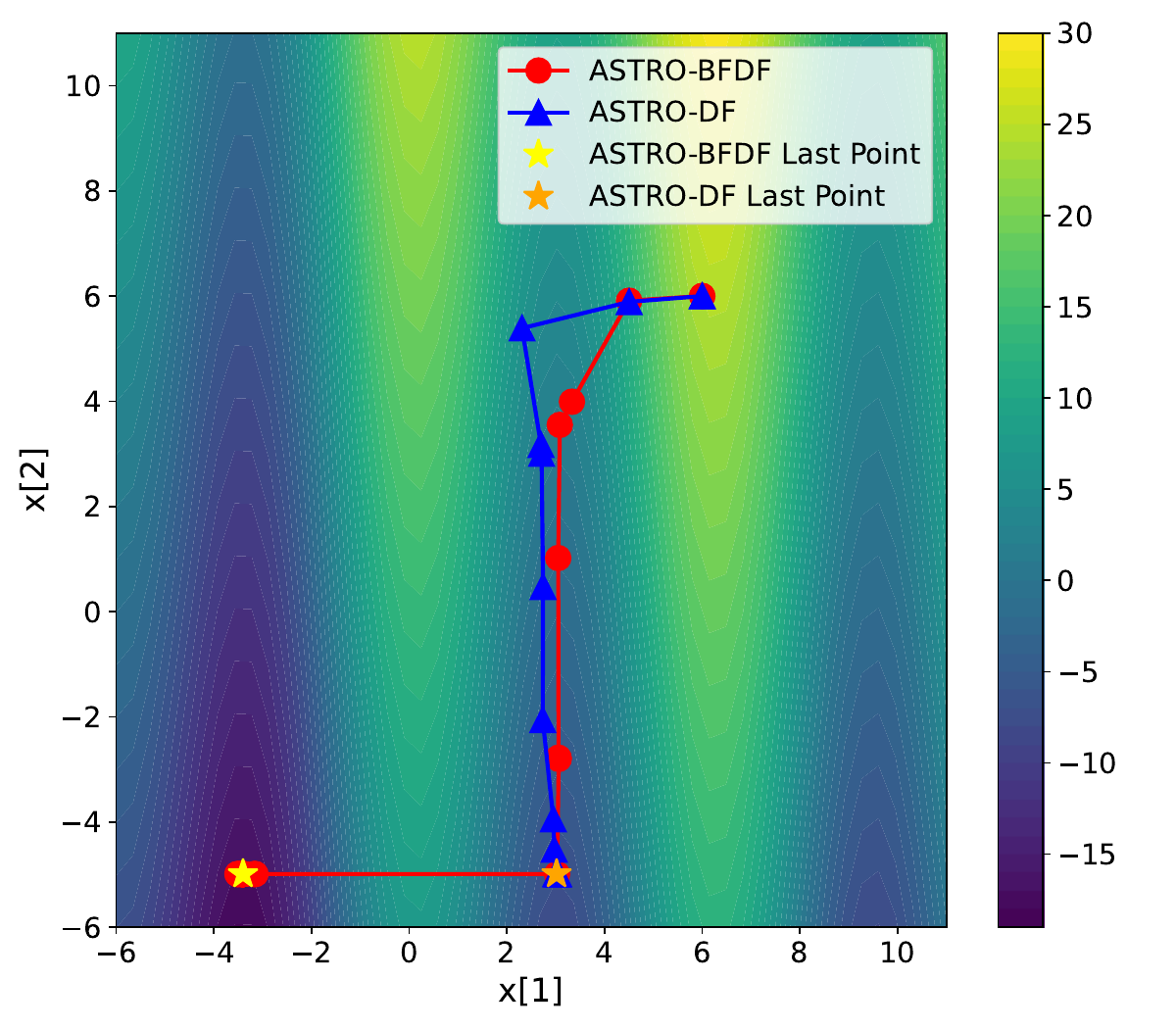}}\label{fig:traj-branin-6}}
\caption{{One sample path of $\{\BFX_k\}$ with ASTRO-DF and ASTRO-BFDF on Branin function with $\kappa_{cor}=0$ and a budget of 1000 HF oracle calls.}} 
\label{fig:trajectory-branin}
\end{figure}

The Branin function is an example for which BF optimization { can be highly effective } due to the structure of the LF function. In Figure \ref{fig:branin-objective}, {the LF function } possesses a favorable structure, {which aligns with the global optimum of the HF function. Consequently, although directly finding the global optimum of the HF function is challenging due to its non-convexity, } the solver { can leverage the LF function to guide the search toward the region near the global optimum. } {See Figure~\ref{fig:trajectory-branin}. Since the HF function is non-convex, ASTRO-DF often converges to a local optimum. In contrast, ASTRO-BFDF leverages the LF function and guides $\{\BFX_k\}$ toward the global optimum. Notably, even if $\{\BFX_k\}$ initially approaches to a local optimum, ASTRO-BFDF can escape by using locally convex structure of the LF function and ultimately reach the global optimum (see Figure \ref{fig:traj-branin-6}). } 

\subsection{Problems with DES}
In this section, we test more realistic problems using DES for the HF and LF simulation oracles. { Specifically, we consider } two problems: an M/M/1 queue problem and an inventory problem. In both cases, the DES model operates until a defined end time, denoted { by } $T$, thereby enabling the acquisition of BF results through variations in $T$. {A simulation model with a longer end time is typically considered an HF model, as longer runs tend to produce more accurate estimates of the simulation output. For example, consider the case where the objective is to minimize inventory costs over 100 days. The HF model simulates the entire 100-day inventory system, capturing long-term dynamics and providing more accurate estimates of cumulative costs. In contrast, the LF simulation model runs over a shorter horizon, such as 30 days, which reduces computational time but may yield less precise estimates of steady-state performance. }  In this setting, the cost ratio between the HF and LF models stands at $1:0.3$. A notable { feature } of this problem is that running one replication of the HF model inherently produces one replication of the LF model without incurring additional computational expenses.

\begin{figure} [htp]
\centering
\includegraphics[width=0.7\columnwidth]{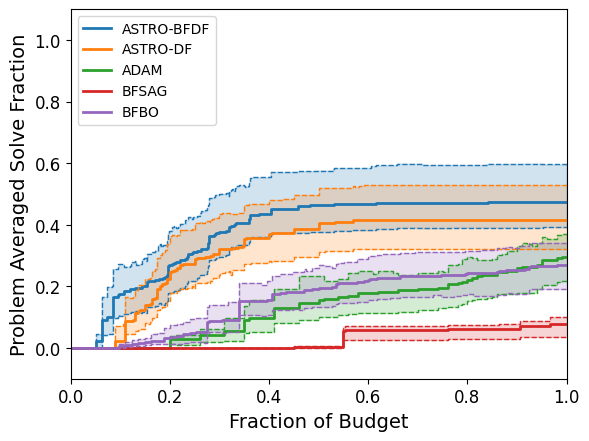}
\caption{{Solvability profiles of 25 problems with DES measured  at a 0.01 optimality gap, with 95\% confidence intervals from 20 runs of each algorithm. }} 
\label{fig:solva-des}
\end{figure}

Before delving into the details of each problem, we { present } the solvability profile { for } 25 instances (See Figure~\ref{fig:solva-des}), including 5 instances from the M/M/1 problem and 20 instances from the inventory problem. The cost ratio between the HF and LF models for both problems is $1:0.3$, indicating that the LF oracle simulates the system for $0.3T$ days. {ASTRO-BFDF demonstrates not only a faster convergence but also an enhanced ability to identify superior solutions by the end of the allocated budget.}

\subsubsection{M/M/1 Queue Problem}
We employ a model that simulates { an } M/M/1 queue, characterized by exponential distributions for both inter-arrival and service times. { Let $\lambda$ and $\mu$ denote the rate parameters for interarrival and service times, respectively. The HF M/M/1 queue model simulates 100 random arrivals into the system, generating inter-arrival and service times for each customer. Based on these stochastic realizations, it estimates the average sojourn time for each arrival, representing the time a customer spends in the system from entry to departure under a given configuration. The randomness in the arrival and service processes is captured by $\xi^h$; therefore, employing CRN implies generating the same sequence of arrivals and service times across different system configurations. The HF function is defined as $F^h(\mu,\xi^h) = 100^{-1}\sum_{i=1}^{100} s_i(\mu,\xi^h) + 0.1\mu^2$, where $s_i(\mu,\xi^h)$ denotes the sojourn time of the $i$-th customer, and the term $ 0.1\mu^2$ represents a penalty interpreted as an investment to increase the service rate $\mu$. Therefore, } the objective is to minimize the expected average sojourn time and a penalty, where $\mu$ acts as the decision variable. 

\begin{figure} [htp]
\centering
\subfloat[Independent sampling and $\lambda = 1$]{%
\resizebox*{7cm}{!}{\includegraphics{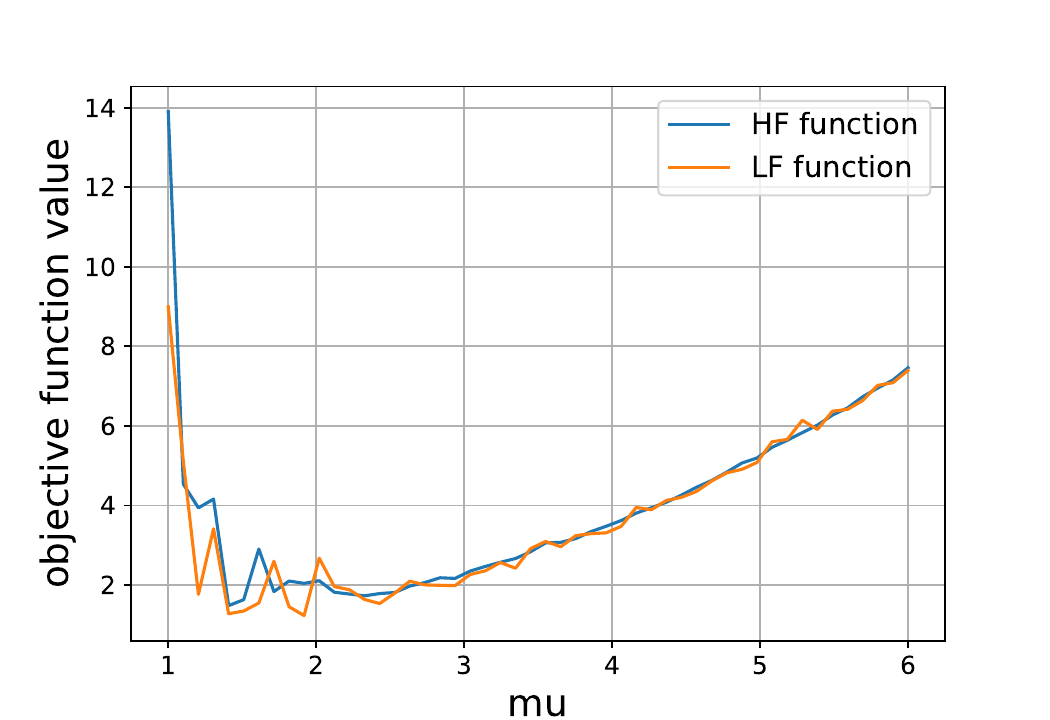}}\label{fig:mm1wocrn}}%\hspace{2pt}
\subfloat[CRN and $\lambda = 1$]{%
\resizebox*{7cm}{!}{\includegraphics{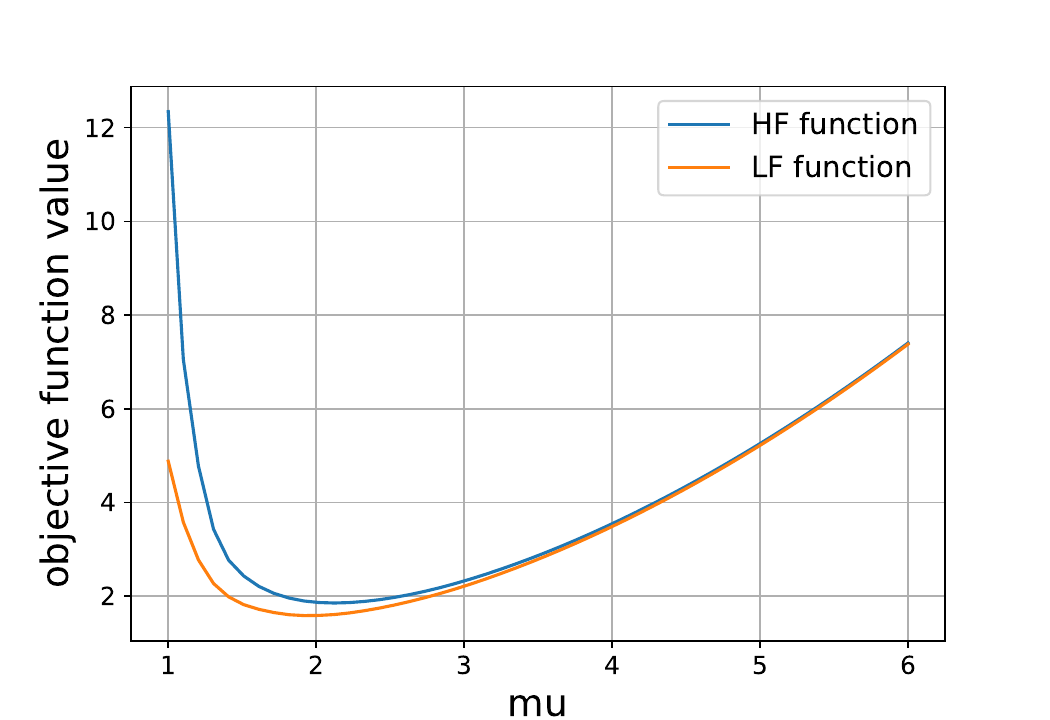}}\label{fig:mm1crn}}\\ 
\subfloat[Independent sampling and $\lambda = 5$]{%
\resizebox*{7cm}{!}{\includegraphics{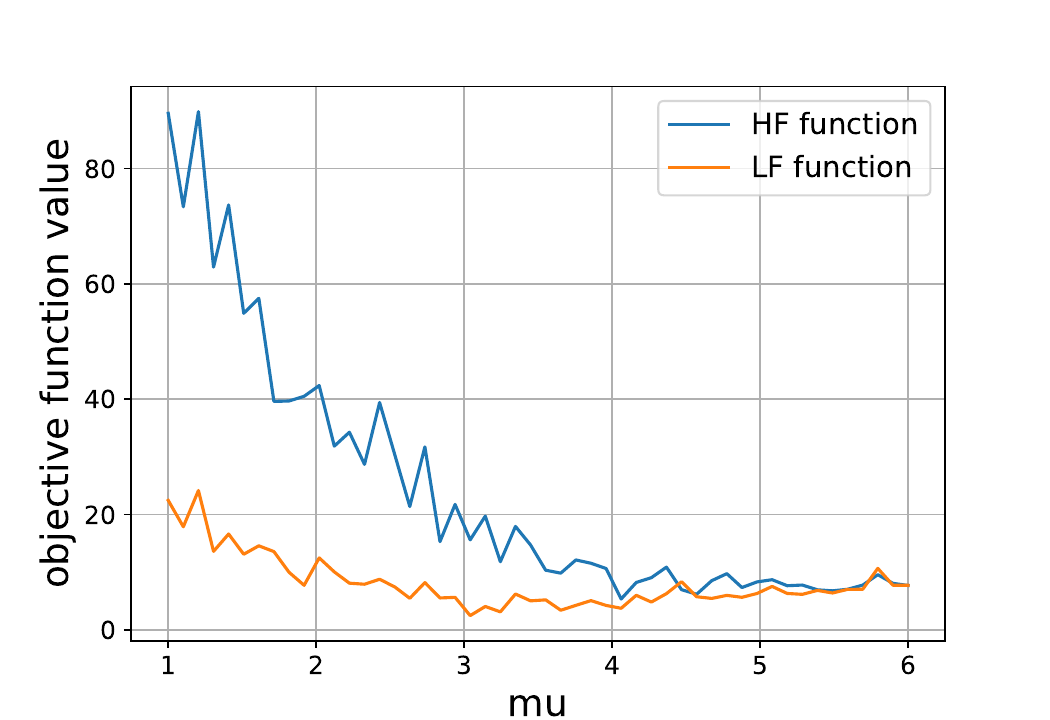}}\label{fig:mm1wocrn10}}%\hspace{2pt}
\subfloat[CRN and $\lambda = 5$]{%
\resizebox*{7cm}{!}{\includegraphics{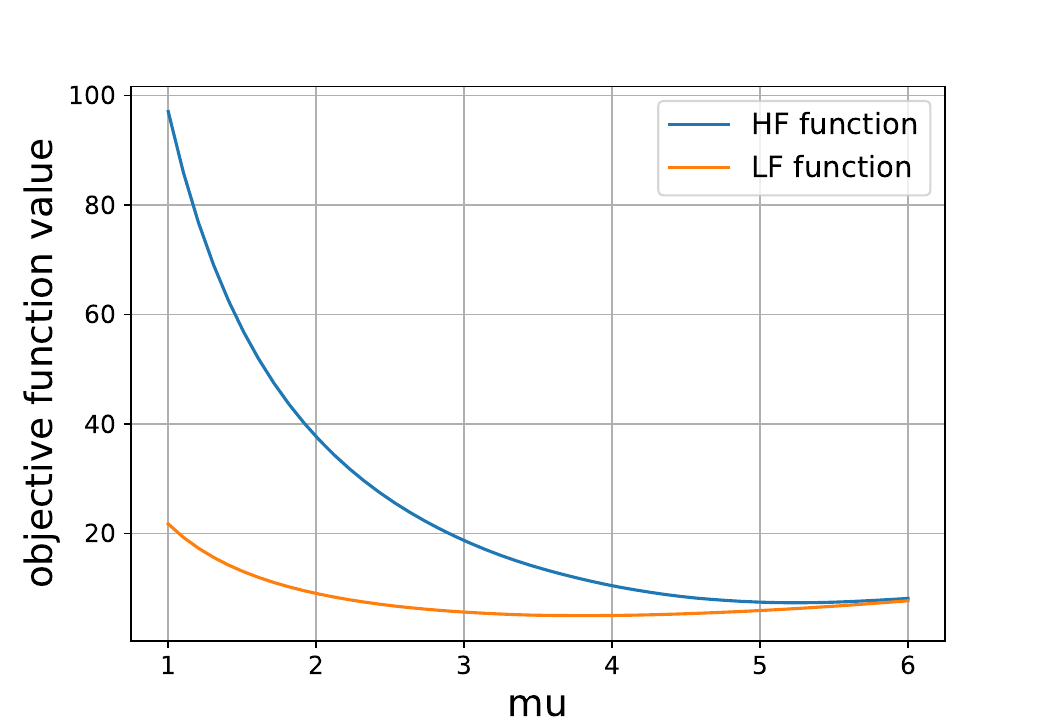}}\label{fig:mm1crn10}}
\caption{{Objective function trajectories } of the M/M/1 problem with and without CRN. When employing CRN, the { LF and HF } functions exhibit smoothness.}
\label{fig:mm1}
\end{figure}

One important characteristic of the problem is that $F^h(\cdot,\xi)$ and $F^{\ell}(\cdot,\xi)$ are smooth functions for any { $\xi^h\in\Xi^h$, as illustrated in } Figure~\ref{fig:mm1}. The difficulty in addressing this problem with the LF function { is } that its gradient is relatively smaller compared to that of the HF function, which becomes clearer when $\lambda$ { is large } (see Figure~\ref{fig:mm1wocrn10} and~\ref{fig:mm1crn10}). {In traditional TR methods, } the criterion for successful iterations hinges on comparing the reduction in { the } model and { the } function estimates.  %{as shown in Step~\ref{HF:delta-l-update} of Algorithm~\ref{alg:TRO-MFDF}}. 
Therefore, when the { small } gradient of the LF function { leads to only a slight reduction in the model value}, a candidate point can be accepted, albeit leading away from the optimal solution. { To address this issue, } the additional condition is { introduced},  {$\Ftilde_k(\BFX_k^0)-\Ftilde_k(\BFX_k^{s,\ell}) \ge \eta\zeta(\Delta_k^h)^2$}, for a successful iteration in Algorithm~\ref{alg:TRO-LFDF} (see Remark~\ref{remark:suff-red-lf}).

\begin{figure} [htp]
\centering
\subfloat[$\lambda = 1$]{%
\resizebox*{7cm}{!}{\includegraphics{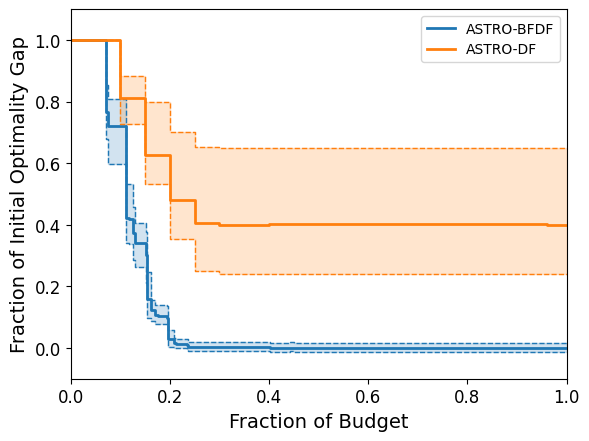}}\label{fig:MM1-1}}%\hspace{2pt}
\subfloat[$\lambda = 5$]{%
\resizebox*{7cm}{!}{\includegraphics{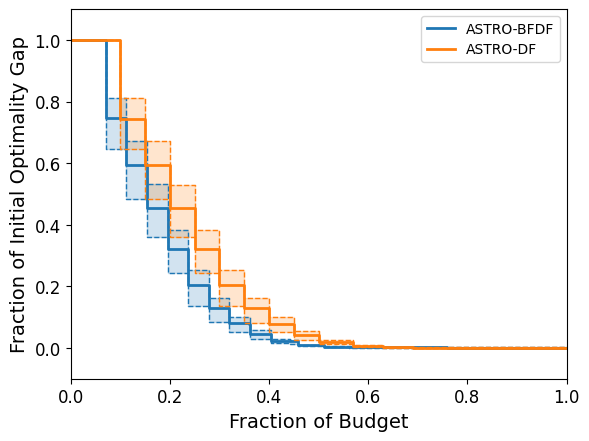}}\label{fig:MM1-5}}
\caption{Fraction of the optimality gap with 95\% confidence intervals from 20 runs of each algorithm. } 
\label{fig:MM1-opt-gap}
\end{figure}

We conducted { experiments } on 5 instances of the M/M/1 problem { by } varying $\lambda$ across the range $\{1, 2, \dots, 5\}$, { as shown in } Figure~\ref{fig:solva-des}. Figure~\ref{fig:MM1-opt-gap} illustrates the optimization progress for two scenarios: one where $\lambda = 1$ and another where $\lambda = 5$. 
In the scenario where $\lambda = 1$, as the incumbents approach the optimal solution, it becomes essential for the TR to contract appropriately to achieve an accurate gradient approximation. While contracting the TR, ASTRO-DF exhausts its budget entirely, which explains its slower convergence in Figure~\ref{fig:MM1-1}. In contrast, ASTRO-BFDF is capable of rapidly identifying a near-optimal solution.
The primary reason is that the gradient of the { LF-based } local model is inherently small, enabling us to sustain successful iterations before the TR initiates sequential contraction. Conversely, when $\lambda = 5$, the gradient of the local model for the LF function becomes { extremely small, } prompting a { termination } of LF function utilization after just a few iterations.  
{ As a result, } in Figure~\ref{fig:MM1-5}, the optimization trajectory of ASTRO-BFDF appears similar to that of ASTRO-DF, but ASTRO-BFDF demonstrates slightly faster convergence due to the variance-reduced function estimates provided by BFAS.

\subsubsection{$(s,S)$ Inventory Problem}
We { now } consider { an } $(s,S)$ inventory model. At each time step $t$, the demand $D_t$, which follows { an } exponential distribution with { mean } $\mu_D$, is generated. At the end of each time step, the inventory level is calculated { as $I^e_t = I^s_t - D_t + O_t$, where $I^e_t$ denotes the end-of-day inventory, $I^s_t$ is the starting inventory on hand, and $O_t$ is the amount of inventory that was ordered earlier and arrives on that day. } If { $I^e_t$ } is below $s$, an order is placed { to restore the inventory up to $S$}. Lead times follow { a } Poisson distribution with mean $\mu_L$ time steps. The { objective } is to find the best $s$ and $S$ for minimizing the average costs, which is composed of backorder costs, order costs, and holding costs. {Backorder costs refer to penalties incurred when demand cannot be immediately satisfied; order costs consist of a fixed cost applied whenever an order is placed, regardless of its size, and a variable cost based on the order quantity; and holding costs represent the cost of maintaining positive inventory at the end of each day. As in the M/M/1 problem, randomness in demands and lead times is captured by $\xi^h$; thus, employing CRNs ensures that the same sequence of demands and lead times is used across different system configurations. } 

\begin{figure} [htp]
\centering
\subfloat[Independent sampling]{%
\resizebox*{7cm}{!}{\includegraphics{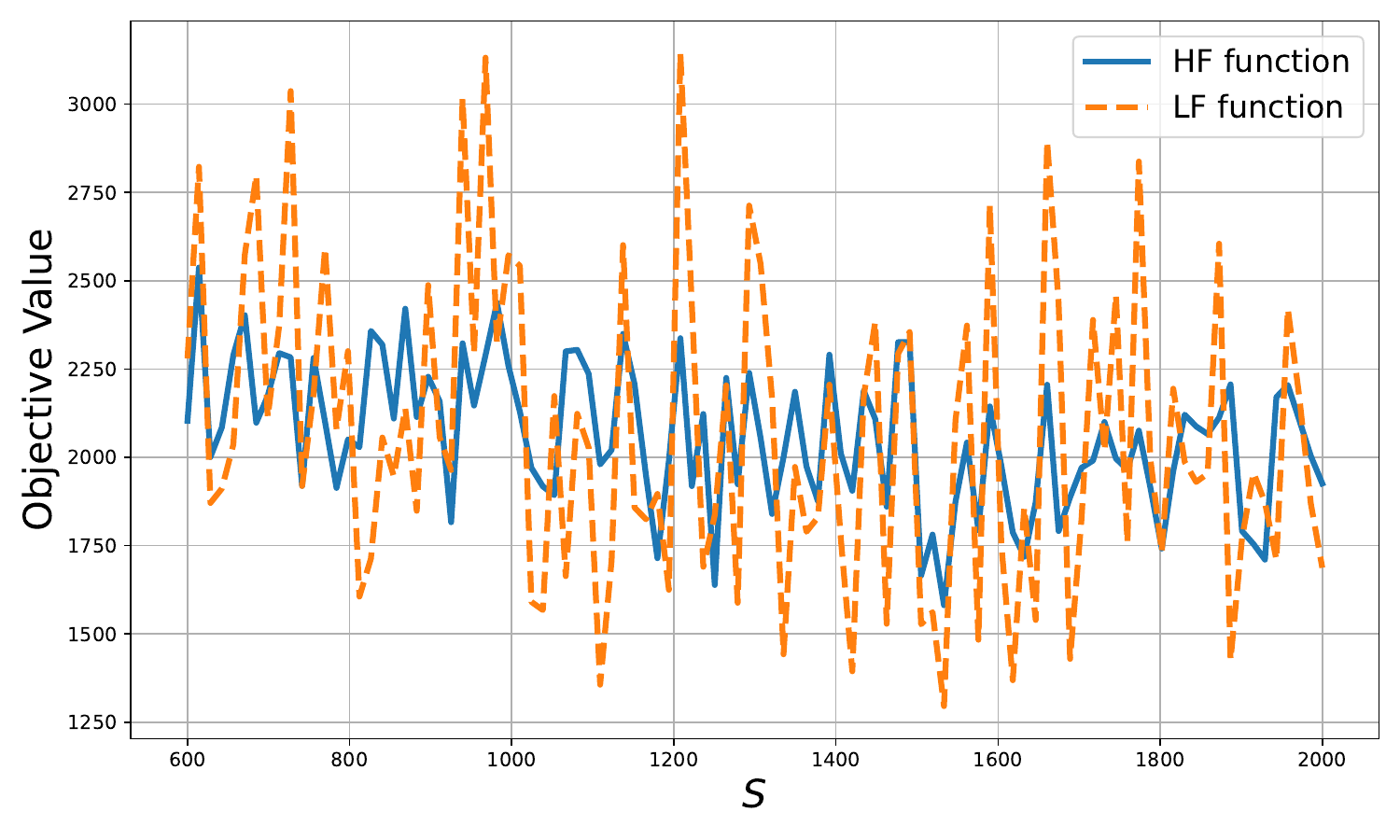}}\label{fig:sscontwocrn}}%\hspace{2pt}
\subfloat[CRNs]{%
\resizebox*{7cm}{!}{\includegraphics{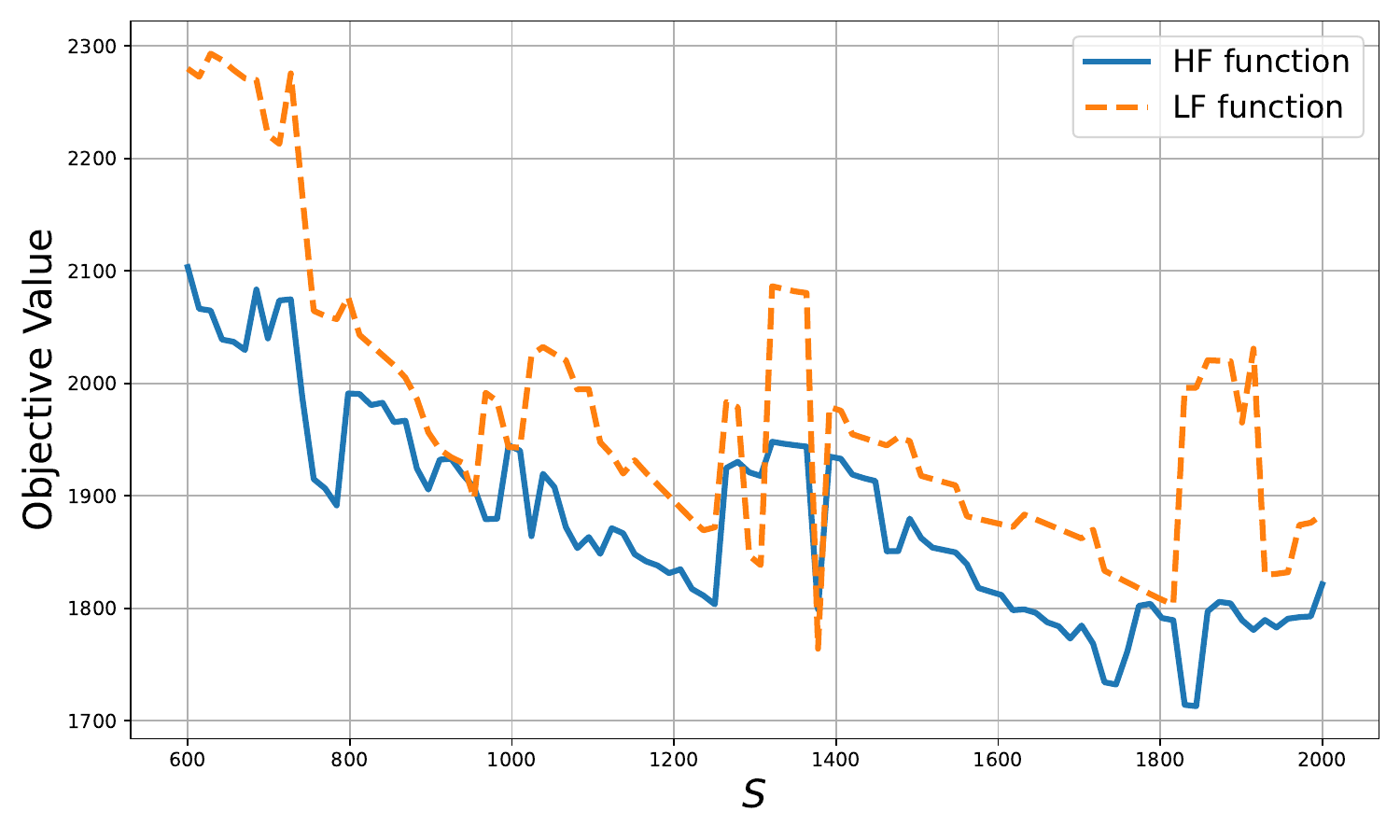}}\label{fig:sscontcrn}}
\caption{{Objective function trajectories of the inventory problem with and without CRN, evaluated over varying $S$ with fixed $s = 500$. The demand and lead time means are $\mu_D = 400$ and $\mu_L = 3$, respectively. }}
\label{fig:sscont-objective}
\end{figure}

\begin{figure} [htp]
\centering
\subfloat[$\{\BFX_k\}$]{%
\resizebox*{6cm}{!}{\includegraphics{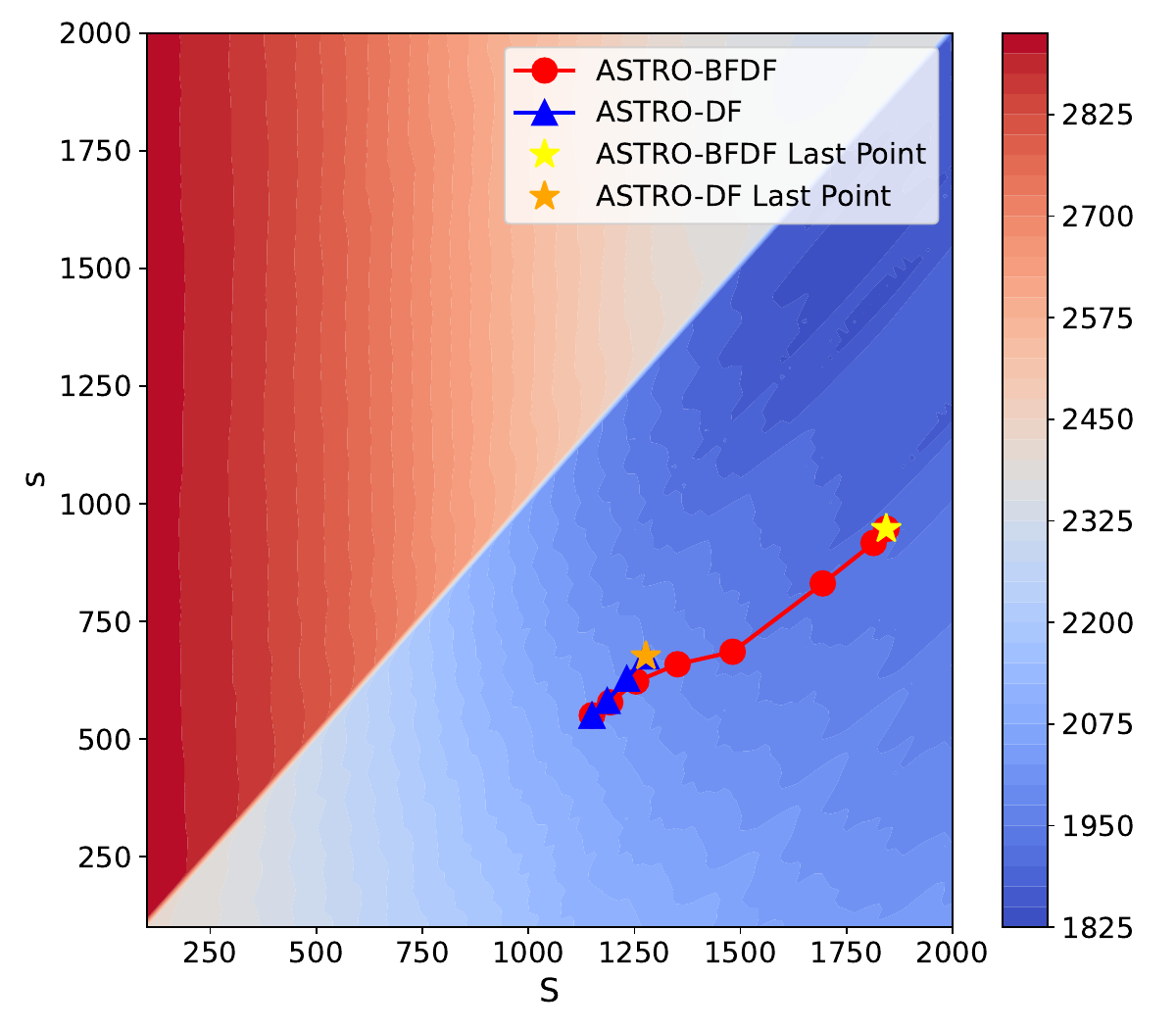}}\label{fig:traj-sscont}}%\hspace{2pt}
\subfloat[Objective Function Value]{%
\resizebox*{7cm}{!}{\includegraphics{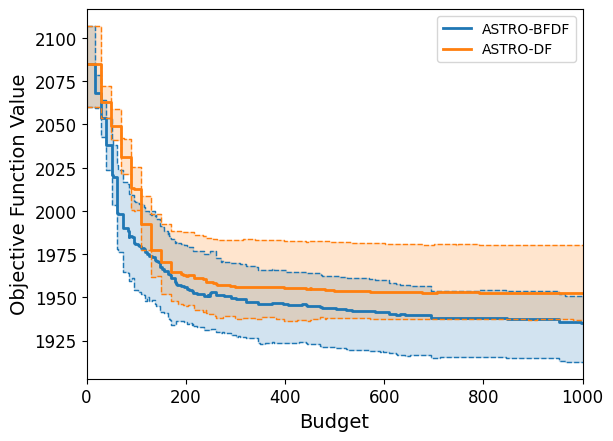}}\label{fig:400-3objective}}

\caption{{(a) illustrates a sample path of $\{\BFX_k\}$ for ASTRO-DF and ASTRO-BFDF on the inventory problem with $\mu_D=400$, $\mu_L=3$, and a budget of 1000 HF evaluations, where contours represent objective function estimates based on 50 samples. (b) presents the average objective function value over 20 runs under the same setting.}}
\label{fig:sscont}
\end{figure}

This problem is significantly more challenging than the M/M/1 problem due to the inherent non-smoothness (see Figure~\ref{fig:sscont-objective}). Therefore, it is highly probable that the majority of incumbent sequences converges to local optima, regardless of the solvers used. {See Figure~\ref{fig:sscont}. Although ASTRO-BFDF may still converge to local optima (see Figure~\ref{fig:sscont2}), it is generally more effective than ASTRO-DF at navigating the rugged objective landscape, often identifying better solutions by leveraging BF models. In particular, ASTRO-BFDF can explore a broader region of the search space by maintaining a larger $\Delta_k^h$, increasing the likelihood of escaping poor local optima in practice. } We { tested } 20 instances of the { inventory } problem with parameters { combinations } $\mu_D = \{25,50,100,200,400\}$ and $\mu_L = \{1,3,6,9\}$. In most cases, ASTRO-BFDF converged faster than ASTRO-DF, and sometimes { achieved } better solutions, { as evidenced by its higher fraction of solved problems at the budget limit (see Figure~\ref{fig:solva-des}).}

\begin{figure} [htp]
\centering
\subfloat[$\{\BFX_k\}$]{%
\resizebox*{6cm}{!}{\includegraphics{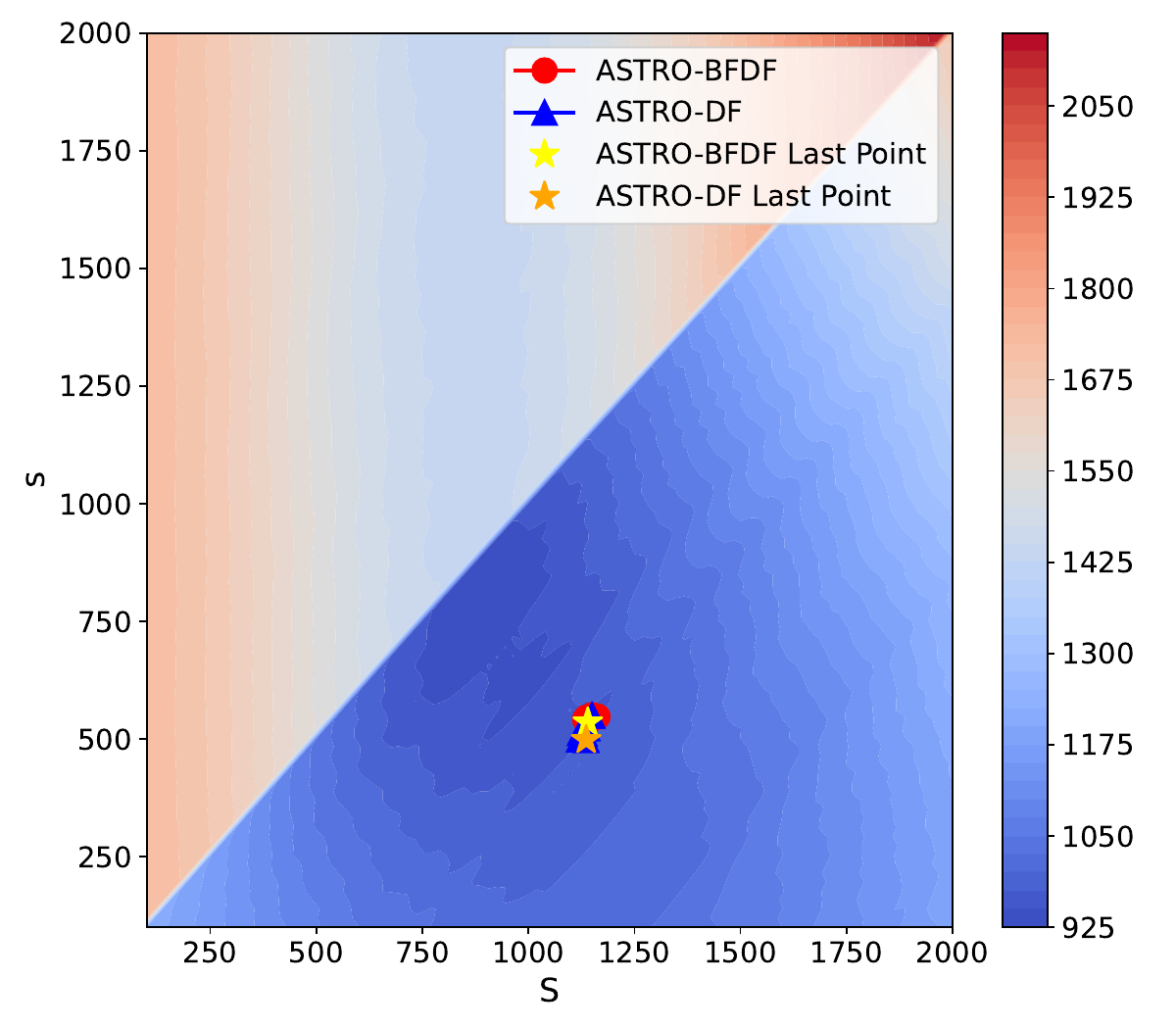}}\label{fig:traj-sscont}}%\hspace{2pt}
\subfloat[Objective Function Value]{%
\resizebox*{7cm}{!}{\includegraphics{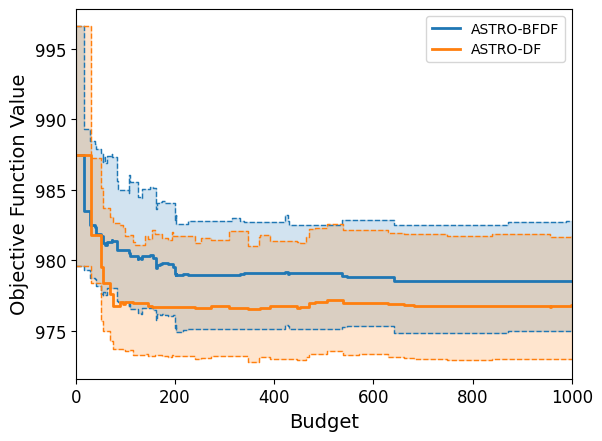}}\label{fig:200-3objective}}

\caption{{Same experimental setup as in Fig~\ref{fig:sscont} but with different parameters with $\mu_D = 200$ and $\mu_L = 3$. Both ASTRO-DF and ASTRO-BFDF failed to escape convergence to poor solution and the average difference in objective values between the two solvers is less than 5.}}
\label{fig:sscont2}
\end{figure}

\section{Conclusion}
This paper introduces ASTRO-BFDF, a novel stochastic TR algorithm tailored for addressing bi-fidelity simulation optimization. ASTRO-BFDF { is characterized by } two key features: First, it { dynamically } utilizes bi-fidelity Monte Carlo or crude Monte Carlo, { adaptively } adjusting sample sizes for both fidelity oracles within BFAS. 
%This ensures accurate estimation of function values, with the accuracy required for both function and gradient determined by the progress of optimization. 
Second, it guides incumbents toward the neighborhood of { a } stationary point of the HF function by utilizing the LF function. 
These two features { enable } faster convergence { and improved computational efficiency, } as demonstrated on synthetic and { DES-based problems}. 
We also demonstrate the asymptotic behavior of the incumbents generated by ASTRO-BFDF, which converges to the stationary point almost surely.

\section*{Acknowledgements}
This work was authored by the National Renewable Energy Laboratory, operated by Alliance for Sustainable Energy, LLC, for the U.S. Department of Energy (DOE) under Contract No. DE-AC36-08GO28308. Funding for the algorithmic development and numerical experiment work was provided by Laboratory Directed Research and Development investments. The views expressed in the article do not necessarily represent the views of the DOE or the U.S. Government. The U.S. Government retains and the publisher, by accepting the article for publication, acknowledges that the U.S. Government retains a nonexclusive, paid-up, irrevocable, worldwide license to publish or reproduce the published form of this work, or allow others to do so, for U.S. Government purposes.

\appendix

{
\section{Proof of Theorem \ref{thm:asfinitedelta2}}
\label{apdx:proofasfinitedelta2}
\begin{proof}
    Let us first set $\Ebar^{i,q}_{k}(N_k^i) = N(\BFX_k^i)^{-1}\sum^{N(\BFX_k^i)}_{j=1} E^{i,q}_{k,j}$ and $o_k = (\sigma^q_0)^2 \kappa^{-2} (\Delta_k^q)^{-4}$, implying $N(\BFX_k^i) \ge \lambda_k o_k$. Then we have
    \begin{equation*}
    \begin{split}
        \mbP\left\{|\Ebar^{i,q}_{k}(N_k^i)|>c_f(\Delta_k^q)^2 \middle| \mcF_{k-1}\right\} &
        \leq\mbP\left\{\sup_{n\geq\lambda_k o_k}|\Ebar^{i,q}_{k}(n)|>c_f (\Delta_k^q)^2\ \ \middle\vert\  \mcF_{k-1}\right\} \nonumber\\
        &\leq\sum_{n\geq\lambda_k o_k}\mbP\left\{\left|\frac{1}{n}\sum_{j=1}^{n}E^{i,q}_{k,j} \right| > c_f (\Delta_k^q)^2\ \middle\vert\ \mcF_{k-1}\right\}\nonumber\\   
        &\leq\sum_{n\geq\lambda_k o_k} 2 \exp{\left(-n\frac{c_f^2 (\Delta_k^q)^4}{(2 c_f (\Delta_k^q)^2 b^q +2(\sigma^q)^2)}\right)}\nonumber\\
        &= \sum_{n\geq 0}2 \exp\left(-c_k(\lambda_k o_k+n)\right)= 2 \frac{\exp\left(-\lambda_kc_k\nu_k\right)}{1-\exp(-c_k)},
    \end{split}        
    \end{equation*}
     where $c_k = \frac{c_f^2 (\Delta_k^q)^4}{2(c_f \Delta_{\max}^2 b^q + (\sigma^q)^2)}$. The third inequality is obtained using Assumption~\ref{assum:martingale} and Lemma~\ref{lem:bernstein}.
    Given that $\lambda_k=\lambda_0 (\log k)^{1+\epsilon_\lambda}$ and $c_k o_k = \mcO(1),$ there exists some $\varepsilon >0$ such that for sufficiently large $k$, the following holds.
    \begin{equation}
        \mbP\left\{|\Ebar_k^{i,q}(N_k^i)|\geq c_f (\Delta_k^q)^{2}\right\}=\mbE\left[\mbP\left\{|\Ebar_k^{i,q}(N_k^i)|\geq c_f (\Delta_k^q)^{2}\ \middle\vert\ \mcF_{k-1}\right\}\right]\leq k^{-1-\varepsilon}.\label{eq:prob-good-estimate}
    \end{equation}
    Since the right-hand side of~\eqref{eq:prob-good-estimate} is summable in $k$, the theorem holds.
\end{proof}
}

\section{Proof of Lemma \ref{lem:deltaconverge}} \label{apdx:proofdeltaconverge}
\begin{proof}
    Let us begin by noting that we have established from Step \ref{HF:delta-l-update} in Algorithm \ref{alg:TRO-MFDF} and Step \ref{LF:delta-h-update} in Algorithm \ref{alg:TRO-LFDF} that $\Delta_k^h \ge \Delta_k^{\ell}$ almost surely for any $k \in \mbN$. Hence, if $\Delta_k^h$ converges to zero almost surely, so does $\Delta_k^{\ell}$.
    Let us define the following index sets,
    \begin{equation*}
    \begin{split}
        \mcH &= \{k\in \mbN:(\rhohat_k > \eta) \cap (
        \mu\|\nabla M^h_k(\BFX_k^0) \| \ge \Delta_k^h) \cap (I_k^h \textit{ is True})\}, \\
        \mcL &= \{k\in \mbN:I_k^h \textit{ is False}\}.
    \end{split}    
    \end{equation*}
    %If $\mcH$ is finite, $\Delta_k^h$ converges to zero due to infinite shrinkage within unsuccessful iterations, making the statement of the theorem hold trivially. 
    From Assumption \ref{assum:fcd}, we have, for any $k \in \mcH$,
    \begin{equation} \label{eq:frd-out}
    \begin{split}
    \Ftilde(\BFX^0_{k}) - \Ftilde(\BFX^s_{k}) &\ge {\Ftilde(\BFX^0_{k}) - \Ftilde(\BFX^{s,h}_{k})} \ge \eta  [M^h_{k}(\BFX^0_{k}) - M^h_{k}(\BFX_{k}^{{s,h}})]  \\ 
    &\geq \frac{1}{2}\eta {\kappa_{fcd} \| \nabla M_k^h(\BFX^0_k) \| \min \Bigg \{\frac{\| \nabla M_k^h(\BFX^0_k)\|}{\|\sfH^h_{k} \|},\Delta^h_{k} \Bigg \}} > \kappa_R (\Delta_k^h)^2, 
    \end{split}
    \end{equation}
    where $\kappa_R= \min\{\eta\kappa_{fcd}{(2\mu)}^{-1}\min \{{(\mu\kappa^h_\sfH)}^{-1},1\},\zeta\}$.
    Note that \eqref{eq:frd-out} holds regardless of whether $\BFX_k^s$ comes from minimizing $M_k^{\ell}$ or $M_k^h$. We also obtain from Step \ref{LF:success-ratio} in Algorithm \ref{alg:TRO-LFDF} that, for any $k \in \mcL$,
    \begin{equation}
    \label{eq:frd-in}
     \Ftilde(\BFX^0_{k}) - \Ftilde(\BFX^s_{k})
    \ge \zeta(\Delta_k^h)^2 \ge \kappa_R(\Delta_k^h)^2.   
    \end{equation}
    Hence, for any $k \in \mcK= \mcH \cup \mcL$,
    \begin{equation*}
    \begin{split}
    \kappa_R\sum_{\substack{k\in\mcK }}(\Delta_k^h)^2 &\le \sum_{\substack{k\in\mcK }} (f^h(\BFX_{k}) - f^h(\BFX_{k+1}) + \Etilde_k^0-\Etilde_k^s) \le f^h(\BFx_0) - f^h_* +\sum_{k=0}^{\infty}|\Etilde_k^0-\Etilde_k^s|,
    \end{split}
    \end{equation*}
    where $f^h_*$ is the optimal value of $f^h$.
     We note that $\mcH$ and $\mcL$ are disjoint sets and for any $k \not\in \mcK$, $\Delta_{k+1}^h = \gamma_2 \Delta^h_k$.
    Let $\mcK = \{k_1,k_2,\dots\}$, $k_0 = -1,$ and $\Delta^h_{-1}=\Delta^h_0/\gamma_2$. Then from the fact that $\Delta^h_k\le \gamma_1\gamma_2^{k-k_i-1}\Delta^h_{k_i}$ for $k=k_i+1,\dots,k_{i+1}$ and each $i$, we obtain 
    \begin{equation*}   
    \sum_{k=k_i+1}^{k_{i+1}}(\Delta_k^h)^2 \le \gamma_1^2(\Delta^h_{k_i})^2\sum_{k=k_i+1}^{k_{i+1}}\gamma_2^{2(k-k_i-1)} 
    \le \gamma_1^2(\Delta_{k_i}^h)^2\sum_{k=0}^{\infty}\gamma_2^{2k} = \frac{\gamma_1^2}{1-\gamma_2^2}(\Delta_{k_i}^h)^2.
    \end{equation*}
    By Lemma \ref{lem:mfmc-asfinite} and the fact that $\Delta_k^{\ell} \le \Delta_k^h$, there must exist a sufficiently large $K_\Delta$ such that $|\Etilde_k^0-\Etilde_k^s| < c_{\Delta}(\Delta^h_k)^2$ for any given $c_{\Delta} > 0$ and any $k \ge K_\Delta$. Then, we have 
    \begin{align*}
    \sum_{k=0}^{\infty}(\Delta^h_k)^2  &\le \frac{\gamma_1^2}{1-\gamma_2^2}\sum_{i=0}^{\infty}(\Delta_{k_i}^h)^2< \frac{\gamma_1^2}{1-\gamma_2^2}\left(\frac{(\Delta^h_0)^2}{\gamma_2^2}+\frac{f^h(\BFx_0)-f^h_*+E'_{0,\infty}}{\kappa_R}\right)\\
    &<\frac{\gamma_1^2}{1-\gamma_2^2}\left(\frac{(\Delta^h_0)^2}{\gamma_2^2}+\frac{f^h(\BFx_0)-f^h_*+E'_{0,K_\Delta-1}+E'_{K_\Delta,\infty}}{\kappa_R}\right),
    \end{align*}
    where $E'_{i,j} = \sum_{k=i}^{j}|\Etilde_k^0-\Etilde_k^s|$. 
    Then we get from $E'_{K_\Delta,\infty} < c_{\Delta}\sum_{K_\Delta}^\infty(\Delta^h_k)^2$ that
    \begin{align*} 
    \sum_{k=K_\Delta}^{\infty}(\Delta^h_k)^2 &< \frac{\gamma_1^2}{1-\gamma_2^2}\left(\frac{(\Delta^h_0)^2}{\gamma_2^2}+\frac{f^h(\BFx_0)-f^h_*+E'_{0,K_\Delta-1}}{\kappa_R}\right)\left(1-\frac{\gamma_1^2}{1-\gamma_2^2}\frac{c_\Delta}{\kappa_R}\right)^{-1}
    \end{align*}
   Therefore, $\Delta^h_k \xrightarrow[]{w.p.1} 0\text{ as } k \rightarrow \infty$ and the statement of the theorem holds. 
\end{proof}

{
\section{Proof of Theorem \ref{lem:successful-iter}}
\label{apdx:proof-suc-iter}
The proof trivially follows from  Lemma 4.4 with the adaptive sampling rule (A-0) in \cite{ha2023}.
\begin{proof}
    Let $\omega \in \Omega$ and we will omit $\omega$ to simplify notation. We first note that for any $k \in \mbN$, when the minimizer of the LF-based local model in Algorithm \ref{alg:TRO-LFDF} is accepted as a next iterate, $I_k^h$ is already False. Otherwise, the HF-based local model is constructed in Algorithm \ref{alg:TRO-MFDF}. Then we have
    \begin{equation} \label{eq:apdx-model}
        M_k^h(\BFX_k^{s,h}) = \Ftilde(\BFX_k^0) + \nabla M_k^h(\BFX_k^0)^\intercal \BFS_k + \frac{1}{2} \BFS_k^\intercal \sfH^h_k \BFS_k,
    \end{equation}
    where $\BFS_k$ is the step size. We also know from Taylor's theorem that
    \begin{equation} \label{eq:apdx-fn}
         \Ftilde(\BFX_k^{s,h}) = f^h(\BFX^{0}_k) + \nabla f^h(\BFX^{0}_k)^\intercal \BFS_k + \int_0^1 \left(\nabla f^h(\BFX_k^{0}+t\BFS_k)- \nabla f^h(\BFX_k^{0})\right)^\intercal\BFS_k \mathrm{d}t +  \Etilde_k^{s}.
    \end{equation}
    By subtracting \eqref{eq:apdx-model} by \eqref{eq:apdx-fn}, we can obtain, for sufficiently large $k$,    
    \begin{align}
    |M_k^h(\BFX_k^{s,h}) - \Ftilde(\BFX_k^{s,h})| &\leq |(\nabla M_k^h(\BFX_k^{0}) -\nabla f^h(\BFX_k^{0}))^\intercal\BFS_k|+ \frac{1}{2}\left|\BFS_{k}^\intercal  \sfH^h_{k} \BFS_{k}\right| + |\Etilde_k^0-\Ebar_k^{s}|
    \nonumber\\
    & +\left|\int_0^1 \left(\nabla f(\BFX_k^{0}+t\BFS_k)- \nabla f(\BFX_k^{0})\right)^\intercal\BFS_k \mathrm{d}t\right| \nonumber\\
    & \leq (\kappa_{eg1} + 2p\kappa_{eg2}c_f) (\Delta_k^h)^2 + 2c_f (\Delta_k^h)^2 + \frac{1}{2} (\kappa_{Lg}+\kappa^h_{\sfH})(\Delta_k^h)^2,
\label{eq:bound-model-pred}\end{align}
where the last inequality follows from \eqref{eq:gradient-gap}, Lemma \ref{lem:mfmc-asfinite}, and Assumptions \ref{assum:fn} and \ref{assum:hessian-norm}. Moreover, we have from Assumption \ref{assum:fcd} and $\Delta_k^h \le c_d \|\nabla M_k^h(\BFX_k^0)\|$ that
\begin{equation}
\label{eq:model-reduction}
\begin{split}    
    M_k^h(\BFX_k^{0})-M_k^h(\BFX_k^{s,h}) &\ge \frac{1}{2}\kappa_{fcd} \| \nabla M^h_k(\BFX_k^{0})\| \min\left\{ \frac{\| \nabla M^h_k(\BFX_k^{0})\|}{\| \nabla^2 M^h_k(\BFX_k^{0})\|}, \Delta_k^h \right\}\\
    &\ge \frac{1}{2 c_d}\kappa_{fcd} \min\left\{\frac{1}{c_d\kappa^h_{\sfH}},1\right\} (\Delta_k^h)^2
\end{split}
\end{equation}
Then we obtain from \eqref{eq:bound-model-pred} and \eqref{eq:model-reduction} that the success ratio becomes
\begin{equation*}
    |1-\rhohat_k| = \frac{|M_k^h(\BFX_k^{s,h})-\Ftilde(\BFX_k^{s,h})|}{|M_k^h(\BFX_k^0)-M_k^h(\BFX_k^{s,h})|} \le \frac{(\kappa_{eg1} + 2p\kappa_{eg2}c_f + 2c_f + \frac{1}{2} (\kappa_{Lg}+\kappa^h_{\sfH}))(\Delta_k^h)^2}{(2c_d)^{-1}\kappa_{fcd} \min\{(c_d \kappa_{\sfH}^h)^{-1},1\} (\Delta_k^h)^2}.
\end{equation*}
As a result, when $c_d \le ((1-\eta)\kappa_{cfd})(2\kappa_{eg1} + 4p\kappa_{eg2}c_f + 4c_f +  \kappa_{Lg}+\kappa^h_{\sfH})^{-1},$ we have $\rhohat_k \ge \eta,$ proving the sought result.
\end{proof}
}

\section{Implementation Details} \label{apdx:implementation}
{The two most important hyperparameters are the step sizes and sample sizes for the solvers. To tune them, we tested a range of values for each solver on synthetic problems with the cost ratio $1:0.1$ and problems with DES. In this section, we present the detailed setup for the optimizers used in Section \ref{sec:numerical}, including the hyperparameter tuning of sample sizes and step sizes. We begin with ASTRO-DF and ASTRO-BFDF.
}

\subsection{{Setup for ASTRO-DF and ASTRO-BFDF}}
{ASTRO-DF and ASTRO-MFDF } used the same parameters (e.g., TR radius $\Delta_k$, success ratio $\eta_1$) where possible. 
%ADAM and Nelder-Mead have used the default setting outlined in SimOpt github \cite{simoptgithub}. 
In terms of the design set selection for the model construction, ASTRO-DF has used $2d+1$ design points with the rotated coordinate basis (See history-informed ASTRO-DF \cite{ha2023green}). In the bi-fidelity scenario, we have employed two distinct design sets ($\mcX_k$ and $\mcX_k^{\ell}$) at Step \ref{ASBFTRO:designsetselect} in Algorithm \ref{alg:TRO-MFDF} and Step \ref{ASLFTRO:designsetselect} in Algorithm \ref{alg:TRO-LFDF} respectively. $\mcX_k$ is selected to construct the local model for the HF function, implying that the computational costs for estimating the function value at $\mcX_k$ is relatively high. Hence, the design set will be selected by reusing the design points within the TR and the corresponding replications as much as possible. To achieve this, we first pick $d+1$ design points to obtain sufficiently affinely independent points by employing Algorithm 4.2 in \cite{wild2008orbit}. After that, we pick additional $d$ design points following the opposite direction to construct the quadratic interpolation model with diagonal Hessian. $\mcX_k^{\ell}$ consists of $2d+1$ design points, selected using the coordinate basis to minimize deterministic error owing to the lower cost of the LF oracle. In this scenario, the design set $\mcX_k^{\ell}$ is optimally designed for design sets of any size ranging from $d+2$ to $2d+1$ (see \cite{tom2023optimalpoised}). 

\begin{figure} [htp]
\centering
\subfloat[ASTRO-BFDF on synthetic problems]{%
\resizebox*{7cm}{!}{\includegraphics{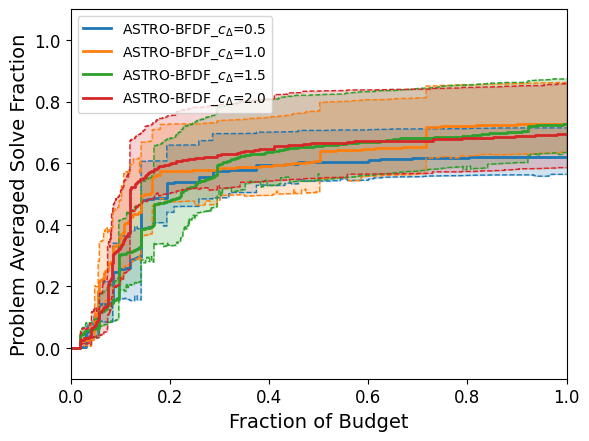}}\label{fig:syn-astrobfdf}}%\hspace{2pt}
\subfloat[ASTRO-DF on synthetic problems]{%
\resizebox*{7cm}{!}{\includegraphics{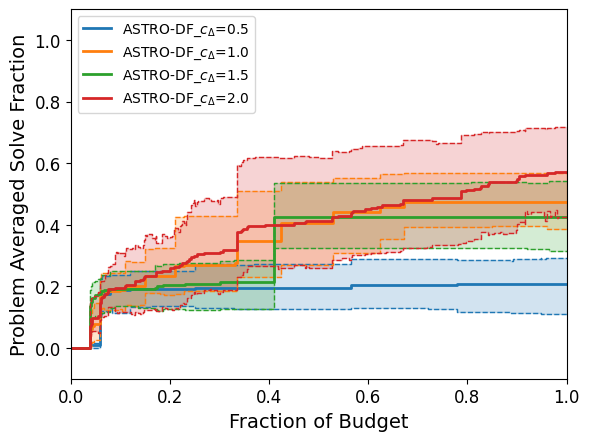}}\label{fig:syn_astrodf}}\\

\subfloat[ASTRO-BFDF on problems with DES]{%
\resizebox*{7cm}{!}{\includegraphics{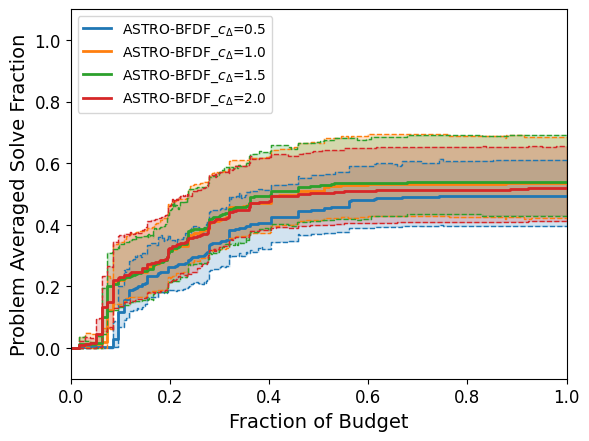}}\label{fig:des-astrobfdf}}%\hspace{2pt}
\subfloat[ASTRO-DF on problems with DES]{%
\resizebox*{7cm}{!}{\includegraphics{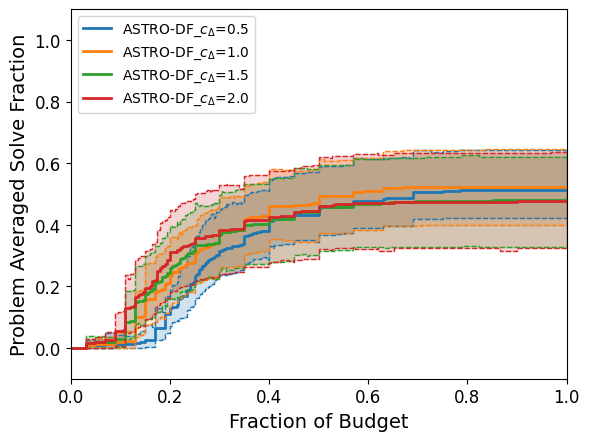}}\label{fig:des-astrodf}}
\caption{{Solvability profiles with 1\% optimality gap for hyperparameter tuning of ASTRO-MFDF and ASTRO-DF.}}
\label{fig:ht-astro-syn}
\end{figure}

{For hyperparameter tuning, since the sample size and step size controlled by trust-region radius, we focused on the initial trust-region radius, which is the primary hyperparameter affecting performance. When box constraints exist, it is natural to set the initial trust region as a portion of the feasible region. Since all synthetic problems have box constraints, we set $\Delta_0^h$ and $\Delta_k^{\ell}$ to $c_{\Delta} \times 0.1 \times \min_{i \in {1,2,\dots,d}} (u_i - l_i)$, where $c_\Delta$ is a positive constant, and $u_i$ and $l_i$ denote the upper and lower bounds for the $i$-th coordinate, respectively. We first tested $c_{\Delta} \in \{0.5,1,1.5,2\}$ on synthetic problems. See Figure \ref{fig:syn-astrobfdf} and \ref{fig:syn_astrodf}. In contrast, DES-based problems do not have box constraints. Instead, each problem can generate random solutions based on the input parameters provided by SimOpt \cite{simoptgithub}. Accordingly, we generated $1000\times d$ random solutions $\{\BFX^r\}$, where $r$ indexes each of the $1000\times d$ solutions, and set $\Delta_0^h$ and $\Delta_0^{\ell}$ to $c_{\Delta} \sum_{i=1}^d (\max_r\{X^r_i\}-\min_r\{X^r_i\})$, where $X_i^r$ denoted the $i$-th coordinate of $\BFX^r$. We tested $c_{\Delta} \in \{0.5,1.0,1.5,2\}$ on DES-based problems as well. See Figure \ref{fig:des-astrobfdf} and \ref{fig:des-astrodf}. As a result, we use $c_\Delta = 2$ for ASTRO-BFDF and ASTRO-DF in Section~\ref{sec:numerical}. }

\newpage
\subsection{{Setup for ADAM and BFSAG}}
{Since ADAM and BFSAG are designed for settings where gradient information is directly available, we approximate the gradient for the targeted problem \eqref{eq:problem} using the finite-difference method, with the perturbation constant set equal to the step size. The sample sizes $n$ is used to estimate the function value and the gradient is approximated using the function estimates. In BFSAG, the sample size for high-fidelity simulations is set to $n$, while the sample size for low-fidelity simulations is set to $2n$. For hyper-parameter tuning, we first fix $n$ as 10 and tested variety step sizes from 0.01 to 2 (Figure \ref{fig:adam_lr} and \ref{fig:bfsag_lr}). After that, we chose the best step sizes for each solver and tested various $n \in \{10,20,30\}$ (Figure \ref{fig:adam_ss} and \ref{fig:bfsag_ss}). As a result, $\text{lr}=1$ and $n=10$ for ADAM and $\text{lr}=0.01$ and $n=10$ for BFSAG is used in the synthetic problems presented in Section \ref{sec:numerical}.}

\begin{figure} [htp]
\centering
\subfloat[ADAM: Tuning step sizes ($\text{lr}$)]{%
\resizebox*{7cm}{!}{\includegraphics{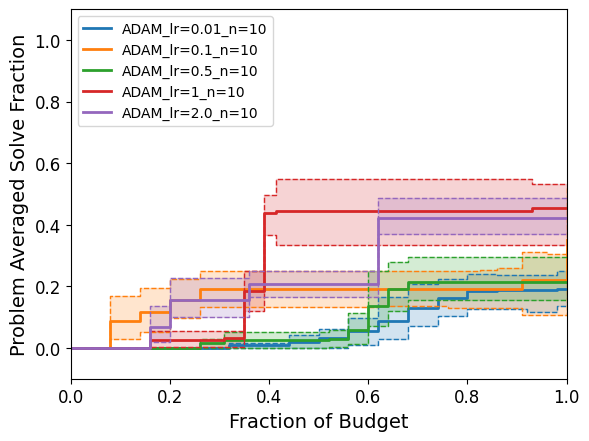}}\label{fig:adam_lr}}%\hspace{2pt}
\subfloat[ADAM: Tuning sample sizes ($n$)]{%
\resizebox*{7cm}{!}{\includegraphics{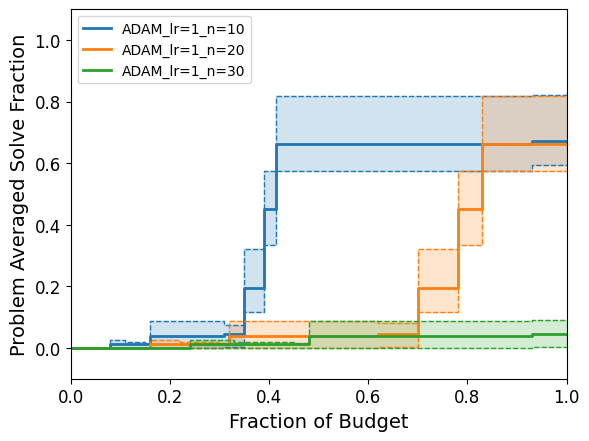}}\label{fig:adam_ss}}\\

\subfloat[BFSAG: Tuning step sizes ($\text{lr}$)]{%
\resizebox*{7cm}{!}{\includegraphics{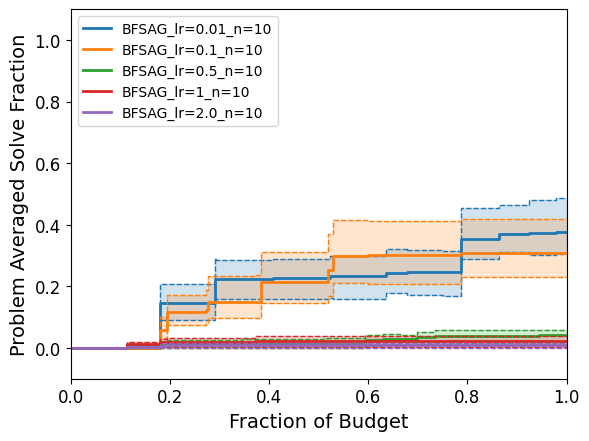}}\label{fig:bfsag_lr}}%\hspace{2pt}
\subfloat[BFSAG: Tuning sample sizes ($n$)]{%
\resizebox*{7cm}{!}{\includegraphics{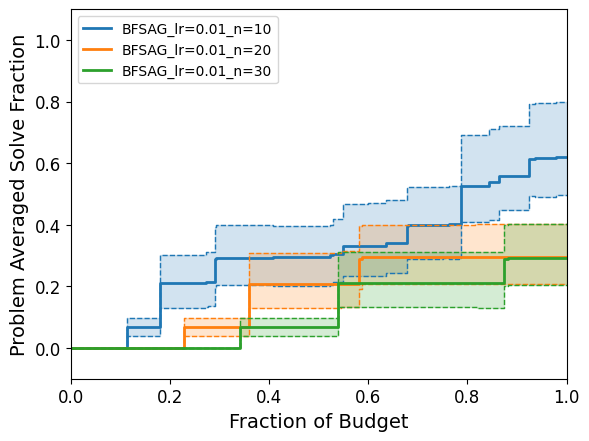}}\label{fig:bfsag_ss}}
\caption{{Solvability profiles for hyperparameter tuning of ADAM and BFSAG on synthetic problems. (a) and (c) show the impact of different step sizes with a fixed sample size $(n=10)$, while (b) and (d) show the impact of different sample sizes with the best step size found in (a) and (c).}}
\label{fig:ht-adam-bfsag}
\end{figure}

{For the DES-based problems, we again tuned the step sizes first with $n=10$, and then tuned the sample sizes. We used the same step size setup as ASTRO-BFDF to ensure a fair comparison, i.e., $\text{lr} = c_{\Delta} \sum_{i=1}^d (\max_r\{X^r_i\}-\min_r\{X^r_i\})$, where $X_i^r$ denoted the $i$-th coordinate of $\BFX^r$. We tested $c_{\Delta} \in \{0.5,1.0,1.5,2\}$ and $n\in\{10,20,30\}$ on DES-based problems as well. See Figure \ref{fig:des-astrobfdf}. As result, we use $c_{\Delta} = 2$ and $n=10$ for ADAM and BFSAG on the DES-based problems in Section \ref{sec:numerical}.}

\begin{figure} [htp]
\centering
\subfloat[ADAM: Tuning step sizes ($c_{\Delta}$)]{%
\resizebox*{7cm}{!}{\includegraphics{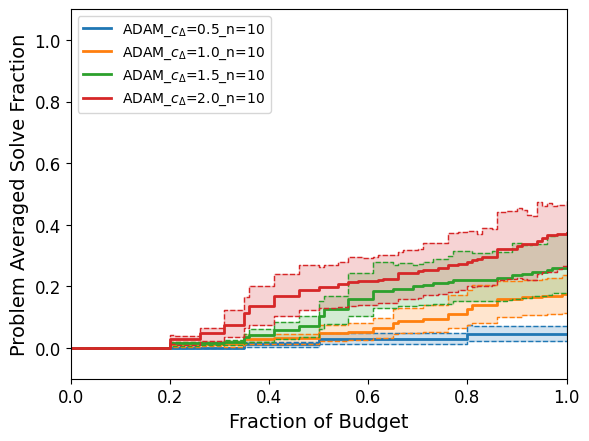}}\label{fig:adam_lr}}%\hspace{2pt}
\subfloat[ADAM: Tuning sample sizes ($n$)]{%
\resizebox*{7cm}{!}{\includegraphics{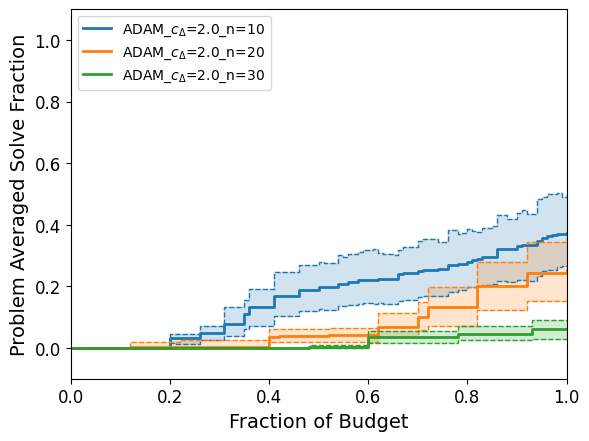}}\label{fig:adam_ss}}\\

\subfloat[BFSAG: Tuning step sizes ($c_{\Delta}$)]{%
\resizebox*{7cm}{!}{\includegraphics{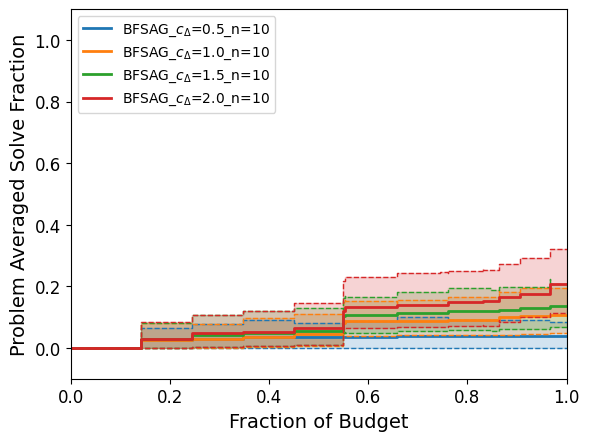}}\label{fig:bfsag_lr}}%\hspace{2pt}
\subfloat[BFSAG: Tuning sample sizes ($n$)]{%
\resizebox*{7cm}{!}{\includegraphics{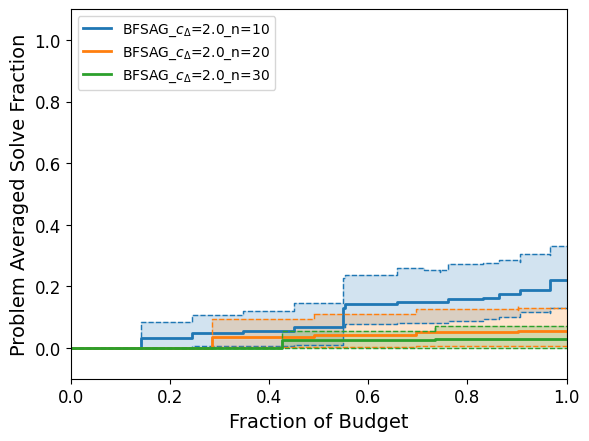}}\label{fig:bfsag_ss}}
\caption{{Solvability profiles for hyperparameter tuning of ADAM and BFSAG on problems with DES. (a) and (c) show the impact of different step sizes with a fixed sample size $(n=10)$, while (b) and (d) show the impact of different sample sizes with the best step size found in (a) and (c).}}
\label{fig:ht-adam-bfsag}
\end{figure}

\newpage
\subsection{{Setup for BFBO}}
{Since BFBO is a global optimizer, a step size is not required. Instead, the next candidate is selected by optimizing an acquisition function. Expected improvement combined with stochastic co-kriging is used as the acquisition function, following the widely adopted standard for bi-fidelity Bayesian optimization discussed in Section \ref{sec:intro}. The stochastic co-kriging model is implemented using SMT \cite{saves2024smt} in Section \ref{sec:numerical}. With the similar setup with BFSAG, the sample size for high-fidelity simulations is set to $n$, while the sample size for low-fidelity simulations is set to $2n$. In particular, when the new design points is selected by minimizing the expected improvement, the HF and LF function estimates are obtained by \eqref{eq:crude-mc} with $n$ and $2n$ replications, respectively. The co-kriging model is then retrained using all available function estimates. For hyper-parameter tuning, we tested $n \in \{10,20,30\}.$ See Figure \ref{fig:ht-bfbo}. While $n=10$ provides the best setup, generating a single sample path of $\{\BFX_k\}$ can take over an hour for the 20-dimensional stochastic Rosenbrock function with a budget of 4,000 high-fidelity evaluations, due to the need to retrain the co-kriging model each time a new design point is explored. For reference, generating a sample path using ASTRO-MFDF typically takes less than five minutes with the same setup.}

\begin{figure} [htp]
\centering
\subfloat[On synthetic problems]{%
\resizebox*{7cm}{!}{\includegraphics{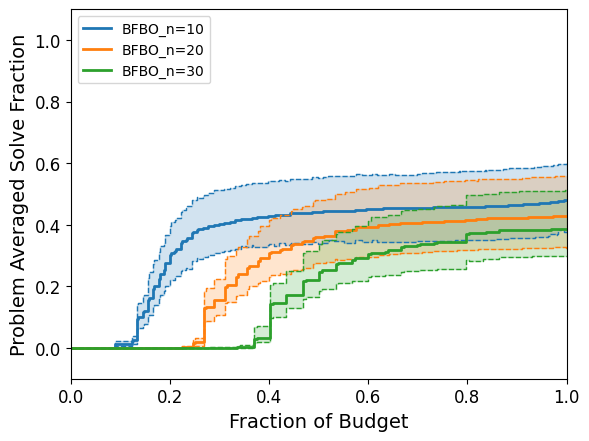}}\label{fig:syn_bfbo}}%\hspace{2pt}
\subfloat[On problems with DES]{%
\resizebox*{7cm}{!}{\includegraphics{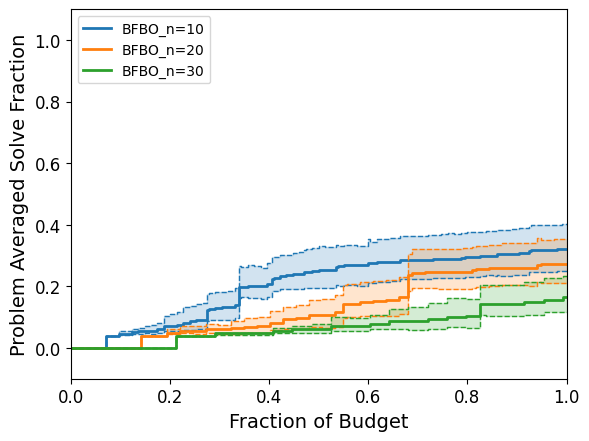}}\label{fig:des_bfbo}}
\caption{{Solvability profiles with 1\% optimality gap for hyperparameter tuning of BFBO.}}
\label{fig:ht-bfbo}
\end{figure}

\section{{Bi-fidelity Deterministic Functions}}
\label{apdx:function-des}
{This section provides details on the bi-fidelity deterministic functions used in Section \ref{sec:numerical}. Four deterministic HF functions are used in total: three benchmark problems—Branin, Colville, and Forretal functions—are taken from \cite{song2019radial}, while the fourth, the Rosenbrock function, is adapted for bi-fidelity optimization as introduced in \cite{mainini2022mfexample}. We begin by presenting the closed-form expressions of the functions and the corresponding problem dimensions (see Table \ref{tab:equations}). Since Branin, Colville, and Rosenbrock are multi-dimensional functions, all variables are fixed except $x[1]$, which is varied to visualize the loss landscape and provide intuition about the effect of $\kappa_{cor}$ (See Figure \ref{fig:loss-branin},\ref{fig:loss-colville}, and \ref{fig:loss-rosen}). The loss landscapes for the Forretal functions are shown in Figure \ref{fig:for-loss}. Overall, as $\kappa_{cor}$ increases, the LF function $f^{\ell}$ appears to provide more useful information for optimizing the HF function $f^h$.  }

\begin{figure} [htp]
\centering
\subfloat[$\kappa_{cor}=0.1$]{%
\resizebox*{7cm}{!}{\includegraphics{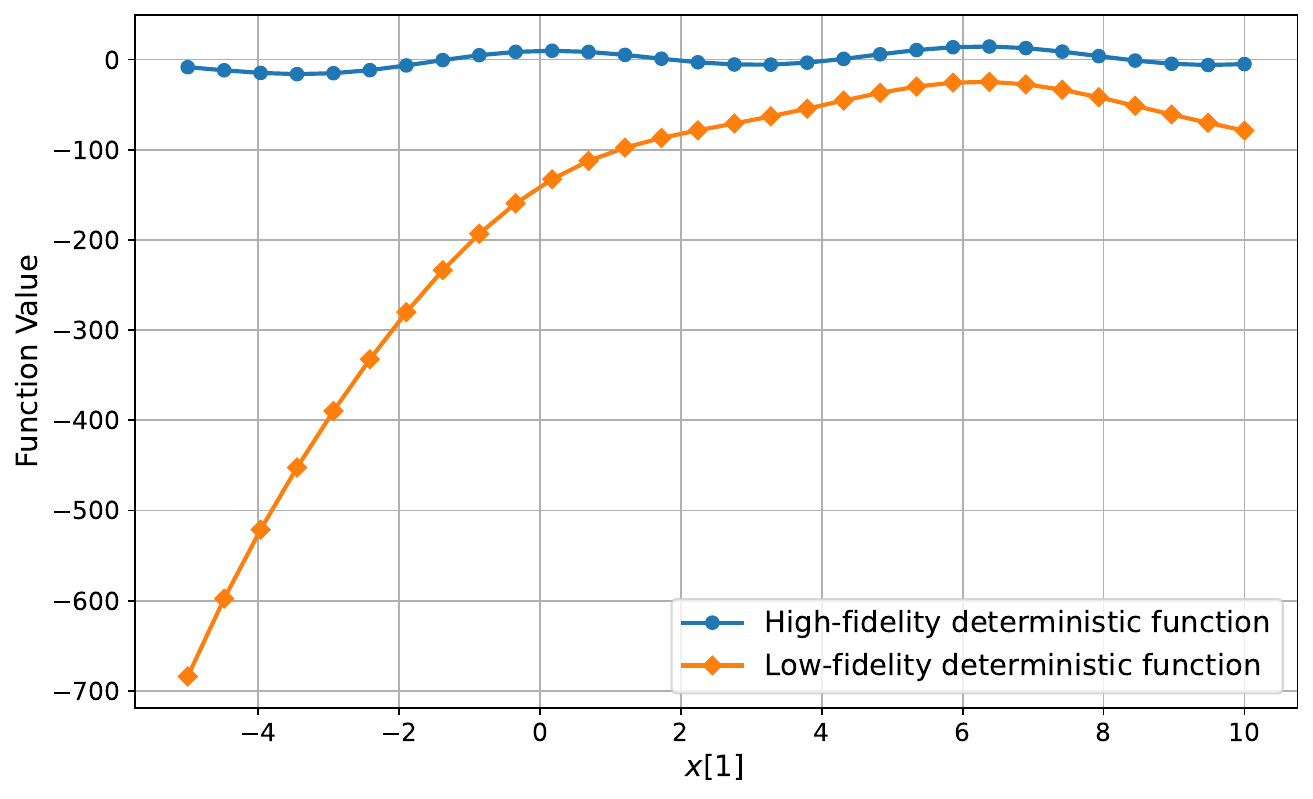}}\label{fig:loss-branin-0.1}}%\hspace{2pt}
\subfloat[$\kappa_{cor}=0.9$]{%
\resizebox*{7cm}{!}{\includegraphics{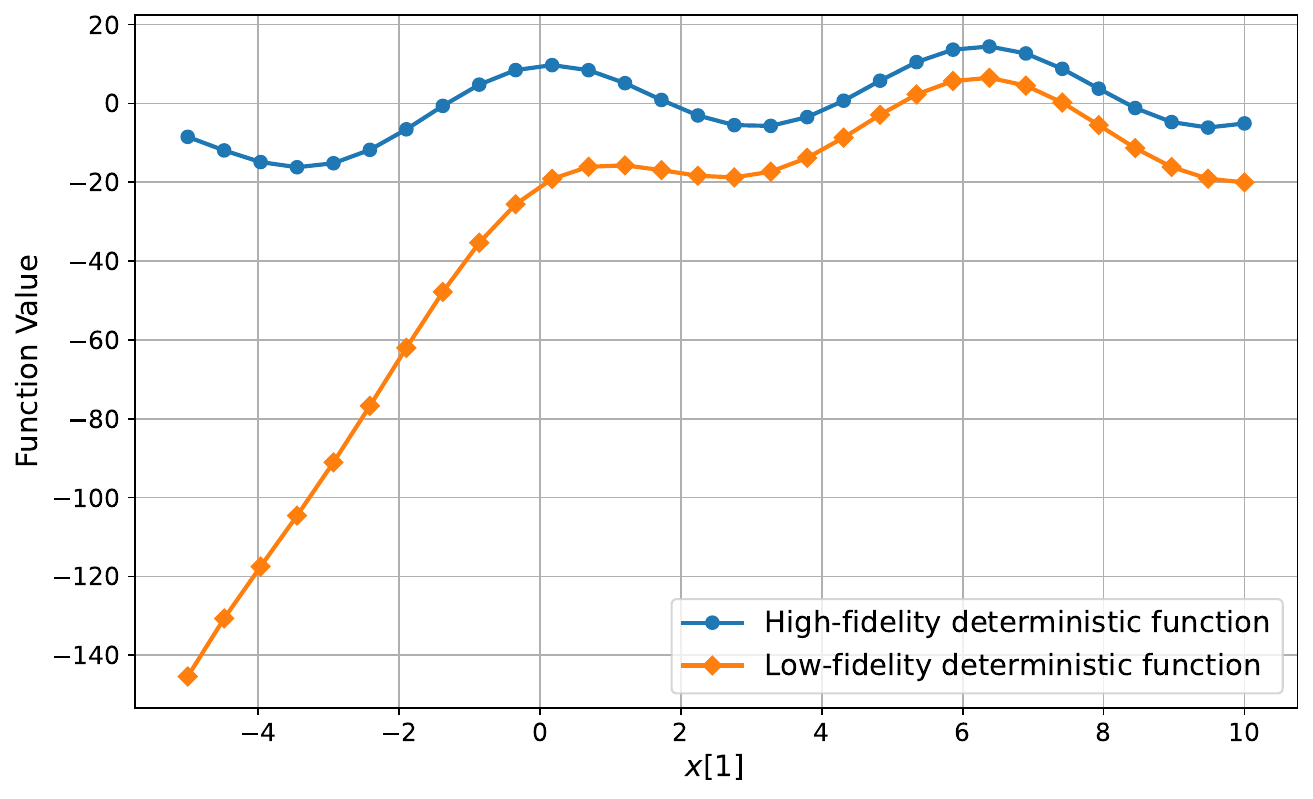}}\label{fig:loss-branin-0.9}}
\caption{{Marginalized loss landscapes of the Branin function with two different $\kappa_{cor}$ values, obtained by varying $x[1]$ while keeping all other elements fixed.}}
\label{fig:loss-branin}
\end{figure}

\begin{figure} [htp]
\centering
\subfloat[$\kappa_{cor}=0.1$]{%
\resizebox*{7cm}{!}{\includegraphics{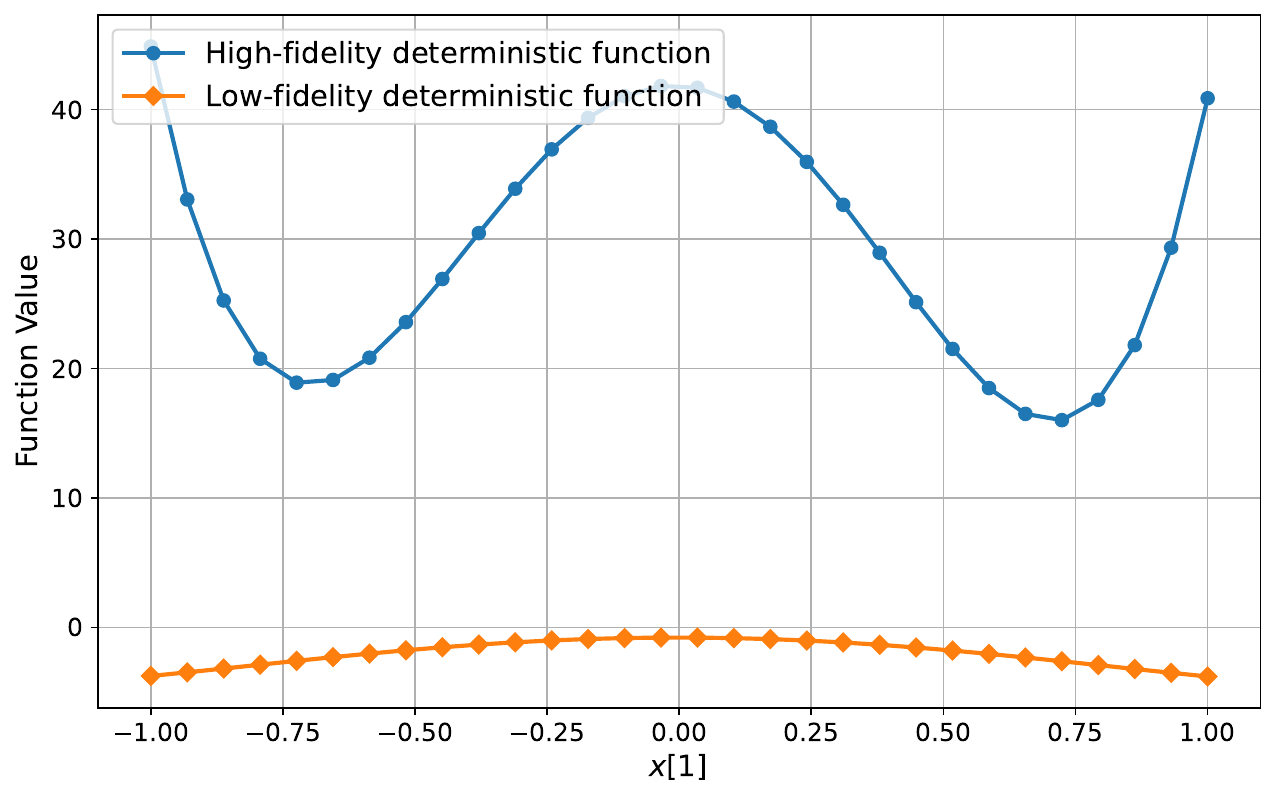}}\label{fig:loss-branin-0.1}}%\hspace{2pt}
\subfloat[$\kappa_{cor}=0.9$]{%
\resizebox*{7cm}{!}{\includegraphics{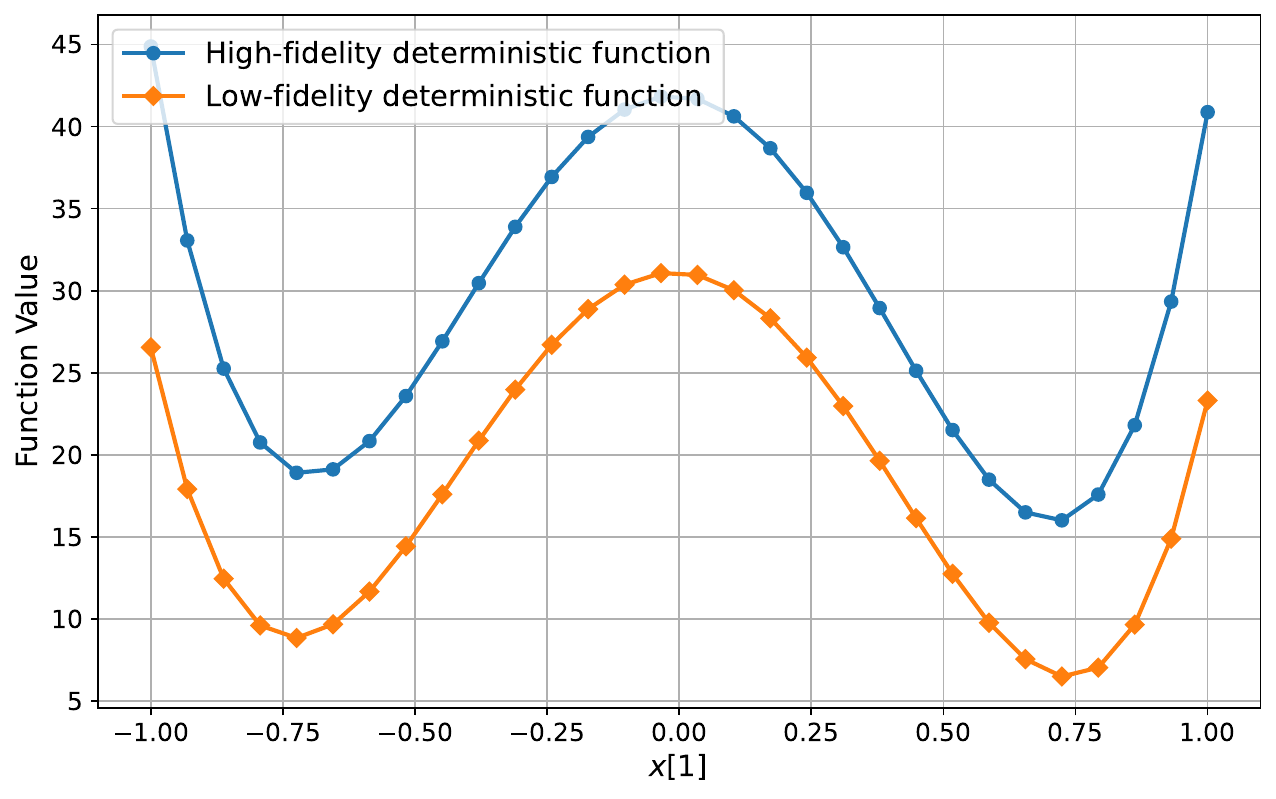}}\label{fig:loss-branin-0.9}}
\caption{{Marginalized loss landscapes of the Colville function with two different $\kappa_{cor}$ values, obtained by varying $x[1]$ while keeping all other elements fixed.}}
\label{fig:loss-colville}
\end{figure}

\begin{figure} [htp]
\centering
\subfloat[$\kappa_{cor}=0.1$]{%
\resizebox*{7cm}{!}{\includegraphics{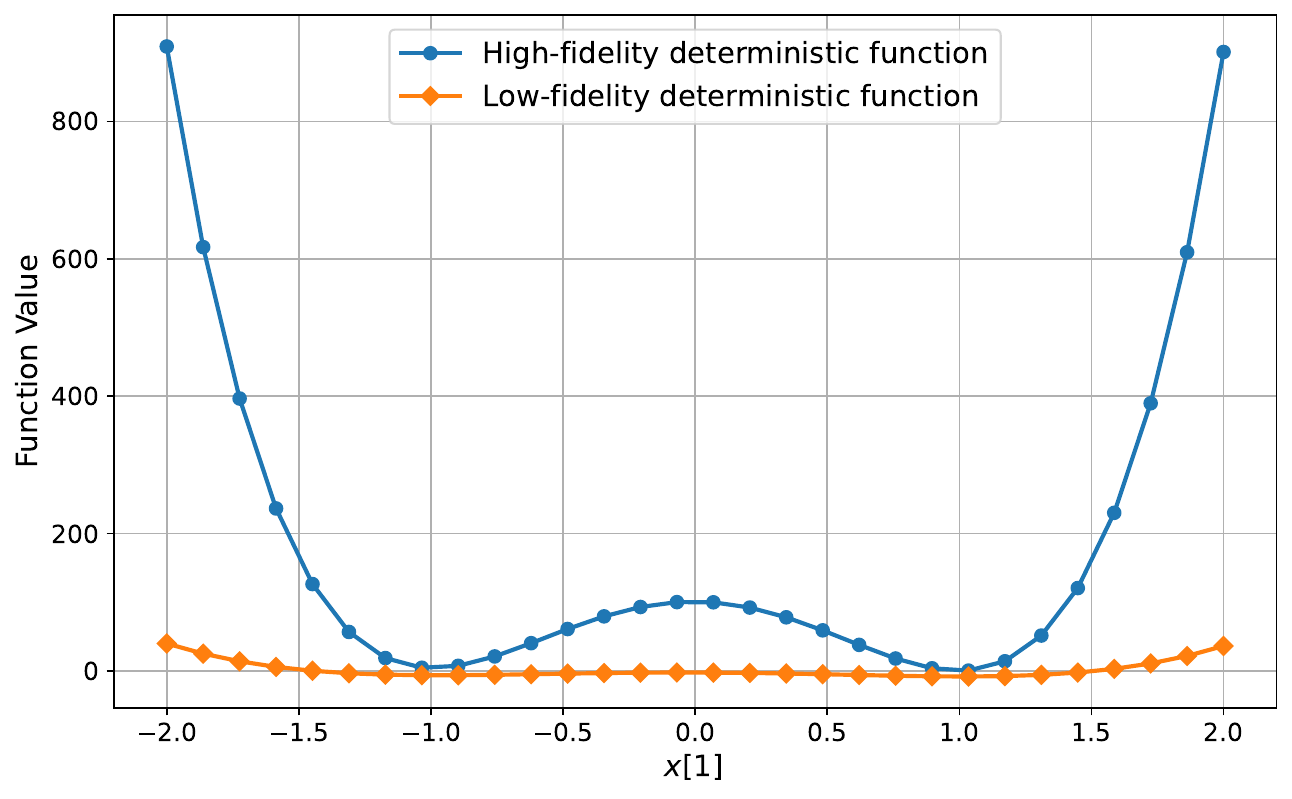}}\label{fig:loss-branin-0.1}}%\hspace{2pt}
\subfloat[$\kappa_{cor}=0.9$]{%
\resizebox*{7cm}{!}{\includegraphics{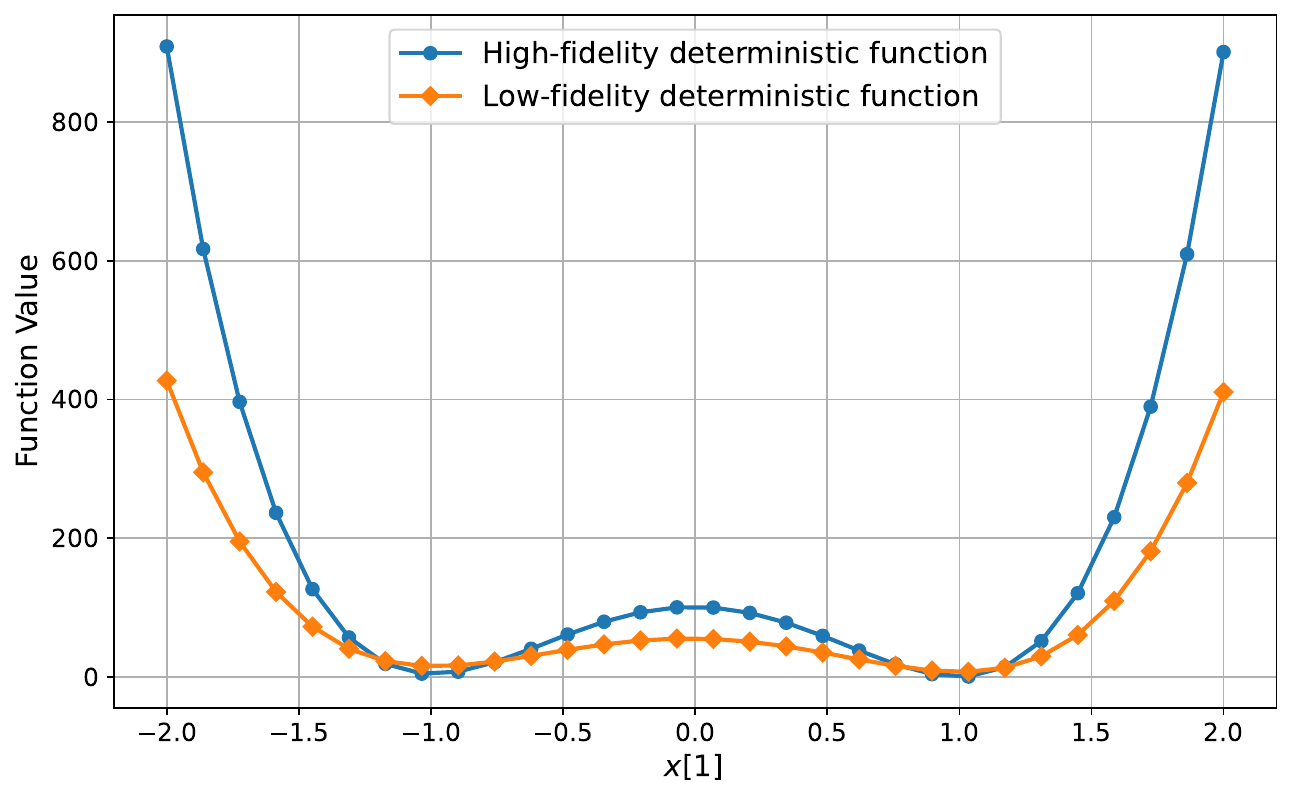}}\label{fig:loss-branin-0.9}}
\caption{{Marginalized loss landscapes of the Rosenbrock function with two different $\kappa_{cor}$ values, obtained by varying $x[1]$ while keeping all other elements fixed.}}
\label{fig:loss-rosen}
\end{figure}

\begin{table}[!htbp]
\vspace{0.1in}
\centering
\caption{Bi-fidelity deterministic functions}
\label{tab:equations}
\begin{tabular}{|l|l|c|}
\hline
\textbf{Function} & \textbf{Equation(s)} & \textbf{Dimension} \\ \hline\hline

\multirow{2}{*}{Forretal} 
  & $\cdot\ f^h(\BFx) = (6x - 2)^2 \sin(12x - 4)$ & \multirow{2}{*}{1} \\
  & $\cdot\ f^{\ell}(\BFx) = (-2 - \kappa_{cor}^2 +4\kappa_{cor}) f^h + 10(x - 0.5) - 5$ & \\ \hline

\multirow{4}{*}{Branin} 
  & $\cdot\ f^h(\BFx) = \left(x[2] - \frac{5.1}{4\pi^2}x[1]^2 + \frac{5}{\pi}x[1] - 6\right)^2$ & \multirow{4}{*}{2} \\
  & \qquad\qquad\qquad\qquad$+ 10\left(1 - \frac{1}{8\pi}\right)\cos{x[1]} + 10$ & \\
  & $\cdot\ f^{\ell}(\BFx) = f^h - \left(0.5 \kappa_{cor}^2 -2 \kappa_{cor} + 1.7\right)$ & \\
  & \qquad\qquad\qquad\qquad$\cdot \left(x[2] - \frac{5.1}{4\pi^2}x[1]^2 + \frac{5}{\pi}x[1] - 6\right)^2$ & \\ \hline

\multirow{5}{*}{Colville} 
  & $\cdot\ f^h(\BFx) = 100(x[1]^2 - x[2])^2 + (x[1] - 1)^2 + (x[3] - 1)^2$ & \multirow{6}{*}{4} \\
  & \qquad\qquad$+ 10.1\left((x[2] - 1)^2 + (x[4] - 1)^2\right)$ & \\
  & \qquad\qquad$+ 19.8(x[2] - 1)(x[4] - 1) + 90(x[3]^2 - x[4])^2$ & \\
  & $\cdot\ f^{\ell}(\BFx) = f^h\left(\kappa_{cor}^2\BFx\right)$ & \\
  & \qquad\qquad$- (\kappa_{cor}+0.5)\left(5x[1]^2 + 4x[2]^2 + 3x[3]^2 + x[4]^2\right)$ & \\ \hline

\multirow{3}{*}{Rosenbrock} 
  & $\cdot\ f^h(\mathbf{x}) = \sum_{i=1}^{D-1} \left[100(x[i+1] - x[i]^2)^2 + (1 - x[i])^2\right]$ & \multirow{3}{*}{20} \\
  & $\cdot\ f^{\ell}(\mathbf{x}) = \kappa_{cor}\sum_{i=1}^{D-1} \left[50(x[i+1] - x[i]^2)^2 + (-2 - x[i])^2\right]$ & \\
  & \qquad\qquad$-\sum_{i=1}^{D} 0.5x[i]$ & \\ \hline
\end{tabular}
\end{table}

\newpage
%\end{appendices}
\bibliography{main-paper}   % name your BibTeX data base

\end{document}